\documentclass[12pt, oneside]{amsart}   	
\usepackage[margin= 3cm]{geometry}                		
\usepackage{graphicx}				
\usepackage{mathrsfs}
\usepackage{amsthm}
\usepackage{amsmath}
\usepackage{hyperref}
\usepackage{tikz-cd}
\usepackage{dsfont}
\usepackage{tikz}
\usepackage{tikz-3dplot}
\usetikzlibrary{math}
\newcounter{thm}
\newtheorem{remark}[thm]{Remark}
\newtheorem{proposition}[thm]{Proposition}
\newtheorem{definition}[thm]{Definition}
\newtheorem{theorem}[thm]{Theorem}
\newtheorem{example}[thm]{Example}
\newtheorem{lemma}[thm]{Lemma}
\newtheorem{corollary}[thm]{Corollary}

\newcommand{\NN}{\mathbb{N}}
\newcommand{\ZZ}{\mathbb{Z}}
\newcommand{\kk}{\mathds{k}}

\newcommand{\fgt}{\operatorname{fgt}}
\newcommand{\val}{\operatorname{val}}
\newcommand{\vir}{\operatorname{vir}}
\newcommand{\rf}{\operatorname{ref}}
\newcommand{\orb}{\operatorname{orb}}
\newcommand{\vdim}{\operatorname{vdim}}

\newcommand{\pt}{\operatorname{pt}}

\newcommand{\ev}{\operatorname{ev}}
\newcommand{\Ev}{\operatorname{Ev}}

\newcommand{\Spec}{\operatorname{Spec}}
\newcommand{\gp}{\operatorname{gp}}

\newcommand{\trop}{\operatorname{trop}}
\newcommand{\out}{\operatorname{out}}
\newcommand{\Pic}{\operatorname{Pic}}

\newcommand{\st}{\operatorname{st}}
\newcommand{\im}{\operatorname{im}}
\newcommand{\forget}{\operatorname{forget}}
\newcommand{\proj}{\operatorname{Proj}}
\newcommand{\rel}{\operatorname{rel}}
\newcommand{\coker}{\operatorname{coker}}
\makeatletter
\@namedef{subjclassname@2020}{\textup{2020} Mathematics Subject Classification}
\makeatother

\usetikzlibrary{fadings}

\numberwithin{equation}{section}
\numberwithin{thm}{section}
\numberwithin{figure}{section}

\title[Log GW theory of log \'etale models]{Punctured log Gromov-Witten theory of log modifications and double ramification cycles with target log variety}
\author{Samuel Johnston}
\address{Samuel Johnston, Massachusetts Institute of Technology}
\email{smj553@mit.edu}
\date{}		

\begin{document}

\begin{abstract}
We expand upon a previous study conducted by the author in \cite{pbirinv} on the behavior of punctured log Gromov-Witten theory under log \'etale modifications $\widetilde{X} \rightarrow X$, giving expressions for log Gromov-Witten classes on $\widetilde{X}$ in terms of log Gromov-Witten classes on $X$, facilitating a complete reduction of the punctured log Gromov-Witten theory of any modification $\widetilde{X}$ of an snc log scheme $X$ to the punctured log Gromov-Witten theory of $X$. 
 
 We apply this result in two settings. First, we prove a log-orbifold correspondence equating the logarithmic invariants of the canonical wall structure of \cite{scatt} with a class of orbifold invariants considered in \cite{orbBL}, and more generally construct a family of algebra homomorphisms from the intrinsic mirror algebra $R_{(X,D)}$ of \cite{int_mirror} to appropriate power series rings generalizing the broken line expansion in \cite{scatt}. 
 
Second, we show the punctured log Gromov-Witten classes of split toric bundles are effectively reconstructed in terms of the punctured log Gromov-Witten classes of the base. Additional input for the second application is the introduction and study of double ramification cycles with target log variety, generalizing the double ramification cycles with target variety investigated in \cite{DRtarget}.
\end{abstract}

\maketitle

\section{Introduction}

A core aspect of log enumerative geometry is invariance under log modifications. A property first established by Abramovich and Wise in \cite{bir_GW} for ordinary log Gromov-Witten theory and extended by the author to the general setting of punctured invariants in \cite{pbirinv}, it implies that the log Gromov-Witten theory for a log smooth scheme $(X,D)$ is a sector of the log Gromov-Witten theory of any log modification $(\widetilde{X},\widetilde{D})$ of $(X,D)$. Moreover, many relations among log Gromov-Witten classes only manifest after passing to sufficiently modified targets, as highlighted in \cite{logroot,BNR2,logWDVV,degen}. In the present work, we show how to explicitly express the punctured log Gromov-Witten classes of any such modification $\widetilde{X}$ in terms of the punctured log Gromov-Witten classes of $X$:

For the setup, let $\widetilde{X} \rightarrow X$ be a log \'etale modification induced by a subdivision of the tropicalization $\widetilde{\Sigma(X)} \rightarrow \Sigma(X)$, and let $\pmb\gamma$ and $\pmb\tau$ be decorated realizable tropical types of log stable map to $\widetilde{X}$ and $X$ respectively such that stabilization induces a morphism $st: \mathscr{M}(\widetilde{X},\pmb\gamma) \rightarrow \mathscr{M}(X,\pmb\tau)$. We additionally let $\ev: \mathscr{M}(\widetilde{X},\pmb\gamma) \rightarrow \Ev:=\prod_{l \in L(G_{\gamma})} \widetilde{X}_{\pmb\sigma(l)}$ be the evaluation map of underlying stacks.


%
%
%
%
%

For $\sigma \in \widetilde{\Sigma(X)}$, we let $\text{Star}(\sigma)$ and $\widetilde{\Sigma(X)}_{\sigma}$ be the cone complexes with cones given by $\sigma'$ and $\sigma'/\sigma$ for $\sigma \subset \sigma' \in \widetilde{\Sigma(X)}$ respectively. For a tropical type $\tau'$ of punctured tropical map to $\Sigma(X)$ marked by $\tau$, we consider the tropical evaluation map of topological spaces 
$$\Sigma(\ev_{\tau'}): \tau' \rightarrow \prod_{l \in L(G_{\tau})} \widetilde{\Sigma(X)}$$
given by the product of the evaluation maps at vertices containing legs. This is induced by a piecewise linear map on some subdivision of $\widetilde{\tau}'$ of $\tau'$. After making an appropriate choice of log structure for $\Ev$, we have $\Sigma(\Ev) = \prod_{l \in L(G_{\tau})} \widetilde{\Sigma(X)}_{\pmb\sigma(l)}$. We now let $P\in PP(\Sigma(\Ev))$ be a piecewise polynomial which we pullback to a function on $\prod_{l \in L(G_{\tau})} \text{Star}(\pmb\sigma(l))$ along the quotient map. We let $\Sigma(\ev_{\tau'})^*(P)$ be the partially defined function on $\tau'$ given by pulling back $P$ along the partially defined map $\tau' \rightarrow \prod_{l \in L(G_{\tau})} \text{Star}(\pmb\sigma(l))$, yielding a piecewise polynomial function on some subcomplex of a subdivision of $\tau'$. 

 The piecewise polynomial $P$ also induces a cohomology class $A^{\deg\text{ }P}(\prod_{l \in L(G_{\tau})}\widetilde{X}_{\pmb\sigma(l)})$, which we will also refer to as $P$ with context disambiguating notation. The following theorem expresses the pushforward of any such class to $\mathscr{M}(X,\pmb\tau)$ purely in terms of the tropicalization of the moduli space $\mathscr{M}(X,\pmb\tau)$:



\begin{theorem}\label{mthm1}
With notation as above, there exists a piecewise polynomial $Q_{\gamma}$ on a subdivision of $\tau$, and extensions to piecewise polynomials on subdivisions of $\tau'$ for every tropical type marked by $\tau$ such that for any piecewise polynomial function $P \in PP(\Sigma(\Ev))$, $Q_{\gamma}\Sigma(\ev_{\tau'})^*(P)$ is a globally defined piecewise polynomial function on a subdivision of $\tau'$, and the following equation in $A_*(\mathscr{M}(X,\pmb\tau))$ holds:

\begin{equation}\label{maineq1}
st_*(\ev^*(P)\cap[\mathscr{M}(\widetilde{X},\pmb\gamma)]^{\vir}) =  \sum_{\pmb\tau \subset \pmb\tau'} \frac{\deg_{\tau'}(Q_{\gamma}\Sigma(\ev_{\tau'})^*(P))}{|Aut(\pmb\tau'/\pmb\tau)|}[\mathscr{M}(X,\pmb\tau')]^{\vir}
\end{equation}
In the above equation, $\deg_{\tau'}(Q_{\gamma}\Sigma(\ev_{\tau'})^*(P))$ is a tautological operational Chow class in the sense of \cite{tauttarget} associated with a polynomial on $\tau'$ in edge lengths and PL functions on $\Sigma(X)$ pulled back along evaluation maps associated with sections of the universal tropical curve, computed from a $\mathbb{G}_m^{\dim \tau'}$ equivariant integral on any proper toric variety in an explicit collection cofinal under toric blowups.

\end{theorem}

\begin{corollary}[Corollary \ref{reduce}]\label{mcr1}
For $(X,D)$ a simple normal crossings pair, the primary punctured log Gromov-Witten classes of any toroidal modification $\widetilde{X}$ are expressed explicitly in terms of the punctured log Gromov-Witten classes of $X$.
\end{corollary}

The two sides of Equation \ref{maineq1} feature two ways of imposing constraints in log Gromov-Witten theory. The left hand side features the standard way of imposing constraints in Gromov-Witten theory; we pullback cohomology classes of the target along an evaluation map and cap with the virtual class. However, since we are working on a log modification of $\widetilde{X}$ of $X$, the space of insertions is larger. On the right hand side of Equation \ref{maineq1}, we see the appearance of ``tropical" constraints, given by picking a realizable tropical type which is marked by the starting tropical type. This has the effect of picking out a stratum in the moduli space of additionally degenerate curves, and in particular does not refer to a blowup of the target. This method of fixing constraints features prominently in applications of log Gromov-Witten theory to the construction of canonical wall structures for mirrors to log Calabi-Yau manifolds, appearing in \cite{scatt}, as well as in the Quiver DT/log GW correspondence investigated by Arg\"uz and Bousseau in \cite{QDTlog}. Theorem \ref{mthm1} demonstrates how these two different types of constraints are related. 


Theorem \ref{mthm1} can be viewed as a localization scheme for the evaluation of punctured log Gromov-Witten invariants of $\widetilde{X}$ in terms of punctured log Gromov-Witten invariants of $X$. The main obstacle to producing a more explicit description of $\deg_{\tau'}(P)$ are the toric singularities encoded by the tropical moduli spaces $\tau'$. Without further assumptions concerning the geometry of the target, there is no restriction on the types of cones which may appear, see \cite{tropvak}. 
 
\begin{remark}
\begin{enumerate}
\item The operational Chow classes on the righthand side of Equation \ref{maineq1} depend on a choice of extension of a polynomial on a cone $\tau$ to a piecewise polynomial on the tropical moduli space of tropical maps marked by $\tau$. Different choices for extending to the full tropical moduli space lead to different fans involved in the computation of higher codimension contributions, which give boundary corrections for the difference in lower codimension.
 \item The descendent theory of $\widetilde{X}$ can also be expressed in terms of the descendent theory of $X$ via the usual boundary corrections for comparing different $\psi$-classes under stabilization maps. We leave the details of this extension to the interested reader.
\end{enumerate}
\end{remark}

%

\begin{example}
An instructive example of the above relation to consider is when $(X,D)$ is a simple normal crossings toric pair. For concreteness, we consider $(\mathbb{P}^2,D)$, with $D = D_1 + D_2+D_3$ the toric boundary. We consider the moduli spaces of genus $0$ degree $4$ curves intersecting the boundary at $6$ points with contact orders $(1,1,0)$, $(1,0,1)$ and $(0,1,1)$. Let $x_1$ be a marked point with contact order $(1,0,1)$, $x_2,x_3$ be marked points with contact order $(1,0,1)$, and $x_4,x_5$ be marked points with contact order $(0,1,1)$. We also let $l_i \in L(G_{\tau})$ be the leg associated with $x_i$ for any tropical type $\tau$ marked by $\beta$. After blowing up each zero stratum, giving $\widetilde{X}$, we may lift the contact orders so that these contact orders are transverse with respect to the toric boundary, and in particular $\Ev =( \mathbb{P}^1)^5$. We consider the log Gromov-Witten invariant:
\begin{equation}\label{exinv}
\int_{[\mathscr{M}(\widetilde{X},\beta)]^{\vir}} \prod_{i=1}^5\ev_{x_i}^*([\pt])
\end{equation}

Since this is a genus $0$ log Gromov-Witten invariant of a toric variety of low degree, the integral is easily calculated using the tropical correspondence theorem of \cite{NStrop} to be $1$. We will consider how Equation \ref{maineq1} yields this equation. We note that the point classes for the strata $\widetilde{X}_{\pmb\sigma(l)} \cong \mathbb{P}^1$ each come from primitive non-negative piecewise linear functions on $\Sigma(\widetilde{X})_{\pmb\sigma(l)} \cong \Sigma_{\mathbb{P}^1}$, specifying a maximal cone of $\Sigma(\widetilde{X})$ for every leg $l$. Letting $P = \prod_i^5 P_i$ be the product of all of these PL functions pulled back along the evaluation map $\Sigma(\mathscr{M}(X,\pmb\tau)) \rightarrow \Sigma(\Ev)$, for $\pmb\tau$ a tropical type marked by $\beta$, we have $P|_{\tau} \not= 0$ if and only there is a tropical curve of type $\tau$ with leg $l_i$ mapping into the interior of a relevant maximal cone of $\Sigma(\widetilde{X})$. Additionally, note that for any PL function $P_i$ on $\Sigma(\mathscr{M}(\widetilde{X},\pmb\gamma))$ pulled back along an evaluation map $\ev_{l_i}:\Sigma(\mathscr{M}(\widetilde{X},\pmb\gamma)) \rightarrow \Sigma_{\mathbb{P}^1}$ factoring through a single cone for which $P_i|_{\gamma} \not= 0$, then $P_i\cap [\mathscr{M}(\widetilde{X},\pmb\gamma)]^{\vir} = 0$. 

We first consider the types contributing to the expression for $\ev_{x_1}^*P_1\cap [\mathscr{M}(\widetilde{X},\beta)]^{\vir}$ given by Theorem \ref{mthm1}. One such contribution comes from the zero stratum with tropical type $\pmb\gamma$ whose tropicalization is depicted in Figure \ref{fig1}. For any other tropical type $\gamma'$ contributing to the expression for $\ev_{x_1}^*P_1\cap[\mathscr{M}(\widetilde{X},\beta)]^{\vir}$, there exists a leg $l_i$ such that the cone $\gamma'_{l_i}$ associated with the leg $l_2$ maps into a top dimensional cone of $\Sigma(\widetilde{X})$, and the PL function $P_i$ pulls back from $\Sigma(X)$, implying $\prod_{i=2}^5P_i [\mathscr{M}(\widetilde{X},\pmb\gamma')]^{\vir} = 0$. Thus, it suffices to to calculate $\st_*\prod_{i=2}^5P_i[\mathscr{M}(\widetilde{X},\pmb\gamma)]^{\vir}$ for $\st: \mathscr{M}(\widetilde{X},\pmb\gamma) \rightarrow \mathscr{M}(X,\pmb\tau)$ the stabilization map and $\pmb\tau$ the decorated tropical type determined by $\pmb\gamma$.

 To use Theorem \ref{mthm1} to evaluate $\pi_*\prod_{i=2}^5[\mathscr{M}(\widetilde{X},\pmb\gamma)]^{\vir}$, we observe that the only dimension $5$ tropical types marked by $\gamma$ for which $\Sigma(\ev)^*(\prod_{i=2}^5 P_i) \not= 0$ is the tropical type $\tau'$ depicted in Figure \ref{fig2} and $Q_{\gamma}\prod_{i=2}^5 P_i$ vanishes on all faces of $\tau'$. It follows that:
\begin{equation}
\begin{split}
\st_*P\cap [\mathscr{M}(\widetilde{X},\pmb\gamma)]^{\vir} &= \deg_{\tau'}(Q_{\gamma'}\prod_{i=2}^5 P_i)[\mathscr{M}(X,\pmb\tau')]^{\vir} \\
&= |\coker(Q_{\gamma'}\times \prod_{i=2}^5 P_i: \tau^{'\gp}_{\NN} \rightarrow \mathbb{Z}^5)|[\mathscr{M}(X,\pmb\tau')]^{\vir} \\
&= [\mathscr{M}(X,\pmb\tau')]^{\vir}.
\end{split}
\end{equation}
The second equality follows from the fact that $\deg_{\tau'}(Q_{\gamma}\prod_{i=2}^5P_i)$ is simply the corresponding equivariant integral on the affine toric variety $\tau'$ when $Q_{\gamma}\prod_{i=2}^5P_i$ vanishes on all faces of $\tau'$, which is easily calculated as the index of a lattice map when $\tau'$ is simplicial. In particular, we recover the fact that the invariant of interest is $\deg[\mathscr{M}(X,\pmb\tau')]^{\vir} = 1$.

\begin{figure}
\centering
\begin{tikzpicture}
\fill[white!85!blue] (-6,0)--(-3,0)--(-3,3)--(-6,3)--cycle;
\fill[white!85!blue] (-6,0)--(-3,0)--(-3,-3)--(-9,-3)--cycle;
\fill[white!85!blue] (-6,0)--(-9,-3)--(-9,3)--(-6,3)--cycle;
\draw[black] (-6,0)--(-6,3);
\draw[black] (-6,0)--(-9,0);
\node at (-6,3.5) {$D_3$};
\draw[black] (-6,0)--(-3,0);
\draw[black] (-6,0)--(-3,3);
\node at (-2.5,0) {$D_2$};
\draw[black] (-6,0)--(-9,-3);
\draw[black] (-6,0)--(-6,-3);
\node at (-9.5,-1){$l_1$};
\node at (-9.5,-3.5) {$D_1$};
\draw[ball color=red] (-6,-1) circle (0.5mm);
\draw[ball color=red] (-5.5,0) circle (0.5mm);
\draw[ball color=red] (-5,1) circle (0.5mm);
\draw[->,color = red] (-6,-1)--(-6,-3);
\draw[->,color = red] (-5.95,-1)--(-5.95,-3);
\draw[->,color = red] (-6,-1)--(-9,-1);
\draw[-,color = red] (-6,-1)--(-5,1);
\draw[->,color = red] (-5,1)--(-3,3);
\draw[->,color=red] (-5,1.05)--(-3.05,3);
\draw[->,color=red](-5,1)--(-9,1);
\end{tikzpicture}
\caption{The tropical type $\gamma$
\label{fig1}}
\end{figure}
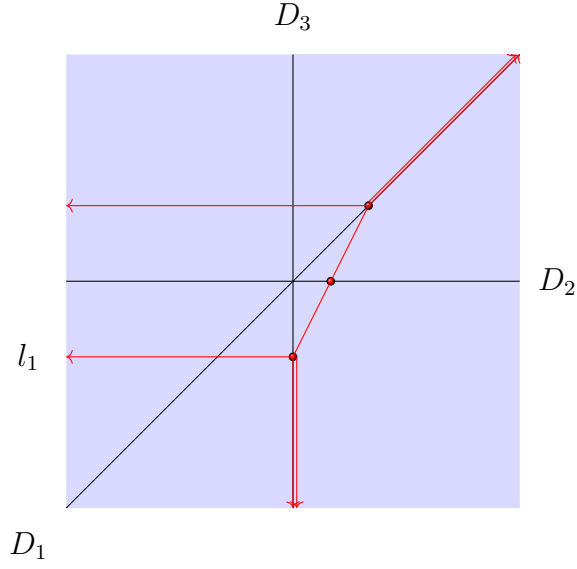


\begin{figure}[h]
\centering
\begin{tikzpicture}
\fill[white!85!blue] (-6,0)--(-3,0)--(-3,3)--(-6,3)--cycle;
\fill[white!85!blue] (-6,0)--(-3,0)--(-3,-3)--(-9,-3)--cycle;
\fill[white!85!blue] (-6,0)--(-9,-3)--(-9,3)--(-6,3)--cycle;
\draw[black] (-6,0)--(-6,3);
\draw[dashed] (-6,0)--(-9,0);
\node at (-6.5, -3.5) {$l_2$};
\node at (-5.5,-3.5) {$l_3$};
\node at (-6,3.5) {$D_3$};
\draw[black] (-6,0)--(-3,0);
\draw[dashed] (-6,0)--(-3,3);
\node at (-2.5,0) {$D_2$};
\node at (-2.75, 2.75) {$l_4$};
\node at (-3.5,3.25) {$l_5$};
\draw[black] (-6,0)--(-9,-3);
\draw[dashed] (-6,0)--(-6,-3);
\node at (-9.5,-1){$l_1$};
\node at (-9.5,-3.5) {$D_1$};
\draw[ball color=red] (-6.5,-1) circle (0.5mm);
\draw[ball color=red] (-5.75,-.25) circle (0.5mm);
\draw[ball color=red] (-5.625,0) circle (0.5mm);
\draw[ball color=red] (-5.5,.25) circle (0.5mm);
\draw[ball color= red] (-5.5,1) circle (0.5mm);
\draw[->,color = red] (-6.5,-1)--(-6.5,-3);
\draw[->,color = red] (-6.5,-1)--(-9,-1);
\draw[-,color = red] (-6.5,-1)--(-6,-.5);
\draw[-,color = red] (-6,-.5)--(-5.75,-.25);
\draw[-,color = red] (-5.75,-.25)--(-5.625,0);
\draw[->,color = red] (-5.75,-.25)--(-5.75,-3);
\draw[-,color = red] (-5.625,0)--(-5.5,.25);
\draw[->,color = red] (-5.5,.25)--(-3,2.75);
\draw[-,color = red] (-5.5,.25)--(-5.5,.5);
\draw[-,color = red] (-5.5,.5)--(-5.5,1);
\draw[->,color=red] (-5.5,1)--(-3.5,3);
\draw[->,color=red](-5.5,1)--(-9,1);

\end{tikzpicture}
\caption{The tropical type $\tau'$
\label{fig2}}
\end{figure}
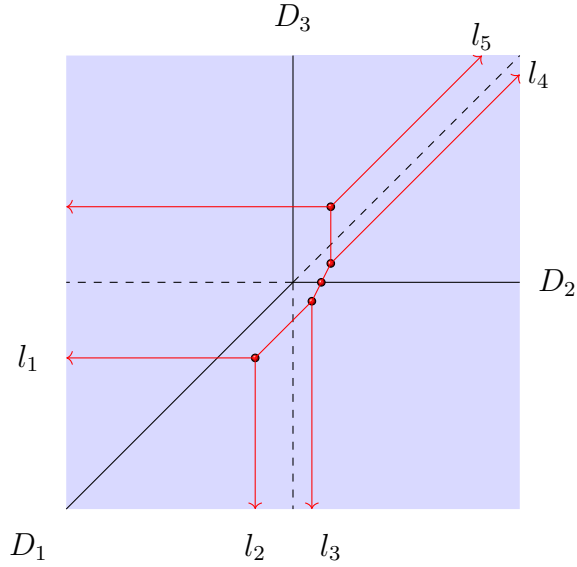

\end{example}

We provide two primary applications of Equation \ref{maineq1}. 

\subsection{Log-orbifold correspondence for canonical wall structures}

In the setting of mirror symmetry and canonical wall structure for log Calabi-Yau pairs, Gross and Siebert in \cite{scatt} use realizable tropical types referred to as ``wall types" and ``broken line types" to define the necessary enumerative input to construct the canonical wall structures for a log Calabi-Yau pair. Separately, You in \cite{orbBL} constructs analogues of these broken line invariants using orbifold Gromov-Witten theory, later using these invariants in \cite{yougamma} to study the Gamma conjecture for Fano varieties. We use Theorem \ref{mthm1} to compare these invariants. Namely, we use Theorem \ref{mthm1} to provide an alternate description the broken line/wall type invariants as integrals on the space of log maps to a log \'etale modification, from which the log-orbifold correspondence of \cite{BNR2} will allow us to deduce the following corollary:

\begin{corollary}[Proposition \ref{blorb}, Proposition \ref{worb}]
For a log Calabi-Yau pair $(X,D)$ in the sense of \cite{scatt} with essential skeleton $B \subset \Sigma(X)$, $p \in B$ an integral point, and $x \in \sigma_x \subset B$ a general point of a cone $\sigma_x$, the broken-line expansions $\vartheta_{p,x}$ is equal to the orbifold theta-function $\vartheta_{p}(x)$ of \cite{orbBL} of a suitable log \'etale model $\widetilde{X}$ of $X$. Additionally, the canonical wall structure can be expressed using the orbifold Gromov-Witten theory of log \'etale modifications of $X$. 
\end{corollary}

The pairs $(X,D)$ considered in \cite{scatt} satisfy the condition that $D$ contains a zero dimensional stratum. However, the definitions for broken lines expansions given as a result of the above corollary extends to general log Calabi-Yau pairs without the existence a zero dimensional stratum. We show that this assignment extends to a morphism of algebras:

\begin{theorem}[Theorem \ref{blhom}]
For $(X,D)$ a log Calabi-Yau pair in the sense of \cite{int_mirror} and $x \in \sigma_x \subset B(X)$ general, and $\kk[Q] = \kk[NE(X)]/I$ for $I$ a monomial ideal of $\kk[NE(X)]$ such that only finitely many curve classes are not contained in $I$, the assignment sending a theta function $\vartheta_p$ to its associated broken line expansion $\vartheta_{p,x}$ extends to a morphism of $\kk[Q]$ algebras $R_{(X,D)} \rightarrow \kk[\sigma_{x,\NN}^{\gp}][Q]$, with the constant term of a product of broken line expansions given by a descendent log Gromov-Witten invariant.
\end{theorem} 
As in \cite{logWDVV}, the necessary relations required for the result are naturally expressed in terms of invariants on log \'etale modifications of $X$, but yield interesting relations among punctured log Gromov-Witten classes on $X$ by Theorem \ref{mthm1}, generalizing \cite[Theorem $6.1$]{scatt} to the general log Calabi-Yau setting. 

When $X\setminus D$ is an affine log Calabi-Yau which admits a compactification to a Fano variety $X'$, the descendent log Gromov-Witten invariants appearing as the constant term in products of broken line expansions are related to the quantum periods of $X'$ by \cite[Theorem $1.1$]{qpint}. As a result, we produce power series which are weak mirrors to Fano compactifications of $X\setminus D$:

\begin{corollary}\label{fanocor}
For a smooth Fano variety $X$ and $D = D_1+\cdots+D_m \in |-K_X|$ satisfying the conditions of \cite[Theorem $1.1$]{qpint} with $X\setminus D$ having essential skeleton $B$, for any $x \in B$ general with integral tangent space $\Lambda_x$, there is a power series $W_{D,x} \in \kk[\Lambda_x][[NE(X)]]$ which is a weak Landau-Ginzburg mirror to $X$. 
\end{corollary}

\subsection{Log Gromov-Witten theory of split toric bundles}

A second application is in the study of the log Gromov-Witten theory of split toric bundles. Such spaces arise as log \'etale modifications $X \rightarrow S$, with $S$ having generic non-trivial log structure. Theorem \ref{mthm1} can be used to express the punctured log Gromov-Witten theory of $X$ in terms of the punctured log Gromov-Witten theory of $S$.

\begin{theorem}\label{mthm2}
Given a smooth projective variety $S$ equipped with an snc divisor and any split toric bundle $X \rightarrow S$, the punctured log Gromov-Witten classes of $X$ can be effectively reconstructed in terms of the punctured log Gromov-Witten classes of $S$. In particular, the punctured log Gromov-Witten classes of $X$ push forward to tautological classes in the Chow ring of moduli space of stable curves if the punctured log Gromov-Witten classes of $S$ push forward to tautological classes. 
\end{theorem}

The special case when $S$ has trivial log structure states that the punctured log Gromov-Witten theory of a toric bundle over $S$ can be reconstructed in terms of the absolute Gromov-Witten theory of $S$. When $S = \Spec \kk$, this recovers the results of Molcho and Ranganathan in \cite{logint} and Ranganthan and Urundolil Kumaran in \cite{DRtoric}. The authors of \cite{DRtoric} relate the log Gromov-Witten classes of a toric variety to explicit expressions involving piecewise polynomial classes and the toric contact cycle on the moduli space of stable curves. In our setting, we will require a generalization of the toric contact cycle. Such a generalization is naturally expressed using punctured log Gromov-Witten theory, and we refer to this class as the \emph{toric contact cycle with target log variety}. The following theorem shows that these punctured log Gromov-Witten classes are explicitly expressed in terms of the log Gromov-Witten classes of $S$:

\begin{theorem}[Theorem \ref{hrankt}]\label{mthm3}
For a log smooth projective variety $S$, $\pmb\tau$ a decorated realizable tropical type with a single vertex $v \in V(G_{\tau})$ and $k$ legs with total curve class $\textbf{A} \in H_2(S)$, line bundles $L_1,\ldots,L_k \in Pic(S)$, and vectors $a_i \in \ZZ^{k}$ with $\sum_i a_{ij} = L_j\cdot \textbf{A}$ for all $L_j$, let $[DR_{(L_i,a_i),\textbf{A}}(S,\pmb\tau)] \in A_*(\mathscr{M}(S,\pmb\tau))$ be the associated toric contact cycle with target log variety defined in Equation \ref{drlogt}. Then $[DR_{(L_i,a_i),\textbf{A}}(S,\pmb\tau)]$ is a rational linear combination of log Gromov-Witten classes of $\mathscr{M}(S,\pmb\tau)$. 
\end{theorem}


The proof of Theorem \ref{mthm3} follows by an induction on the number of line bundles $k$. Crucial to carrying out the induction is that the base $S$ is allowed to have log structure and $\pmb\tau$ is allowed to have non-trivial tropical moduli. When the log structure is trivial and $k=1$, we will show that this cycle recovers the double ramification cycle with target, which has an explicit formula given in \cite{DRtarget}. More generally, by appropriately adapting arguments of Molcho from \cite{BNtaut}, we produce an expression for the double ramification cycle with log target which is sufficient to prove Theorem \ref{mthm3} by induction on $k$.



\subsection{Future direction}

The strategy used to produce Equation \ref{maineq1} work more broadly for providing formulae for the pushforward of piecewise polynomial classes on log \'etale modifications. Many structures and tools available in Gromov-Witten theory (gluing, localization, Givental formalism) are only expected to become available after a sufficient log \'etale modification of the moduli spaces in log Gromov-Witten theory. The tools for producing Equation \ref{maineq1} should offer a pathway of transporting such structures and tools to the study of intersection theory on a fixed log \'etale model of the moduli stack of log stable maps.

Finally, Maulik and Ranganathan in the recent preprint \cite{MRrecon} use degenerations and the expanded formalism for log Gromov-Witten theory to prove a reconstruction result for the log Gromov-Witten theory of simple normal crossings pairs $(X,D)$ in terms of the absolute Gromov-Witten theory of all strata. One key ingredient to extending this result to the setting of punctured log maps is understanding spaces of maps into targets with generically non-trivial log structure. An important development in this direction is Theorem \ref{mthm3}. The author expects these two works, together with slight improvements to the gluing schemes for punctured log curves, should yield a similar reconstruction result for the entire punctured log Gromov-Witten theory for a simple normal crossings pair, providing an affirmative answer to a question posed by Maulik and Ranganathan in \cite{MRrecon}. 

\subsection{Acknowledgement}
I would like to thank Robert Crumplin, Mark Gross, Ajith Kumaran, Xuanchun Lu, Davesh Maulik, Dhruv Ranganathan and Peter Zaika for many helpful discussions related to this work. I would especially like to thank Dhruv Ranganathan who suggested the approach taken to establish the $k=1$ case of Theorem \ref{mthm3}. The author is supported by NSF DMS-2502976.

\section{Logarithmic and tropical preliminaries}

Throughout this paper, given a fan $\Sigma$ in $\mathbb{R}^n$, we refer to the associated toric variety as $Z_{\Sigma}$. If $\sigma$ is a cone and $\sigma'\subset \sigma$ is a lattice coarsening, there is an associated root stack of $Z_\sigma$ associated with $\sigma'$, which we denote by $Z_{\sigma'}$. We recall the definitions and constructions surrounding Artin fans which will be invoked in this paper:

\begin{definition}
For $\sigma$ a rational polyhedral cone, the Artin cone  $\mathcal{A}_{\sigma} = [Z_{\sigma}/T_{\sigma}]$ is the global quotient of the toric variety $Z_{\sigma}$ by its dense algebraic torus. An Artin fan is a logarithmic algebraic stack with a strict \'etale cover by Artin cones.

For a face $\omega \subset \sigma$, let $Z_{\sigma,\omega}$ be the closure of the torus stratum in $Z_\sigma$ associated with the interior of the cone $\omega$, and we define the associated idealized Artin cone $\mathcal{A}_{\sigma,\omega}$ to be the stack quotient $[Z_{\sigma,\omega}/T_{\sigma}]$. 
\end{definition}

Artin fans give a useful method for faithfully embedding a category of tropical objects into the category of algebraic stacks, as the following theorem makes precise:

\begin{theorem}[\cite{stacktrop} Theorem $6.11$]
There is an equivalence of $2$-categories between the categories of Artin fans and cone stacks. For a cone $\sigma$, the equivalence is given by sending $\sigma$ to $\mathcal{A}_\sigma$. 
\end{theorem}
By the theorem above, for any subdivision or lattice coarsening of a cone stack $\widetilde{\Sigma} \rightarrow \Sigma$, there is a corresponding morphism of Artin fans $\mathcal{A}_{\widetilde{\Sigma}} \rightarrow \mathcal{A}_{\Sigma}$, and we refer to the morphism on either side of the equivalence as a subdivision or lattice coarsening. The corresponding morphism of Artin fans is proper and log \'etale by \cite[Theorem $2.4.1$]{bir_GW}.

The data of an fs log scheme $X$ is given by a map to Olsson's stack of log structure $X \rightarrow Log$. We define an Artin fan for $X$ to be any Artin fan $\mathcal{X}$ such that $X \rightarrow Log$ admits a factorization through the natural map $\mathcal{X} \rightarrow Log$. Moreover, given a log \'etale morphism $\widetilde{\mathcal{X}} \rightarrow \mathcal{X}$ induced by subdivisions or lattice coarsenings, we may construct a log \'etale modification or more generally log \'etale alteration $\widetilde{X} \rightarrow X$ given by pulling back the log \'etale morphism of Artin fans.

While the canonical assignment of an Artin fan to an fs log stack in \cite[Proposition $3.2.1$]{bound} fails to be functorial, there nonetheless exists a functor, constructed in \cite[Appendix $C$]{punc}, from the category fs log stacks (and more generally fine log stacks) to the category of \emph{generalized cone complexes}, as defined in \cite[$\S 2.2$]{ACP}. For $\Sigma$ a generalized cone complex we denote by $PL(\Sigma)$ the $\mathbb{R}$ vector space of piecewise linear functions on $\Sigma$, and $PP(\Sigma)$ the space of piecewise polynomial functions. Note that we have a map $Sym^*(PL(\Sigma)) \rightarrow PP(\Sigma)$, but this map is not typically surjective. However, this map is surjective for a simple log algebraic stacks, and every log smooth stack has a log \'etale model which is a simple log algebraic stack, see \cite[Lemma $3.6$]{intDR}. When a smooth log scheme $X$ is induced by a simple normal crossings divisor $D = D_1+\cdots+D_k$, a component $D_i$ induces a piecewise linear function on $\Sigma(X)$ which has slope $1$ along the ray associated with $D_i$ and $0$ along all other rays. We also denote by $D_i \in PL(\Sigma(X))$ the PL function associated with $D_i$, with context distinguishing these uses.

\subsection{Punctured log Gromov-Witten Theory}
We now fix a log smooth morphism $X \rightarrow B$, where $B$ is either log smooth over a $\Spec\kk$ for $\kk$ an algebraically closed field of characteristic $0$ or the standard log point $\Spec \kk^{\dagger} = (\Spec\kk,\NN)$, with Artin fans $\mathcal{A}_X$ and $\mathcal{A}_B$. For $\sigma \in \Sigma(X)$, we define $\sigma_{\NN}$ the set of integral points of $\sigma$ and $\sigma_{\NN}^{\gp}$ the collection of integral tangent vectors at any point of $\sigma \setminus \partial \sigma$. Finally, for a generalized cone complex $\Sigma$, define $\Sigma^{[k]}$ to be the collection of $k$-dimensional cones of $\Sigma$, $\Sigma(\NN) = \cup_{\sigma\in \Sigma} \sigma_{\NN}$ and $\Sigma(\ZZ) = \cup_{\sigma \in \Sigma} \sigma_{\NN}^{\gp}$. For either a cone $\sigma \in \Sigma(X)$ or a point $p \in \sigma\setminus \partial \sigma$, we let $X_{\sigma}$ or $X_p$ refer to closure of the stratum with generic point mapping to the point of $\mathcal{A}_X$ associated with the cone $\sigma$. 

We will primarily study the enumerative geometry of log smooth pairs $(X,D)$ in this paper using punctured log Gromov-Witten theory. Developed originally by Abramovich, Chen, Gross and Siebert in \cite{punc}, this is a more flexible framework for the studying the relative Gromov-Witten theory of log smooth morphisms $X \rightarrow B$, in particular providing a useful way of recursively describing the boundary of moduli spaces of log stable maps. The primary input yielding this additional flexibility comes from tropical geometry.

More precisely, along with fixing genus and degrees in specifying moduli spaces of log stable maps, we will also fix the data of a \emph{realizable tropical type}. This comes from the data of a PL map $\Gamma_{\tau} \rightarrow \Sigma(X)$ from a source meterized graph with legs $\Gamma_{\tau}$ to $\Sigma(X)$. Letting $G_{\tau}$ be the abstract graph associated to $\Gamma_{\tau}$, such a map defines contact orders $u(l),u(e) \in \Sigma(X)(\ZZ)$ for all legs $l \in L(G_{\tau})$ and edges $e \in E(G_{\tau})$, and target cones $\pmb\sigma(v) \in \Sigma(X)$ for all vertices $v \in V(G_{\tau})$. In addition, for any realizable tropical type $\tau$, \cite{punc} provides a monoid whose dual is a cone which will also be denoted by $\tau$ which is the base of the universal family of tropical maps of type $\tau$. More precisely, for any family of tropical maps $\Gamma/\omega \rightarrow \Sigma(X)$ of type $\tau$ parameterized by a base cone $\omega$, there exits a unique integral linear map $\omega \rightarrow \tau$ such that the family over $\omega$ is pulled back from the universal family over $\tau$. Given another tropical type $\tau'$, we say that $\tau'$ is marked by $\tau$ if there is a face $\tau \subset \tau'$ such that the restriction of the universal family over $\tau'$ to the face $\tau$ yields the universal family over $\tau$. For a simple normal crossings pair \footnote{\cite{punc} defines types of punctured log maps in larger generality in the presence of monodromy in the log structure. For our purposes, defining such types in the setting of simple normal crossings pairs suffices.}, we further consider combinatorial types $\beta$ of punctured log map in which we specify only the total curve class $\textbf{A}$, genus $g$, and the tangency conditions $u(l) \in \pmb\sigma(l)^{\gp}_{\NN}$ of legs. We say a tropical type $\pmb\tau$ is marked by $\beta$ if $\sum_{v \in V(G_{\tau}} \textbf{A}(v) = \textbf{A}$, $\sum_{v \in V(G_{\tau})} g(v) = g$, and the tangency conditions of the legs of $\tau$ agree with those specified by $\beta$.

Additionally, letting $M$ be a monoid of effective curve classes, we consider decorations $g: V(G_{\tau}) \rightarrow \NN$ and $\textbf{A}: V(G_{\tau}) \rightarrow M$ of the vertices of the graph $G_{\tau}$ with genus and curve classes respectively. We call a tropical type with this additional data a decorated tropical type, which we denote by $\pmb\tau$. A decorated tropical type $\pmb\tau$ specifies moduli spaces $\mathscr{M}(X,\pmb\tau)$ and $\mathfrak{M}(\mathcal{X},\pmb\tau)$ of punctured log maps marked by the type $\tau$ with targets $X$ and $\mathcal{X}$ respectively. For the purpose of gluing, it will also be necessary to consider the stack $\mathfrak{M}^{\ev}(\mathcal{X},\pmb\tau) = \mathfrak{M}(\mathcal{X},\pmb\tau) \times_{\prod_{l \in L(G_{\tau})}\mathcal{X}_{\pmb\sigma(l)}} X_{\pmb\sigma(l)}$. Finally, for a realizable tropical type $\pmb\tau$, we refer to the punctured log Gromov-Witten classes of $\mathscr{M}(X,\pmb\tau)$ as classes of the form $\alpha\cap[\mathscr{M}(X,\pmb\tau')]^{\vir}$ for $\pmb\tau'$ a decorated tropical type marked by $\pmb\tau$ and $\alpha \in A^*(\mathscr{M}(X,\pmb\tau))$ operational Chow classes given as polynomials in tautological classes, in the sense of \cite{tauttarget} pulled back along $\mathscr{M}(X,\pmb\tau') \rightarrow \mathscr{M}(X)$ which forgets log structures.


\section{Log Gromov-Witten invariants of $\tilde{X}$ from log Gromov-Witten invariants of $X$}

Letting $X \rightarrow B$ be as in the previous section, we consider $X$ equipped with it's relative Artin fan $X \rightarrow\mathcal{X} = \mathcal{A}_X \times_{\mathcal{A}_B} B$, and suppose we have a log \'etale modification $\pi: \widetilde{X}\rightarrow X$ pulled back from a subdivision and lattice coarsening of Artin fans $\widetilde{\mathcal{A}}_X \rightarrow \mathcal{A}_X$. Letting $\Sigma(\widetilde{X})\rightarrow \Sigma(X)$ be the corresponding morphism in rational polyhedral cone complexes, consider a realizable tropical type $\gamma$ of punctured tropical map $\Gamma_\gamma \rightarrow \Sigma(\widetilde{X})$. We recall from \cite{pbirinv} the notion of a tropical lift:

\begin{definition}[\cite{pbirinv} Definition $4.1$]\label{dtlift}
Suppose we have realizable tropical types $\gamma$ and $\tau$ of punctured tropical maps to $\Sigma(\widetilde{X})$ and $\Sigma(X)$ respectively. Then $\gamma$ is a tropical lift of $\tau$ if there exists an inclusion of cones $i: \gamma \rightarrow \tau$ not factoring through a strict face inclusion $\tau' \subset \tau$ and a puncturing of the family of tropical curves $\overline{\Gamma}^\circ\subset \overline{\Gamma}_{\gamma}$ over the cone $\gamma$ induced by $i$ which is contained in the domain of definition of the punctured tropical map $(\overline{\Gamma}_{\gamma}^{\circ}/\gamma,f)$ with target $\Sigma(X)$ induced by $i$ such that the universal family of punctured tropical maps over $\gamma$ is given by $\overline{\Gamma}^\circ \times_{\Sigma(X)} \Sigma(\tilde{X}) \rightarrow \Sigma(\widetilde{X})$ in the category of rational polyhedral cone complexes.

A decorated tropical type $\pmb\gamma$ is a \emph{decorated lift} of a decorated tropical type $\pmb\tau$ if the underlying tropical type $\gamma$ is a tropical lift of the underlying type $\tau$, and after applying the morphism $\pi_*:H_2(\widetilde{X}) \rightarrow H_2(X)$ to the decoration of $\gamma$, the resulting decorated tropical type of map to $\Sigma(X)$ is marked by $\pmb\tau$. 
\end{definition}

We have a corresponding moduli space of log stable maps $\mathscr{M}(\widetilde{X}/B,\gamma)$, and by \cite[Theorem $4.6$]{pbirinv} the space is non-empty only if $\gamma$ is a tropical lift of some tropical type $\tau$ of a punctured tropical map $\Gamma_{\tau} \rightarrow \Sigma(X)$. If we additionally enhanced $\gamma$ to a decorated tropical type $\pmb\gamma$, then $\pmb\gamma$ is the lift of a unique decorated tropical type $\pmb\tau$, and stabilization induces a morphism of moduli stacks $\mathscr{M}(\widetilde{X}/B,\pmb\gamma) \rightarrow \mathscr{M}(X/B,\pmb\tau)$. In this setting, we recall the decorated version of the main theorem of \cite{pbirinv}, as well as a corollary which holds under an additional assumption on $\gamma$:

\begin{theorem}\label{decthm}
For a decorated tropical type $\pmb\gamma$ lifting the decorated tropical type $\pmb \tau$, denote by $\mathscr{M}_{\pmb\tau}$ the image of $\mathscr{M}(\widetilde{X}/B,\pmb\gamma) \rightarrow \mathscr{M}(X/B,\pmb\tau)$. Then $\mathscr{M}_{\pmb\tau}$ is a union of connected components of $\mathscr{M}(X/B,\pmb\tau)$, and there exists a proper log \'etale morphism $\mathfrak{M}_{\gamma\rightarrow \tau} \rightarrow \mathfrak{M}_{\tau}$ such that the following diagram is cartesian:

\begin{equation}\label{cart}
\begin{tikzcd}
\mathscr{M}(\widetilde{X}/B,\pmb\gamma) \arrow{r} \arrow{d} & \mathfrak{M}_{\gamma \rightarrow \tau} \arrow{d} \\
\mathscr{M}_{\pmb\tau} \arrow{r} & \mathfrak{M}_{\tau}.
\end{tikzcd}
\end{equation}
Moreover, the relative perfect obstruction theories attached to the horizontal morphisms are related via pullback.
\end{theorem}

\begin{corollary}\label{puncmcr1}
Suppose $\gamma$ is a lift of $\tau$, and dim $\tau =$ dim $\gamma$. Let $m = \coker|\gamma^{\gp} \rightarrow \tau^{\gp}|$. Then for $s: \mathscr{M}(\widetilde{X}/B,\gamma) \rightarrow \mathscr{M}(X/B,\tau)$ the map induced by $\pi: \widetilde{X} \rightarrow X$: 

\[s_*[\mathscr{M}(\widetilde{X}/B,\pmb\gamma)]^{\vir} = \frac{1}{m}[\mathscr{M}_{\pmb\tau}]^{\vir} .\] 
\end{corollary}
Note that if the decorations are both taken in the monoids $H_2^+(X)$ and $H_2^+(\widetilde{X})$ of effective curve classes in Theorem \ref{decthm}, we have $\mathscr{M}_{\pmb\tau}  =\mathscr{M}(X/B,\pmb\tau)$. To ease notational burden in this section, we will typically write $\mathscr{M}(X/B,\pmb\tau) = \mathscr{M}(X,\pmb\tau)$.

To prove Theorem \ref{mthm1}, we wish to pullback piecewise polynomial functions $P_i \in PP(\Sigma(X))$ to piecewise polynomial functions on $\Sigma(\mathscr{M}(X,\pmb\gamma))$. However, letting $X_s$ be the stratum of $(X,D)$ determined by the contact order of the marked point $x$, the evaluation morphism of algebraic stacks $\ev_x: \mathscr{M}(X,\pmb\gamma) \rightarrow X_s$ typically does not admit an enhancement to a morphism of log stacks. As constructed in \cite{DRtoric}, we modify the log structure on $X_s$ to allow for this as follows:

First, consider the ghost sheaf $\overline{\mathcal{M}}_{X_s}$ of the log structure on $X_s$. This has global sections given by PL functions on the cone complex $\Sigma(X_s) = \text{Star}(\sigma_s)$. Letting $\sigma_s \in \Sigma(X)$ be the cone associated with the stratum $X_s$, define $\Sigma_{\sigma_s}$ to be the cone complex with cones of the form $q(\sigma) \subset \sigma^{\gp}/\sigma_s^{\gp}$ for $q: \sigma^{\gp} \rightarrow \sigma^{\gp}/\sigma_s^{\gp}$ the quotient map, and note we have a natural quotient map $q: \Sigma(X_s) \rightarrow \Sigma(X_s)_{\sigma_s}$. After identifying $\Gamma(\overline{\mathcal{M}}(X_s))$ with PL functions on $\Sigma(X_s)$, we let $\overline{\mathcal{M}}' \subset \overline{\mathcal{M}}_{X_s}$ be the subsheaf with local sections associated with partial PL functions on $\Sigma(X_s)$ which are pulled back from $\Sigma(X_s)_{\sigma_s}$ along $q$. Define the sub log structure $\mathcal{M}'_{X_s} := \overline{\mathcal{M}}' \times_{\overline{\mathcal{M}}_{X_s}} \mathcal{M}_{X_s}$, with structure morphism $\alpha': \mathcal{M}' \rightarrow \mathcal{O}_{X_s}$ given by restricting $\alpha: \mathcal{M}_{X_s} \rightarrow \mathcal{O}_{X_s}$, and call the resulting fs log scheme $X_s'$. We note there is a log morphism $X_s \rightarrow X_s'$ induced by the inclusion of log structures $\mathcal{M}'_{X_s} \subset \mathcal{M}_{X_s}$. 

To construct an enhancement of the morphism $\ev_x: \mathscr{M}(X,\pmb\gamma) \rightarrow X_s$ to a log morphism $\ev_x: \mathscr{M}(X,\pmb\gamma) \rightarrow X_s'$ for $x$ a marked point of the tropical type $\pmb\gamma$, we let $\mathfrak{C} \rightarrow \mathscr{M}(X,\pmb\gamma)$ be the universal curve, $S_x \subset \mathfrak{C}$ the universal marked point. By postcomposing with the log morphism $X \rightarrow X'$ constructed above, we also have a morphism $S_x \rightarrow X_s'$. Note that the contact order $u(x)$ evaluated at all divisors which do not contain $X_s$ is zero by definition. Hence, for $\gamma_l$ the universal family of legs associated with $x$ containing a universal family of vertices $\gamma_v \subset \gamma_l$, we have $q\circ h_{\tau}: \gamma_l \rightarrow \Sigma(X_s)_{\sigma_s}$ is equal to $q \circ \ev_v \circ \pi_v$, for $\pi_v: \gamma_l \rightarrow \gamma_v$ the projection map. In particular, every PL function on $\Sigma(X_s)_{\sigma_s}$ pulls back to a PL function on $\gamma_l$ which give slope zero functions along the family of legs. Since all such functions pullback from PL functions on $\Sigma(\mathscr{M}(X,\pmb\gamma))$, it follows that the generalized Cartier divisor associated with the pullback of any PL function on $\Sigma(X_s') = \Sigma(X_s)_{\sigma_s}$ is the pullback of the generalized Cartier divisor on $\mathscr{M}(X,\pmb\gamma)$ associated with a PL function on $\Sigma(\mathscr{M}(X,\pmb\gamma))$.  These compatibilities ensures that the evaluation map admits a log enhancement, as we record in the following lemma:

\begin{lemma}\label{logenh}
For $l \in L(G_{\gamma})$ with $v \in L$ the unique vertex contained in $L$, the morphism of underlying algebraic stacks $\mathscr{M}(X,\pmb\gamma) \rightarrow X_s$ admits an enhancement to a log morphism $\mathscr{M}(X,\pmb\gamma) \rightarrow X_s'$, which tropicalizes to $\ev_v: \Sigma(\mathscr{M}(X,\pmb\gamma))  \rightarrow \Sigma(X_s)_{\sigma_s}$.
 \end{lemma}
 
Letting $\Ev = \prod_{l \in L(G_{\gamma})} \widetilde{X}_{\pmb\sigma(l)}'$, we define $\ev: \mathscr{M}(\widetilde{X},\pmb\gamma) \rightarrow \Ev$ to be the product of the evaluation maps defined in Lemma \ref{logenh}.

To define the terms appearing in Theorem \ref{mthm1}, consider a toric cone $\sigma$ of dimension $d$ with associated toric variety $Z_{\sigma}$. Any subdivision and lattice coarsening $\widetilde{\sigma} \rightarrow \sigma$ has an associated morphism of toric stacks $Z_{\widetilde{\sigma}} \rightarrow Z_{\sigma}$. Recall the relationship between piecewise polynomials on the fans $\sigma$ and $\widetilde{\sigma}$ and the $\mathbb{G}_m^d$ equivariant Chow theory of $Z_{\sigma}$ and $Z_{\widetilde{\sigma}}$. 

\begin{theorem}[Theorem $1$ \cite{ppchow}, Theorem $3.3.1$ \cite{logint}]
For a fan $\Sigma$ in $\mathbb{R}^{d}$, we have:
\[PP(\Sigma) \cong CH_{\mathbb{G}_m^d}(Z_\Sigma) \cong CH([Z_{\Sigma}/\mathbb{G}_m^d])\]
More generally, for a cone stack $\Sigma$ with associated Artin fan $\mathcal{A}$, we have:

\[PP(\Sigma) \cong CH_{\text{op}}^*(\mathcal{A}).\]
\end{theorem}

To define $\deg_{\sigma}(P)$ for $P \in PP(\widetilde{\sigma})$, we suppose we have a collection splittings of the restriction maps $PL(\omega') \rightarrow PL(\sigma)$ for a subset of faces $\omega \subset \sigma$ sending $L \in PL(\omega)$ to $L^{\text{ext}}$ such that  $L^{\text{ext}}$ is non-negative on $\sigma$ for all $L$ which are non-negative on $\omega$. We call the function $L^{\text{ext}}$ the extension of $L$. 

In the case that both $\sigma$ and $\omega$ are moduli spaces of tropical maps to a cone complex $\Sigma$, $PL(\omega)$ is rationally spanned by edge lengths and the pullback of PL functions on $\Sigma$ along evaluation maps $\ev_{v}: \omega \rightarrow \Sigma$ which are non-negative on $\pmb\sigma(v)$. The universal family $\Gamma_{\sigma} \rightarrow \sigma$ is marked by the tropical type of $\omega$, hence we have a contraction map $\phi: G_{\sigma} \rightarrow G_{\omega}$, inducing a bijection from non-contracted edges of $G_{\sigma}$ to edges of $G_{\omega}$. In particular, for every edge $e \in E(G_{\omega})$, the marking gives a canonical lift of the edge length function $l_e \in PL(\omega)$ to a linear function on $\sigma$. Additionally, for any vertex $v \in V(G_{\omega})$, after making the additional choice of a half edge or leg $e$ containing $v$, for any PL function $L \in PL(\Sigma)$, we extend $\ev_v^*L \in PL(\omega)$ to a PL function on $\sigma$ by $\ev_{v'}^*L$ for $v'$ the vertex contained in the half edge $e$. Any choice of basis for the space of rational PL function on $\omega$ consisting of edge lengths and vertex positions yields a modular choice of such a splitting $PL(\omega) \rightarrow PL(\sigma)$. Such splittings will be used implicitly in the context of moduli of log stable maps.



We construct a fan $\Sigma^{\sigma}$ containing $\sigma$, and define $\deg_{\sigma}(P)$ via an equivariant pushforward for a choice of lift of $P$ associated with our extension data. First, for any face $\omega\subset \sigma$ and a linear function $L \in PL(\omega)$, we let $H_L = \ker(L^{\text{ext}}) \subset \sigma^{\gp}$, and we consider the intersection $H_{\omega}$ of all $H_L$ for $L \in PL(\omega)$, a $\dim \sigma - \dim \omega$ dimensional integral subspace of $\sigma^{\gp}$. Let $\Sigma^{\sigma}$ be any simplicial fan which supports each $H_{\omega}$ as a union of cones. Letting $\widetilde{\Sigma}^{\sigma} \rightarrow \Sigma^{\sigma}$ be any subdivision containing a subdivision of $\tilde{\sigma}$ as a subcomplex. 

To define $\deg_{\sigma}(P)$, we extend a homogeneous piecewise polynomial $P \in PP(\widetilde{\sigma})$ of degree $k$ to a piecewise polynomial function on $\widetilde{\Sigma}^{\sigma}$. Since $\widetilde{\Sigma}^{\sigma}$ is a simplicial fan, for every one dimensional cone $\rho$ of $\widetilde{\Sigma}^{\sigma}$, there exists a unique PL function $Q_{\rho}$ which has slope $1$ along $\rho$ and slope $0$ along all other one dimensional cones, and the collection of $PL$ function $Q_{\rho}$ form a basis for the $\mathbb{Q}$ vector space of PL functions on $\widetilde{\Sigma}^{\sigma}$. More generally, for a cone $\gamma$, we let $Q_{\gamma} = \prod_{\rho \in \gamma^{[1]}} Q_{\rho}$. 


Suppose that $P = Q_{\gamma}P'$. Since $\gamma$ could be a zero dimensional face, we assume this without loss of generality. If $P$ is non-vanishing on $\gamma$, we let $\omega$ be the smallest face of $\sigma$ which contains the image of $\gamma$. If $\omega = \sigma$, then $Q_{\gamma}P'$ vanishes on the boundary of $\sigma$, and we extend $P$ by zero on all other cones of $\widetilde{\Sigma}^{\sigma}$. To extend when $\omega$ is a proper face of $\sigma$, we consider the fan $\widetilde{\Sigma}^{\omega}$ on $\sigma^{\gp}/H_\omega$ induced by the fan on $\widetilde{\Sigma^{\sigma}}$. By the assumption that the extension of non-negative linear functions of $\omega$ are non-negative, the image of $\sigma$ under the quotient map is $\omega$. Thus, up to further subdividing $\widetilde{\Sigma^{\sigma}}$ the quotient map induces a map of fans $\widetilde{\Sigma^{\sigma}} \rightarrow \widetilde{\Sigma^{\omega}}$ mapping $\tilde{\sigma}$ to $\tilde{\omega}$. The piecewise polynomial function $Q_{\gamma}P'|_{\widetilde{\omega}}$ vanishes on the boundary of $\omega$ hence we can extend the former piecewise polynomial by zero on all other cones of $\widetilde{\Sigma}^{\omega}$. Pulling back this piecewise polynomial along $\widetilde{\Sigma}^{\sigma} \rightarrow \widetilde{\Sigma}^{\omega}$ yields an extension.

Using the section $PL(\widetilde{\omega}) \rightarrow PL(\widetilde{\sigma})$ of the restriction map $res_{\omega}: PP(\widetilde{\sigma}) \rightarrow PP(\widetilde{\omega})$ constructed above, we produce an isomorphism $$Sym^k(PL(\widetilde{\sigma})) \cong \oplus_{i=0}^k Sym^i(ker(res_{\omega}))\otimes Sym^{k-i}(PL(\widetilde{\omega})).$$ If $P'$ is non-vanishing on $\gamma$, the splitting ensures a unique expression of the form $P' = P^{\text{ext}} + P''$ with $P^{\text{ext}}$ a non-zero polynomial in extensions of PL functions on $\tilde{\omega}$ and $P''$ a piecewise polynomial which vanishes on $\omega$. The construction of the extension $P^{\text{ext}}$ also gives an extension to a piecewise polynomial on the entire fan $\widetilde{\Sigma}^{\sigma}$. Since $P''|_{\gamma} = 0$, there exists a unique expression $P'' = \sum_{\gamma \subset \gamma'} Q_{\gamma'}P''_{\gamma'}$ with $P''_{\gamma'}$ a piecewise polynomial with $P''_{\gamma'}|_{\gamma'} \not= 0$ and of degree strictly less than $P''$. By extending the constant function to the constant function when $\deg(P') = 0$, we define our extensions in general induction on $\deg(P')$. By construction, we also have the extension of $P + Q$ is given by the sum of the extensions, and we extend to possibly non-homogeneous piecewise polynomials by linearity.

To state a formula for $\deg_{\sigma}(P)$, we recall the definition of equivariant multiplicity of a cone $\sigma \in \mathbb{R}^n$. The equivariant multiplicity $e_\sigma$ is a rational polynomial in $n$ variables, uniquely determined by the following two properties:

\begin{enumerate}
\item If $\sigma_1,\ldots,\sigma_r$ are the maximal cones of a rational polyhedral subdivision of a cone $\sigma$:
\[e_{\sigma} = \sum_{i=1}^r e_{\sigma_i}.\]

\item If $\sigma$ is a unimodular cone spanned by a basis $e_1,\ldots,e_n$ for the lattice $N$ with dual basis $e_1^*,\ldots,e_n^*$, then:
\[e_{\sigma} = \frac{1}{e_1^*\cdots e_n^*}.\]
\end{enumerate}

\begin{definition}\label{degdef}
For any cone $\sigma$ and extension data $PL(\omega) \rightarrow PL(\sigma)$ for faces $\omega \subset \sigma$, we define the polynomial, hence a Chow cohomology class on $B\mathbb{G}_m^{\dim\sigma}$:

\begin{equation}\label{degdef}
\deg_{\sigma}(P)= \sum_{\tau \in \widetilde{\Sigma^{\sigma}}} e_\tau P|_{\tau}.
\end{equation}
\end{definition}
The following proposition shows that $\deg_{\sigma}(P)$ is well defined and vanishes for on a certain class of piecewise polynomial function:

\begin{proposition}\label{degprop}
For a piecewise polynomial $P \in PP(\widetilde{\sigma})$, the rational function $\deg_{\sigma}(P)$ is a homogeneous polynomial of degree $\deg(P) - \dim\sigma$. Moreover, for $P$ any product of piecewise linear functions on $\tilde{\sigma}$ which are extensions of PL functions on a single subdivided face $\widetilde{\omega} \subset \widetilde{\sigma}$, and $\gamma$ a top dimensional face of $\widetilde{\omega}$, we have $\deg_{\sigma}(Q_{\gamma}P) = 0$.



\end{proposition}

\begin{proof}

The fact that $\deg_{\sigma}(P)$ is a homogeneous polynomial of degree $\deg(P) = \dim\sigma$ follows by \cite[Proposition $1.1$]{PPM}. Additionally, given choices of fans $\widetilde{\Sigma}^{\sigma},\widetilde{\Sigma}^{\sigma'}$ used in the definition of $\deg_{\sigma}(P)$, we can pass to a common refinement $\widetilde{\Sigma}^{\sigma''}$, and by the projection formula we have $\deg_{\sigma}(P)$ in both cases can be computed as an equivariant pushforward from $Z_{\widetilde{\Sigma}^{\sigma''}}$.



For the second claim, recall the fan $\widetilde{\Sigma}^{\omega}$ on $\sigma^{\gp}/H_\omega$ induced by the fan on $\widetilde{\Sigma^{\sigma}}$. By the assumption and the construction of the extension, $P$ is the pullback of a product of PL functions along $\widetilde{\Sigma^{\sigma}} \rightarrow \widetilde{\Sigma^{\omega}}$. Since this map corresponds to a proper map of toric varieties $Z_{\widetilde{\Sigma}^{\sigma}} \rightarrow Z_{\widetilde{\Sigma}^{\omega}}$ which maps the positive dimensional stratum $Z_{\widetilde{\Sigma}^{\sigma},\gamma}$ to a zero stratum of $Z_{\widetilde{\Sigma}^{\omega}}$, and the cohomology class corresponding to $P$ is pulled back along this map, the projection formula ensures that the equivariant integral of interest vanishes. 

\end{proof}
 
 \begin{example}\label{simpdeg}
Suppose that $\tau$ is a cone of dimension $n$ and $L_1,\ldots,L_n$ a collection of rational piecewise linear functions on $\tau$ which are linear on an $n$-dimensional simplicial subcone $\gamma \subset \tau$ such that $\prod_i L_i$ is a polynomial function which vanishes on $\partial \gamma \cup (\tau\setminus \gamma)$. Letting $\gamma_{\NN}^{\gp} \subset \tau_{\NN}^{\gp}$ be the maximal sublattice on which all $L_i$ are integral, then 
$$\deg_{\tau}(\prod_i L_i) = \frac{|\coker(\prod_i L_i: \gamma^{\gp}_{\NN} \rightarrow \ZZ^n)|}{|\coker(\gamma^{\gp}_{\NN} \rightarrow\tau_{\NN}^{\gp})|}.$$ 
To see this, for $\rho \in \gamma^{[1]}$, let $L_\rho \in PL(\gamma)$ be the rational primitive integral linear function which on the one dimensional skeleton of $\gamma$ has slope $1$ along $\rho$ and slope $0$ on all other one-dimensional faces of $\tau$. After passing to a subdivision and lattice coarsening $\tilde{\tau}\rightarrow \tau$, giving a log alteration $\pi: Z_{\tilde{\tau}} \rightarrow Z_{\tau}$, these pullback to integral PL functions on $\tilde{\tau}$, and there is a zero stratum of $Z_{\tilde{\tau}}$ with equivariant Euler class $\prod_{\rho\in \gamma^{[1]}} L_\rho$. Since $\pi$ is proper and $\pi_*[Z_{\tilde{\tau}}] = [Z_{\tau}]$, we have $\deg_{\tau}(\prod_i L_i) = \deg_{\tilde{\tau}}(\pi^*\prod_iL_i)$. Since $L_i$ vanishes on all codimension one faces of $\gamma$, we have $\prod_iL_i = \lambda\prod_{\rho \in \tau^{[1]}} L_{\rho}$, where $\lambda$ the absolute value of the determinant of the change of basis matrix relating $L_1,\ldots,L_n$ with $L_{\rho_1},\ldots,L_{\rho_n}$ for an enumeration of dimension $1$ cones $\rho_1,\ldots,\rho_n$, which is the desired ratio of lattice indexes.

\end{example}

In general, for $\sigma$ a cone stack admitting a strict cover by a cone $i: \sigma' \rightarrow \sigma$, we let $\deg_{\sigma}(P) = \deg_{\sigma'}(i^*P)$. For $\deg_{\sigma}(P)$ to be well defined, we require $\deg_{\sigma}(P)$ to be a piecewise polynomial on $\sigma$ and to be independent of the choice of $i$: 

\begin{lemma}
The polynomial $\deg_{\sigma}(P)$ is the pullback of a unique polynomial on $\sigma$ along $i$, and is independent of the choice of cover $i$. 
\end{lemma}

\begin{proof}
By construction, $\deg_{\sigma}(P)$ is the sum of rational functions pulled back along $i$. Since $i^*$ induces a morphism on the $\mathbb{Q}$-algebra of rational functions, it follows that $\deg_{\sigma}(P)$ is pulled back along $i$. For independence of $i$, note that for any other strict cover $i':  \sigma'' \rightarrow \sigma$ by a cone $\sigma''$, note that the pulled back map $\sigma' \times_{\sigma} \sigma'' \rightarrow \sigma'$ is a strict cover, hence must be an isomorphism since cones admit no non-trivial strict cover. Since the analogous map $\sigma' \times_{\sigma} \sigma'' \rightarrow \sigma''$ is also an isomorphism for the same reason, we therefore have an isomorphism $\sigma' \rightarrow \sigma''$ commuting with the map to $\sigma$. The conclusion follows. 
\end{proof}

The second term requiring definition is the piecewise polynomial $Q_{\gamma}$ associated with a cone $\gamma \in \tilde{\tau}$ in a subdivision and lattice coarsening $\tilde{\tau}$ of $\tau$. By passing to a further simplicial subdivision $\widetilde{\tau}$ of $\tau$ which contains $\gamma$ as a union of smooth subcones $\gamma = \cup_{i} \gamma_i$, we let $Q_{\gamma}$ be the piecewise polynomial on $\widetilde{\tau}'$ given by a scalar multiple of $\sum_{\gamma'} Q_{\gamma'} = \sum_{\gamma'} \prod_{\rho \in \gamma^{'[1]}} L_\rho$, where $L_\rho$ is an integral linear function which is $1$ at a primitive point of a one-dimensional cone $\rho$ and zero on all other one-dimensional cones. The scalar multiple we pick is chosen so that $\deg_{\gamma}(Q_{\gamma}) = 1$. It follows from construction that for all cones $\omega \in \widetilde{\tau}$ not containing $\gamma$ as a face, we have $Q_{\gamma}|_{\omega} = 0$ and $Q_{\gamma}$ agrees with the previous given definition in the case that $\gamma$ is a simplicial cone.

Using the piecewise polynomial $Q_{\gamma}$, we prove two lemmas which will be used in the proof of Theorem \ref{mthm1}:
 
\begin{lemma}\label{linlem}
Let $\mathcal{A}_{\Sigma}$ be an Artin fan, $\gamma \in \Sigma$, and $L \in PL(\Sigma)$ such that $L$ is non-negative and $L|_{\gamma} = 0$. We have the following equation in $A_*(X)$:
 
 \begin{equation}\label{PPbase}
 L\cap[\mathcal{A}_{\Sigma,\gamma}]= \sum_{\gamma \subset \gamma'} \deg_{\gamma'}(a_{L,\gamma'}Q_{\gamma'})[\mathcal{A}_{\Sigma,\gamma'}].
 \end{equation}
 In the equation above, $a_{L,\gamma'} = |\coker(L: \gamma^{'\gp}_{\NN}/\gamma^{\gp}_{\NN} \rightarrow \ZZ)|$, with $L$ the induced linear map on the quotient.
 \end{lemma}

\begin{proof}
The non-negative function $L$ induces a map of Artin fans $L: \mathcal{A}_{\Sigma} \rightarrow \mathcal{A}_{\mathbb{R}_{\ge 0}}$, and, $c_1(L)$ is represented by an effective Cartier divisor $L^*([B\mathbb{G}_m])$. By pulling back this Cartier divisor further along an \'etale cover $j:\cup_{\sigma}\mathcal{A}_{\sigma} \rightarrow \mathcal{A}_{\Sigma}$ from a disjoint union of Artin cones and computing using local toric charts, this divisor is given by:
\begin{equation}
\begin{split}
j^*(L\cap [\mathcal{A}_{\Sigma,\gamma}]) &= \sum_{\sigma}\sum_{\substack{\gamma \subset \gamma'\\L|_{\gamma'} \not= 0\text{ }\dim\gamma' = \dim \gamma + 1}} a_{L,\gamma'}[\mathcal{A}_{\sigma,\gamma'}]\\ &= \sum_{\sigma}\sum_{\substack{\gamma \subset \gamma'\\L|_{\gamma'} \not= 0}} \deg_{\gamma'}(a_{L,\gamma'}Q_{\gamma'})[\mathcal{A}_{\sigma,\gamma'}]\\
&= j^*( \sum_{\gamma \subset \gamma'} \deg_{\gamma'}(a_{L,\gamma'}Q_{\gamma'})[\mathcal{A}_{\Sigma,\gamma'}]).
\end{split}
\end{equation}
Since flat pullback on rational Chow groups along a finite \'etale cover is injective, Equation \ref{PPbase} follows.

%

\end{proof}

\begin{lemma}\label{tbund1}
Let $\tau$ be an Artin fan with an \'etale cover by an Artin cone $\tau' \rightarrow \tau$ and $\pi: \tilde{\tau} \rightarrow \tau$ be a log \'etale modification pulling back to a subdivision $\tilde{\tau'} \rightarrow \tau'$. Then for $\tau^\circ \subset \tau$ the unique closed point, any strata closure $\gamma \subset \tilde{\tau}$ mapping to $\tau^\circ$ and any piecewise polynomial function $P \in PP(\text{Star}(\gamma))$, we have $PQ_{\gamma}$ defines a piecewise polynomial function on $\tilde{\tau}$ and:

\begin{equation}\label{ptbund}
\pi_*(PQ_{\gamma}[\tilde{\tau}]) = \deg_{\tau}(PQ_\gamma)[\tau^{\circ}].
\end{equation}

\end{lemma}

\begin{proof}
Up to further subdividing $\tau$, we assume that $Q_{\gamma}$ is a piecewise polynomial function defined on $\tilde{\tau}$. Since $Q_{\gamma}$ vanishes on all cones of $\tilde{\tau}$ not containing $\gamma$ as a face, extending by zero gives an extension of the piecewise polynomial function on $\text{Star}(\gamma)$. Note that the pushforward in Equation \ref{ptbund} is computed by the equivariant pushforward along $Z_{\tilde{\tau'}} \rightarrow Z_{\tau'}$. Standard formulas for equivariant pushforward, for instance \cite[Proposition $1.1$]{PPM} yield Equation \ref{ptbund} in the case $i: \tau' \rightarrow \tau$ is an isomorphism, observing that $\deg_{\tau}(PQ_{\gamma})$ is a polynomial as $PQ_{\gamma}$ vanishes on proper faces of $\tau$ by assumption on $\gamma$. 

More generally, commutativity of flat pullback with proper pushforward yields 
\begin{equation}\label{comfp}
i_*i^*\pi_*(PQ_{\gamma}[\tilde{\tau}]) = i_*\pi_*i^*(PQ_{\gamma}[\tilde{\tau}]).
\end{equation}
 Since $i: \tau' \rightarrow \tau$ is finite \'etale of degree $n$, the lefthand side of Equation \ref{comfp} is $n\pi_*(PQ_{\gamma}[\tilde{\tau}])$. For the righthand of Equation \ref{comfp},we 
 \begin{equation}
\begin{split}
i_*\pi_*i^*PQ_{\gamma}[\tilde{\tau}] &= i_*\deg_{\tau'}(PQ_{\gamma})[\tau^{'\circ}] \\
&= i_*\deg_{\tau'}(PQ_{\gamma})i^*[\tau^{\circ}]\\
&= n\deg_{\tau}(PQ_{\gamma})[\tau^\circ]\\
\end{split}
\end{equation}
Thus, Equation \ref{comfp} yields Equation \ref{ptbund} after scaling both sides of Equation \ref{ptbund} by the non-zero constant $n$, hence gives the desired equation. 
\end{proof}


Before proceeding to Theorem \ref{mthm1}, we require a more explicit description of strata in Artin fans. For this we consider the following generalization of the star fan of a cone $\tau$ in a fan $\Sigma$ to the setting of cone stacks, as well as the extension of piecewise linear functions on cones $\sigma \in \Sigma$:

\begin{definition}[\cite{logtaut} Definition $83$]
Given a cone stack $\Sigma$ and $\sigma \in \Sigma$, the \emph{star cone stack} $\text{Star}(\sigma)$, whose $\sigma''$ valued points are given by diagrams $(\sigma \rightarrow \sigma' \leftarrow \sigma'')$ with $\sigma'$ minimal among objects in $\Sigma$ receiving maps $j',j''$ from $\sigma,\sigma''$, and morphisms $Hom(\sigma \rightarrow \sigma_1' \leftarrow \sigma_1'',\sigma \rightarrow \sigma_2' \leftarrow \sigma_2'')$ given by pairs of morphisms $\phi': \sigma_1' \rightarrow \sigma_2',\phi'': \sigma_1'' \rightarrow \sigma_2''$ making the natural diagram commute. 

We also define $\Sigma_\sigma$ to be the cone stack given by diagrams $(\sigma'/\sigma \leftarrow \sigma'')$ with $\sigma'$ minimal among all objects in $\Sigma$ receiving maps $i',j'$ from $\sigma,\sigma''$, with morphisms defined as before.
\end{definition}

\begin{definition}
For a cone $\tau \in \Sigma$, a system of extensions is a collection of sections $(-)_{\tau'}: PL(\tau') \rightarrow PL(\text{Star}(\tau'))$ of the restriction map $PL(\text{Star}(\tau')) \rightarrow PL(\tau')$ for $\tau \subset \tau'$ sending non-negative functions on $\tau'$ to non-negative functions on $\text{Star}(\tau)$. Given a system of extensions, and $\gamma \in \widetilde{\Sigma}$ a cone of a refinement such that the interior of $\gamma$ maps to the interior of $\tau$, we define the operators $\deg_{\tau'}(Q_{\gamma}(-))$ as in Definition \ref{degdef}. 
\end{definition}

Implicit in the definition above is the claim that the given system of extensions determines an extension of $Q_{\tau}P$ to a piecewise polynomial function on any fan $\widetilde{\Sigma}^{\tau'}$ used to define $\deg_{\tau'}$. This is clear since $Q_{\gamma}$ as a function vanishes on all cones of $\Sigma$ which do not contain $\tau$ as a face, hence $\deg_{\tau'}(Q_{\gamma}P)$ is independent of the choice of extension data on cones of $\Sigma$ not containing $\tau$.

When $\Sigma$ is a moduli stack of tropical maps, then $\text{Star}(\tau)$ is the moduli stack of tropical maps marked by $\tau$, and extensions of linear functions to $\text{Star}(\tau)$ as edge length and the pullback of piecewise linear functions on $\Sigma(X)$ exist. Note that a system of extensions also yield dual retraction maps $\text{Star}(\tau') \rightarrow \tau'$.

The star cone stack construction together with an extension map $PL(\tau) \rightarrow PL(\text{Star}(\tau))$ provide a convenient description of strata in Artin fans which will be used in the proof of Theorem \ref{mthm1}:

\begin{lemma}\label{stratadecomp}
Let $\mathcal{A}_{\Sigma}$ be an Artin fan associated with a cone stack $\Sigma$, and $\mathcal{A}_{\Sigma,\tau}$ be the closure of the stratum of $\mathcal{A}$ corresponding to a cone $\tau \in \Sigma$. Then there is a strict \'etale morphism of log algebraic stacks $\mathcal{A}_{\text{Star}(\tau)} \rightarrow \mathcal{A}_{\Sigma}$ which is finite onto an open substack of $ \mathcal{A}_{\Sigma}$. Moreover, given an extension map $PL(\tau) \rightarrow PL(\text{Star}(\tau))$, there is an isomorphism of algebraic stacks $\mathcal{A}_{\text{Star}(\tau),\tau}:= \mathcal{A}_{\text{Star}(\tau)}\times_{\mathcal{A}_{\Sigma}} \mathcal{A}_{\Sigma,\tau}\rightarrow \mathcal{A}_{\text{Star}(\tau)_{\tau}}\times \mathcal{A}_{\tau,\tau} $ induced by a morphism of log algebraic stacks $\mathcal{A}_{\text{Star}(\tau)}\rightarrow \mathcal{A}_{\Sigma_{\tau}}\times \mathcal{A}_{\tau} $, and the pullback of a $\mathbb{G}_m$-torsor on $\mathcal{A}_{\text{Star}(\tau)}$ associated with an extension of a linear function on $\tau$ is $\pi_{\mathcal{A}_{\tau}}^*L$ for some $\mathbb{G}_m$-torsor $L$.


\end{lemma}

\begin{proof}
For the first statement, note that it suffices to prove the statement after replacing $\mathcal{A}_{\Sigma}$ with the open Artin fan given as the complement of strata whose fiber product with $\text{Star}(\tau)$ over $\Sigma$ is empty. Now consider a strict \'etale cover by Artin cones $\cup_i \mathcal{A}_{\sigma_i} \rightarrow \mathcal{A}_\Sigma$ and the pullback $\cup_i \mathcal{A}_{\sigma_i}\times_{\mathcal{A}_\Sigma} \mathcal{A}_{\text{Star}(\tau)}$, yielding an \'etale cover. By considering the associated functor of points on cones, we have $\sigma_i \times_{\Sigma} \text{Star}(\tau) = \sigma_i$. Since $\sigma_i \rightarrow \mathcal{A}_{\Sigma}$ is strict and \'etale, we have $\mathcal{A}_{\sigma_i} \times_{\mathcal{A}_{\Sigma}} \mathcal{A}_{\text{Star}(\tau)} \rightarrow \mathcal{A}_{\text{Star}(\tau)} \rightarrow \text{Log}$ is strict and \'etale, hence $\mathcal{A}_{\sigma_i} \times_{\mathcal{A}_{\Sigma}} \mathcal{A}_{\text{Star}(\tau)}$ is an Artin fan. Since $\mathcal{A}_{\sigma_i} \times_{\mathcal{A}_{\Sigma}} \mathcal{A}_{\text{Star}(\tau)}$ in particular satisfies the universal property in the $2$-category of Artin fans, by the equivalence of $2$-categories of \cite[Theorem $6.11$]{stacktrop}, we have $\mathcal{A}_{\sigma_i} \times_{\mathcal{A}_{\Sigma}} \mathcal{A}_{\text{Star}(\tau)} \cong \mathcal{A}_{\sigma_i \times_{\Sigma} \text{Star}(\tau)} \cong \mathcal{A}_{\sigma_i}$, and the projection $\mathcal{A}_{\sigma_i} \times_{\mathcal{A}_{\Sigma}} \mathcal{A}_{\text{Star}(\tau)}  \rightarrow \sigma_i$ is an isomorphism, in particular finite \'etale. Since the property of finite \'etale may be verified \'etale local on the target of a morphism, we have $\mathcal{A}_{\text{Star}(\tau)} \rightarrow \mathcal{A}_{\Sigma}$ is finite \'etale.

To prove the second part of the lemma, we first consider the case when $\mathcal{A}_{\Sigma}$ is an Artin cone $\mathcal{A}_{\sigma}$. In this case, $\mathcal{A}_{\sigma,\tau} = [Z_{\sigma,\tau}/T]$, with $Z_{\sigma,\tau}$ the closure of a torus orbit associated with the cone $\tau \subset \sigma$ in the toric variety associated with the cone $\sigma$. Letting $\text{Fix}(\tau) \subset T$ be the subtorus which fixes $Z_{\sigma,\tau}$ pointwise, recall that $Z_{\sigma,\tau}$ is a toric variety associated with the cone $\sigma_{\tau}$, with a dense torus isomorphic to $T' = T/\text{Fix}(\tau)$. Moreover, we have $T \cong \text{Fix}(\tau)\times T'$. To specify a projection $T \rightarrow \text{Fix}(\tau)$, we pick a map associated with the saturated sublattice of characters on $T$ corresponding to linear functions on $\sigma$ which are extensions of linear functions on $\tau$. It follows that:

$$X_{\tau} \cong [Z_{\sigma,\tau}/T\times \text{Fix}(\tau)] \cong [\mathcal{A}_{\sigma_{\tau}}/\text{Fix}(\tau)] \cong \mathcal{A}_{\sigma_{\tau}} \times \mathcal{A}_{\tau,\tau}.$$
This isomorphism is induced by a morphism of log algebraic stacks $ \mathcal{A}_{\sigma} \rightarrow \mathcal{A}_{\sigma_{\tau}} \times \mathcal{A}_{\tau}$ determined by the quotient map and the map $\mathcal{A}_{\sigma} \rightarrow \mathcal{A}_{\tau}$ corresponding to the extension map $PL(\tau) \rightarrow PL(\sigma)$.

In general, the Artin fan $\mathcal{A}_{Star(\tau)}$ admits an \'etale cover by Artin cones $\cup_i \mathcal{A}_{\sigma_i} \rightarrow \mathcal{A}_{Star(\tau)}$. We produce an \'etale cover of $\mathcal{A}_{Star(\tau),\tau}$ by idealized Artin cones by taking the fiber product $\cup_i \mathcal{A}_{\sigma_i}\times_{\mathcal{A}_{Star(\tau)}} \mathcal{A}_{Star(\tau),\tau} \rightarrow \mathcal{A}_{Star(\tau),\tau}$. Note that a component associated with $\sigma_i$ is empty if $\tau$ is not a face of $\sigma$, and observe $\mathcal{A}_{\sigma_i}\times_{\mathcal{A}_{Star(\tau)}} \mathcal{A}_{Star(\tau),\tau} \cong \mathcal{A}_{{\sigma}_{i\tau}} \times \mathcal{A}_{\tau,\tau} $ by the proven case of the lemma. We similarly have an \'etale cover $\cup_i \mathcal{A}_{{\sigma}_{i}} \times \mathcal{A}_{\tau} \rightarrow \mathcal{A}_{Star(\tau),{\tau}} \times \mathcal{A}_{\tau}$, inducing an \'etale cover $\cup_i \mathcal{A}_{\sigma_{i\tau}} \times \mathcal{A}_{\tau,\tau} \rightarrow \mathcal{A}_{Star(\tau)_{\tau}} \times \mathcal{A}_{\tau,\tau}$. The desired map $\mathcal{A}_{Star(\tau)} \rightarrow\mathcal{A}_{Star(\tau)_{\tau}} \times \mathcal{A}_{\tau}$ is constructed by \'etale descent. More precisely, we have 
$$\mathcal{A}_{\sigma_i} \rightarrow \mathcal{A}_{\sigma_{i\tau}}\times \mathcal{A}_{\tau}$$ by the proven case of the lemma, hence we have maps $$\mathcal{A}_{\sigma_i} \rightarrow \mathcal{A}_{\text{Star}(\tau)_{\tau}}\times \mathcal{A}_{\tau}$$
 These maps are also compatible with further restriction to faces by construction by the compatible system of extensions of linear maps on $\tau$. By noting that $\mathcal{A}_{\text{Star}(\tau)_\tau}$ has an \'etale cover by the same idealized Artin cones which cover $\mathcal{A}_{\text{Star}(\tau),\tau}$, we construct a map in the opposite direction again by \'etale descent.  The remaining desired properties follow by construction.

\end{proof}


We prove Theorem \ref{mthm1} by first proving a general version of the relation at the level of the Chow theory for Artin fans, and use virtual pullback in the sense of \cite{vpull} to produce the desired relation.


\begin{theorem}\label{normalform}
Let $X = \mathcal{A}_{\Sigma}$ be an Artin fan and $\pi: Y = \mathcal{A}_{\widetilde{\Sigma}} \rightarrow \mathcal{A}_{\Sigma}= X$ be a log \'etale map given by the composite of a subdivision and lattice coarsening. Suppose we have cones $\gamma \in \widetilde{\Sigma}$ and $\tau \in \Sigma$ with associated strata closures $Y_{\gamma} = \mathcal{A}_{\widetilde{\Sigma},\gamma}$ and $X_{\tau} = \mathcal{A}_{\Sigma,\tau}$ such that $\Sigma(\pi):\widetilde{\Sigma}\rightarrow \Sigma$ induces a map $\gamma \rightarrow \tau$ which does not factor through a strict face of $\tau$. Then given a system of extensions for $\tau$, with corresponding degree operators $\deg_{\tau'}$, we have the following equation:

\begin{equation}
\pi_*(P\cap [Y_{\gamma}]) = \sum_{\substack{\tau\subset \tau'}} \deg_{\tau'}(Q_{\gamma}P)[X_{\tau'}].
\end{equation}
\end{theorem}

\begin{proof}
We begin by making a series of reductions. First, we reduce to the case when $X = \text{Star}(\tau)$. If $\pi|_{Y_{\gamma}*}(P\cap[Y_{\gamma}]) =  \sum_{\substack{\tau\subset \tau'}} \deg_{\tau'}(Q_{\gamma}P)[X_{\tau'}] \in A_*(X_{\tau})$ satisfies the desired equation, then proper pushforward along the closed immersion $X_{\tau} \rightarrow X$ yields the desired equation. On the other hand, if we have strict closed immersions $Y_{\gamma} \rightarrow Y'$ and $X_{\tau} \rightarrow X'$ of Artin fans and a proper log \'etale morphism $\pi': Y' \rightarrow X'$ inducing $\pi|_{Y_{\gamma}}$ which satisfies the desired equation, then the desired equation holds for $\pi|_{Y_{\gamma}}$ after pushing forward along the closed immersion $X_{\tau} \rightarrow X'$ by the projection formula. Finally, the proper pushforward map $A_*(X_{\tau}) \rightarrow A_*(X)$ is injective, hence in the previous case we also conclude the desired equation holds for $\pi|_{Y_{\gamma}}$. Indeed, by repeated applications of excision, we reduce to the case when $X$ is an idealized Artin cone and $X_{\tau}$ is the closure of a stratum, which follows by the standard presentation of the equivariant Chow homology of toric varieties, see \cite[Theorem $2.1$]{eqchow}. Since $X_{\tau}$ is contained in the open substack $X' \subset X$ identified in Lemma \ref{stratadecomp} over which $\text{Star}(\tau)$ is finite \'etale, by Lemma \ref{stratadecomp}, if the theorem was true after replacing $Y$ and $Y_{\gamma}$ with $Y \times_{X} \mathcal{A}_{\text{Star}(\tau)}$ and $Y_{\gamma} \times_{X_{\tau}} (\mathcal{A}_{\text{Star}(\tau)}\times_X X_{\tau})$ respectively, then the desired equation holds as in the proof of Lemma \ref{tbund1}. It therefore suffices to assume $X = \text{Star}(\tau)$.

 We additionally reduce to the case that $P$ is a homogeneous polynomial of degree $k$ in piecewise linear functions. To establish this reduction, let $p: \widetilde{Y} \rightarrow Y$ be a log \'etale modification such that $\widetilde{Y}$ is a simple log algebraic stack possessing a piecewise polynomial function $Q_{\gamma}$ such that $p_*Q_{\gamma} \cap [\widetilde{Y}] = [Y_{\gamma}]$. The existence of such a modification is given by \cite[Lemma $3.6$]{intDR}. Since $\widetilde{Y}$ is simple, there exists a factorization $\prod_i P_i = p^*P$. Moreover, we have $Q_{\gamma}= a_{\gamma}\sum_{\gamma'} Q_{\gamma'}$ for some scalar $a_{\gamma} \in \mathbb{Q}$ with $\gamma' \in \Sigma(\widetilde{Y})$ a cone subdividing $\gamma$ with $\dim\gamma' = \dim\gamma$. If the theorem is known when $P$ factors as a product of piecewise linear functions, we have:

\begin{equation}
\begin{split}
\pi_*(P\cap [Y_{\gamma}]) &= \pi_*p_*(p^*P Q_{\gamma} [\widetilde{Y}]) \\
&= \pi_*p_*\sum_{\gamma'} p^*Pa_{\gamma}Q_{\gamma'}[\tilde{Y}_{\gamma'}] \\
&= \sum_{\tau\subset \tau'} \deg_{\tau}(Pa_{\gamma}\sum_{\gamma'}Q_{\gamma'})[X_{\tau'}]\\
&= \sum_{\tau\subset \tau'}\deg_{\tau}(PQ_{\gamma})[X_{\tau'}].
\end{split}
\end{equation}
We therefore may assume $P$ is a polynomial in piecewise linear functions. By passing to a further subdivision, the argument above also ensures we may assume there is a map $Y \rightarrow \mathcal{A}_{\tilde{\tau}}$ factoring the composition $Y \rightarrow X \rightarrow \mathcal{A}_{\tau}$, with the latter map induced by the extension map $PL(\tau) \rightarrow PL(\Sigma)$. 

We prove the theorem by induction on $k = \deg(P)$. When $k = 0$, if $\dim X_{\tau} < \dim Y_{\gamma}$ hence $\dim\gamma < \dim \tau$, then $\pi_*([Y_{\gamma}]) = 0 = \deg_{\tau}(Q_{\gamma})[X_{\tau}]$. When $\dim X_{\tau} = \dim Y_{\gamma}$, both $X_{\tau}$ and $Y_{\gamma}$ have dense open idealized Artin cones $\mathcal{A}_{\tau,\tau}$ and $\mathcal{A}_{\gamma,\gamma}$ respectively. By \cite[Lemma $4.5$]{pbirinv} we have $\pi_*([\mathcal{A}_{\gamma,\gamma}]) = \frac{1}{|\coker(\gamma_{\NN}^{\gp} \rightarrow \tau_{\NN}^{\gp})|}[\mathcal{A}_{\tau,\tau}]$. Since any geometric generic point of $X_{\tau}$ factors through the open substack $\mathcal{A}_{\tau,\tau} \subset X_{\tau}$, and $\mathcal{A}_{\tau,\tau}\times_{X_{\tau}} Y_{\gamma} = \mathcal{A}_{\gamma,\gamma}$, we have $\pi_*([Y_{\gamma}]) = \frac{1}{|\coker(\gamma_{\NN}^{\gp} \rightarrow \tau_{\NN}^{\gp})|}[X_{\tau}]$. On the other hand, it follows by construction of $Q_{\gamma}$ that $\deg_{\tau}(Q_{\gamma}) = \frac{1}{|\coker(\gamma_{\NN}^{\gp} \rightarrow \tau_{\NN}^{\gp})|}$. Indeed, this is by Example \ref{simpdeg} if $\gamma$ is simplicial, and in general, after passing to a simplicial subdivision $\widetilde{\gamma} \rightarrow \gamma$, we have $Q_{\gamma} = \frac{\sum_{\gamma'} Q_{\gamma'}}{\sum_{\gamma'} \deg_{\gamma}(Q_{\gamma'})}$ for $\gamma'$ maximal cones of $\widetilde{\gamma}$, and we have: 

$$\deg_{\tau}(Q_{\gamma}) = \frac{ \sum_{\gamma'} \deg_{\tau}(Q_{\gamma'})}{\sum_{\gamma'}\deg_{\gamma}( Q_{\gamma'})} = \frac{\sum_{\gamma'} \frac{1}{|\coker(\gamma_{\NN}^{\gp} \rightarrow \tau_{\NN}^{\gp})|}}{\sum_{\gamma'} 1} = \frac{1}{|\coker(\gamma_{\NN}^{\gp} \rightarrow \tau_{\NN}^{\gp})|} $$

Now assume that we have proven the Lemma whenever $k'<k$. Since $X = \text{Star}(\tau)$, every PL function on $\tilde{\tau}$ extends to a PL function on $Y$ respecting the system of extensions on $\tau$. Note additionally that we have $$Sym^k(PL(\widetilde{\text{Star}(\tau)}) = \bigoplus_{1=0}^k Sym^i(PL(\tilde{\tau}))\otimes Sym^{k-i}(\text{ker}(res_{\tilde{\tau}})),$$ with $Sym^i(PL(\tilde{\tau}))$ given by polynomials in extensions of piecewise linear functions on $\tilde{\tau}$. By linearity of both sides of the desired equation, it suffices to demonstrate the desired equation for $P$ given by a polynomial in extensions, and for $P$ divisible by a piecewise linear function $L$ with $L|_{\tilde{\tau}} = 0$. 

We first suppose that $P$ is a polynomial in extensions. To show $\pi_*(P\cap [Y_{\gamma}])$ satisfies the desired equation in this case, observe that the second part of Proposition \ref{degprop} ensures that $\deg_{\tau'}(Q_{\gamma}P) = 0$ for $\tau' \in \Sigma(Y)$ containing $\tau$ as a proper face. Hence, the desired equation reduces to:
\[\pi_*(P [Y_{\gamma}]) = \deg_{\tau}(PQ_{\gamma})[X_{\tau}].\]

To verify this equation, we note that by Lemma \ref{stratadecomp}, there is a natural morphism of log algebraic stacks $X \rightarrow \mathcal{A}_{\Sigma_{\tau}} \times \mathcal{A}_{\tau}$ inducing a map of idealized log stacks $X_{\tau} \rightarrow \mathcal{A}_{\Sigma_{\tau}} \times \mathcal{A}_{\tau,\tau}$ which is an isomorphism of algebraic stacks. Thus, in order to compute $\pi_*(P [Y_{\gamma}])$, we may replace $X$ and $X_{\tau}$ with $\mathcal{A}_{\Sigma_{\tau}}\times \mathcal{A}_{\tau}$ and $\mathcal{A}_{\Sigma_{\tau}}\times \mathcal{A}_{\tau,\tau}$ respectively.  With this replacement, the map $Y_{\gamma} \rightarrow X_{\tau}$ factors through the idealized Artin fan $X_{\tilde{\tau}}:=\mathcal{A}_{\Sigma_{\tau}}\times \mathcal{A}_{\widetilde{\tau},\tau}$. Since every PL function on $\tilde{\tau}$ extends to a PL map on $Y$ respecting the system of extensions for $\tau$, the map $Y \rightarrow X$ also factors through $p: Y \rightarrow X' := \mathcal{A}_{\Sigma_{\tau}} \times \mathcal{A}_{\widetilde{\tau}}$. Since the second part of Lemma \ref{stratadecomp} ensures the piecewise polynomial $P \in PP(\Sigma(Y))$ is pulled back along $p$, after letting $\pi': X' \rightarrow X$ be the log alteration factoring $\pi = \pi'\circ p$, the following series of equations now follows:

\begin{equation}
\begin{split}
\pi_*(P[Y_{\gamma}]) &= \pi_*(PQ_{\gamma}[Y])\\
 &= \pi' p_*(p^{*}(PQ_{\gamma})[Y])\\
&= \pi'_*(PQ_{\gamma}[X']))\\
&= \deg_{\tau}(PQ_{\gamma}) [X_{\tau}].
\end{split}
\end{equation}
The third equality above follows from the projection formula and the last equality follows from Lemma \ref{tbund1}. The main theorem thus holds for $P$ a polynomial in extensions.


We now assume that $P$ is divisible by a piecewise linear function $L$ with $L|_{\tilde{\tau}} = 0$, hence in particular $L|_{\gamma} = 0$. By Lemma \ref{linlem}, we have $\pi_*(P'[Y_{\gamma}]) = \sum_{\gamma\subset \gamma'} a_{\gamma'}\pi_*(P'_{\tau}[Y_{\gamma'}])$ with $P'_{\tau}$ a piecewise polynomial of degree $k-1$ and the sum is over $\dim \gamma + 1$ cones $\gamma'$. Note additionally that $Q_{\gamma'} = L_{\gamma'}Q_{\gamma},$ with $L_{\gamma'}$ the pullback of the primitive non-negative integral function on the quotient $\gamma^{'\gp}_{\NN}/\gamma_{\NN}^{\gp} \cong \ZZ$, and $L|_{\gamma'} = a_{\gamma'}L_{\gamma'}$. We therefore have: 
$$\sum_{\gamma\subset \gamma'}a_{\gamma'}Q_{\gamma'} = \sum_{\gamma\subset \gamma'}a_{\gamma'}L_{\gamma'}Q_{\gamma} = LQ_{\gamma},$$
where the last equality holds since since it holds upon restriction to each cone $\gamma'$. Thus, by the induction hypothesis on the degree $k$ of $P$, we have:
\begin{equation}\label{pushp'}
\begin{split}
\pi_*(P'[Y_{\gamma}]) &= \pi_*(P'_{\tau}\sum_{\gamma \subset \gamma'} a_{\gamma'}[Y_{\gamma'}]) \\
&= \sum_{\gamma\subset \gamma'}\sum_{\tau' \subset \tau''} \deg_{\tau''}(a_{\gamma'}P'_{\tau}Q_{\gamma'})[X_{\tau'}]\\
&= \sum_{\tau\subset \tau' \subset \tau''} \deg_{\tau''}(P'_{\tau}\sum_{\gamma \subset \gamma'}a_{\gamma'}Q_{\gamma'})[X_{\tau'}]\\
&=  \sum_{\tau \subset \tau'} \deg_{\tau'}(P'_{\tau}LQ_{\gamma})[X_{\tau'}]\\
&=  \sum_{\tau \subset \tau'} \deg_{\tau'}(P'Q_{\gamma})[X_{\tau'}].
\end{split}
\end{equation}
Equation \ref{pushp'} verifies the Theorem in the final case, and the desired equation follows by linearity.

\end{proof}


\begin{corollary}\label{pushvpull}
Suppose we have an algebraic stack $\mathfrak{M}$ with log structure and a triple of strict morphisms $\mathscr{M}_{\tau} \rightarrow \mathfrak{M}_{\tau} \rightarrow \mathcal{A}_{\Sigma,\tau} \rightarrow \mathcal{A}_{\Sigma}$ with $\mathcal{A}_{\Sigma}$ an Artin fan and a cone $\tau \in \Sigma$. Suppose further that $\mathfrak{M}_{\tau} \rightarrow \mathcal{A}_{\Sigma,\tau}$ is smooth and $\epsilon_{\tau}: \mathscr{M}_{\tau} \rightarrow \mathfrak{M}_{\tau}$ is equipped with a perfect obstruction theory. Let $\mathcal{A}_{\widetilde{\Sigma}} \rightarrow \mathcal{A}_{\Sigma}$ be a log \'etale modification induced by a subdivision and lattice coarsening. For a strata closure $\mathcal{A}_{\widetilde{\Sigma},\gamma} \rightarrow \mathcal{A}_{\widetilde{\Sigma}}$ such that we have an induced map $\mathcal{A}_{\widetilde{\Sigma},\gamma} \rightarrow \mathcal{A}_{\Sigma,\tau}$, consider the pullback squares:

\begin{equation}
\begin{tikzcd}
\mathscr{M}_{\gamma}\arrow[r] \arrow{d} & \arrow{d}\mathfrak{M}_{\gamma} \arrow[r]\arrow[d] & \mathcal{A}_{\widetilde{\Sigma},\gamma}\arrow[d,"p"] \arrow{r}& \mathcal{A}_{\widetilde{\Sigma}}\arrow{d}\\
\mathscr{M}_{\tau}\arrow[r,"\epsilon_{\tau}"]& \mathfrak{M}_{\tau}\arrow[r,"i"] & \mathcal{A}_{\Sigma,\tau} \arrow{r} & \mathcal{A}_{\Sigma}.
\end{tikzcd}
\end{equation}
Then after equipping $\mathfrak{M}_{\gamma} \rightarrow \mathfrak{M}_{\tau}$ with the pulled back obstruction theory, and for any operational Chow class $P \in CH^{\text{op}}(\mathcal{A}_{\widetilde{\Sigma}})$, for $\gamma' \in \widetilde{\Sigma}$ a strict map from a cone $\gamma'$ and $\tau'\rightarrow \Sigma$ the smallest strict map from a cone with $\tau'\times_{\Sigma} \gamma' = \gamma'$, we have:

\begin{equation}
p_*P\epsilon_{\tau}^!i^*[\mathcal{A}_{\widetilde{\Sigma},\gamma'}] = \sum_{\tau' \subset \tau''}\deg_{\tau''}(P)\epsilon_{\tau}^!i^*[\mathcal{A}_{\Sigma,\tau''}].
\end{equation}
\end{corollary}

\begin{proof}

By an application Costello-Herr-Wise Theorem \cite{pusherr} on proper pushforward and commutativity of virtual pullback with flat pullback on the first line, Theorem \ref{normalform} on the second line, and finally commutativity of virtual pullback with flat pullback on the third line, we have:

\begin{equation}
\begin{split}
 p_*P\epsilon_{\tau}^!i^*[\mathcal{A}_{\widetilde{\Sigma},\gamma'}] &=\epsilon^!_{\tau}i^*p_*P[\mathcal{A}_{\widetilde{\Sigma},\gamma'}]\\
&=\epsilon^!_{\tau}i^*\sum_{\tau' \subset\tau''} \deg_{\tau''}(P)[\mathcal{A}_{\Sigma,\tau''}]\\
&= \sum_{\tau'\subset \tau''}\deg_{\tau''}(P)\epsilon_{\tau}^!i^*[\mathcal{A}_{\Sigma,\tau''}].
\end{split}
\end{equation}
\end{proof}

To set notation, for tropical types $\pmb\tau$ and $\pmb\tau'$ with $\pmb\tau'$ marked by $\pmb\tau$, we let $\mathscr{M}_{\pmb\tau'} \subset \mathscr{M}(X,\pmb\tau)$ be the moduli stack of punctured log maps marked by $\pmb\tau$ which admit a marking by $\pmb\tau'$. There is a natural forgetful map $\mathscr{M}_{\pmb\tau'} \rightarrow \mathfrak{M}_{\pmb\tau'} \subset \mathfrak{M}(\mathcal{X},\pmb\tau)$ for the analogously defined substack $\mathfrak{M}_{\pmb\tau'}$, which posess a pulled back perfect obstruction theory from $\mathscr{M}(X,\pmb\tau) \rightarrow \mathfrak{M}(\mathcal{X},\pmb\tau)$. There exists an \'etale map $\mathfrak{M}(\mathcal{X},\pmb\tau') \rightarrow \mathfrak{M}_{\pmb\tau'}$, and in particular, $\mathfrak{M}_{\pmb\tau'}$ is equidimensional. We define $[\mathscr{M}_{\pmb\tau'}]^{\vir}$ to be the virtual pullback of $[\mathfrak{M}_{\pmb\tau'}]$.
We now apply Corollary \ref{pushvpull} together with Theorem \ref{decthm} to deduce Theorem \ref{mthm1}:

\begin{proof}[Proof of Theorem \ref{mthm1}]
By the log enhancement of the evaluation map provided by Lemma \ref{logenh}, we have $\ev^*(P)$ is the operational Chow class associated with the piecewise polynomial $\ev^{\trop,*}(P)$. Since $\mathscr{M}(\widetilde{X},\pmb\gamma)\rightarrow \mathfrak{M}_{\gamma\rightarrow \tau}$ is strict, $\ev^*(P)$ is also a Chow class pulled back from $\mathfrak{M}_{\gamma\rightarrow \tau}$. Moreover, the log algebraic stack $\mathfrak{M}_{\gamma\rightarrow \tau}$ is log \'etale over $\mathfrak{M}_{\tau}$, realized via the following Cartesian diagram with vertical morphisms log \'etale:

\begin{equation}
\begin{tikzcd}
\mathfrak{M}_{\gamma\rightarrow \tau} \arrow{r}\arrow{d} & \mathcal{A}_{\gamma\rightarrow \tau}\arrow{d}\\
\mathfrak{M}_{\tau} \arrow{r} & \mathcal{A}_{\mathfrak{M}_{\tau}}.
\end{tikzcd}
\end{equation}

As before, strictness of the top horizontal morphism again ensures that $P \in A^*(\mathfrak{M}_{\gamma\rightarrow \tau})$ is pulled back from $A^*(\mathcal{A}_{\gamma\rightarrow \tau})$. Thus, we use the cartesian square of Theorem \ref{decthm} as input to Corollary \ref{pushvpull} to produce the relation:

\begin{equation}\label{pusheq1}
p_*(P\cap[\mathscr{M}(\widetilde{X},\pmb\gamma)]^{\vir}) = \sum_{\tau\subset \tau'} \deg_{\tau'}(PQ_{\gamma})\epsilon_{\tau}^!i^*[\overline{\mathcal{A}}_{\tau'}].
\end{equation}

Note now that the map $i_{\tau,\tau'}: \mathfrak{M}(\mathcal{X},\pmb\tau') \rightarrow \mathfrak{M}(\mathcal{X},\pmb\tau)$ is \'etale over the fiber product $\mathfrak{M}(\mathcal{X},\pmb\tau)\times_{\overline{\mathcal{A}}_{\tau}} \overline{\mathcal{A}}_{\tau'}$ of degree $|Aut(\tau'/\tau)|$, hence $i_{\tau,\tau'*}[\mathfrak{M}_{\pmb\tau'}] = |Aut(\pmb\tau'/\pmb\tau)|i^*[\overline{\mathcal{A}}_{\tau'}]$. Since the obstruction theory for $\mathscr{M}_{\pmb\tau'} \rightarrow \mathfrak{M}_{\pmb\tau'}$ is pulled back from $\mathscr{M}_{\pmb\tau} \rightarrow \mathfrak{M}_{\pmb\tau}$, it follows by the Costello-Herr-Wise pushforward formula \cite{pusherr} that $$\frac{1}{|Aut(\pmb\tau'/\pmb\tau)|}i_{\tau,\tau'*}[\mathscr{M}(X,\pmb\tau')]^{\vir} =\epsilon^!_{\tau}i^*[\overline{\mathcal{A}}_{\tau'}].$$ Substituting this expression in for the relevant term in the sum on the righthand side of Equation \ref{pusheq1} yields the desired expression assuming that $\deg_{\tau'}(PQ_{\gamma})$ is a tautological operational Chow class. 

To establish this final fact, note that $\deg_{\tau'}(PQ_{\gamma})$ is a product of linear functions on $\tau'$. The $\mathbb{Q}$-vector space of $\mathbb{Q}$-linear functions is spanned by the pullback of edge lengths along $\tau' \rightarrow \Sigma(\mathfrak{M}_{g,n})$ and the pullback of piecewise linear functions $L$ on $\Sigma(X)$ along $\ev_v: \tau' \rightarrow \Sigma(X)$ associated with a vertex $v \in V(G_{\tau'})$. These linear functions on $\tau'$ are extended to piecewise linear functions on $\Sigma(\mathscr{M}(X,\pmb\tau'))$ by pulling back edge lengths along $\Sigma(\mathscr{M}(X,\pmb\tau')) \rightarrow \Sigma(\mathfrak{M}_{G_{\tau'}})$, with $\mathfrak{M}_{G_{\tau'}}$ the moduli space of prestable curves with dual graph marked by $G_{\tau'}$, and the pullback of piecewise linear functions $L$ on $\Sigma(X)$ along $\ev_v: \Sigma(\mathscr{M}(X,\pmb\tau')) \rightarrow \Sigma(X)$ along a family of vertices contained in a family of half edges marked by a half edge containing $v$. 

Since the morphism $\mathscr{M}(X,\pmb\tau') \rightarrow \mathfrak{M}_{G_{\tau'}}$ has a log enhancement, the former piecewise linear functions induce operational Chow classes pulled back from $\mathfrak{M}_{G_{\tau'}}$ are given by $-(\psi_{v,e} + \psi_{v',e})$ for $v,v' \in e$ the two vertices contained in $e$, and in particular are tautological classes. For the PL functions pulled back from $\Sigma(X)$, note that the map $f: C \rightarrow X$ has a log enhancement, so $f^*\mathcal{O}(L) = \mathcal{O}(\Sigma(f)^*(L))$. Observe that the piecewise linear function $\Sigma(f)^*(L)$ has a decomposition $\Sigma(\pi)^*(L') + u(e)\rho_l$, where $L' \in PL(\Sigma(\mathscr{M}(X,\pmb\tau')))$ is the piecewise linear pulled back along $\ev_{v}: \Sigma(\mathscr{M}(X,\pmb\tau')) \rightarrow \Sigma(X)$ and $\rho_l$ the piecewise linear function which has slope one along the leg $l$, and slope zero on all other legs or edges. Pulling back along the section $q_e$, taking Chern classes and subtracting $u(e)c_1(q_e^*\mathcal{O}(\rho_l))$ from both sides of the resulting equation gives
 $$c_1(q_e^*f^*\mathcal{O}(L)) - u(e)c_1(q_e^*\mathcal{O}(\rho_l)) = q_e^*(c_1(f^*\mathcal{O}(L)))  + u(e)\psi_e =  c_1(\mathcal{O}(L')).$$
Since both of terms on the lefthand side of the previous equation are tautological classes, the desired expression follows.

\end{proof}

We now show how Theorem \ref{mthm1} allows us to express the punctured log Gromov-Witten theory of a log \'etale modification $\pi: \widetilde{X} \rightarrow X$ of a pair $(X,D)$ with simple normal crossings in terms of the punctured log Gromov-Witten theory of $X$:

\begin{corollary}\label{reduce}
With notation as above, for any decorated tropical type $\pmb\gamma$ of log stable map to $\widetilde{X}$ which is a lift of a decorated tropical type $\pmb\tau$ giving a map $\st: \mathscr{M}(\widetilde{X},\pmb\gamma) \rightarrow \mathscr{M}(X,\pmb\tau)$, any decorated tropical type $\pmb\gamma'$ marked by $\pmb\gamma$ and any class $\prod_{l \in L(G_{\gamma'})}\alpha_l \in A^*(\prod_{l \in L(G_{\gamma'})} \widetilde{X}_{\pmb\sigma(l)})$, the class $\st_*(\ev^*\alpha\cap[\mathscr{M}_{\pmb\gamma'}]^{\vir})$ can be explicitly written in terms of punctured log Gromov-Witten classes on $X$. 
\end{corollary}

\begin{proof}
Observe that it suffices to prove the corollary when $\widetilde{X}$ is smooth with simple normal crossings boundary and $\widetilde{X} \rightarrow X$ is the composite of a sequence of blowups along smooth centers. Indeed, such blowups are cofinal among toroidal modifications of $X$, hence there exists a dominating log modification $\widetilde{X}' \rightarrow \widetilde{X}$ such that the induced map $\widetilde{X}' \rightarrow X$ is a given by a sequence of blowups on smooth centers. Now letting $\pmb\omega$ a decorated tropical lift of $\pmb\gamma$, Corollary \ref{puncmcr1} ensures that for any type $\pmb\omega'$ which is a lift of $\pmb\gamma'$ with $\dim \omega' = \dim\gamma'$, stabilization induces a map $\st: \mathscr{M}(\widetilde{X}',\omega) \rightarrow \mathscr{M}(\widetilde{X},\gamma)$ such that $$\st_*([\mathscr{M}_{\pmb\omega'}]^{\vir}) = q[\mathscr{M}_{\pmb\gamma'}]^{\vir}$$ for some $q \in \mathbb{Q}\setminus \{0\}$. By pulling back insertions to $\widetilde{X}'$, the projection formula implies the necessary reduction. Additionally, by repeatedly applying the lemma in the case of a single blowup, we reduce to the case of a single blowup along a smooth stratum $V \subset X$.



Letting $\pi: \widetilde{\Ev} = \prod_{l \in L(G_{\tau})} \widetilde{X}_{\pmb\sigma(l)} \rightarrow \Ev = \prod_{l \in L(G_{\tau})} X_{\pmb\sigma}(L)$ be the induced map on evaluation spaces, observe that $\pi$ is a product of blowups or projective bundles, depending on whether $X_{\pmb\sigma(l)}$ is contained in $V$. We define $\widetilde{V}_{\pmb\sigma(l)} \subset \widetilde{X}_{\pmb\sigma(l)}$ to be the exceptional divisor if $X_{\pmb\sigma(l)} \not\subset V$, and otherwise $\widetilde{V}_{\pmb\sigma(l)} = \widetilde{X}_{\pmb\sigma(l)}$. For a leg $l$ such that $X_{\pmb\sigma(l)}\not\subset V$, we have by \cite[Proposition $6.7(e)$]{Fulint} that the natural map $A_*(X_{\pmb\sigma(l)})\oplus A_*(\widetilde{V}_{\pmb\sigma(l)}) \rightarrow A_*(\widetilde{X}_{\pmb\sigma(l)}) = A^{dim\text{ }X-*}(\widetilde{X}_{\pmb\sigma(l)})$ is surjective. Corollary \ref{puncmcr1} together with the projection formula allows us to reduce to the case that $\alpha$ is in the image of 
$$\tilde{i}: \prod_{l \in L(G_{\tau})} A^*(\widetilde{V}_{\pmb\sigma(l)}) \rightarrow \prod_{l \in L(G_{\tau})} A^*(\widetilde{X}_{\pmb\sigma(l)}).$$
We prove the corollary with this restriction by induction on the number of legs $l$ of a type $\pmb\gamma'$ marked by $\pmb\gamma$ for which $\widetilde{V}_{\pmb\sigma(l)} \not= \widetilde{X}_{\pmb\sigma(l)}$, which we denote by $n_{\gamma'}$.

In the base case $n_{\gamma'} = 0$, letting $\pmb\tau'$ be the unique decorated tropical type which $\pmb\gamma'$ is a tropical lift of, observe $\widetilde{\Ev}_{\gamma'} \rightarrow \Ev_{\tau'}$ is a product of split projective bundles. It follows that we can use \cite[Theorem $3.3(b)$]{Fulint} to express the class $\alpha$ as:
\[\alpha = \prod_{l \in L(G_{\gamma})}\sum_{j=0}^{\text{codim}(V)-1} \ev_{l}^*c_1(\mathcal{O}_{\widetilde{X}_{\sigma(l)}}(1))^jp^*\alpha_j,\]
for classes $\alpha_j \in A^*(\Ev_{\tau'})$ and $\mathcal{O}_{\widetilde{X}_{\sigma(l)}}(1)$ the dual of the tautological class for the projective bundle $\widetilde{X}_{\pmb\sigma(l)} \rightarrow X_{\pmb\sigma(l)}$. The class $c_1(\mathcal{O}_{\widetilde{X}_{\pmb\sigma(l)}}(1))$ is pulled back from a cohomology class on $\mathcal{A}_{\widetilde{X}_{\pmb\sigma(l)}}$ associated with a piecewise linear function $P$ on $\Sigma(\widetilde{X}_{\pmb\sigma(l)})$. Theorem \ref{mthm1} now allows us to conclude the corollary in the base case of the induction.


Now assume we know the corollary holds for $\pmb\gamma'$ such that $n_{\gamma'} <k$ for some $k > 0$, and suppose we have a type $\pmb\gamma'$ such that $n_{\gamma'} = k$. Let $l \in L(G_{\gamma'})$ be a leg such that $\widetilde{V}_{\pmb\sigma(l)} \not= X_{\pmb\sigma(l)}$, and consider the associated divisor $\widetilde{\Ev}_l = \widetilde{\Ev} \times_{X_{\pmb\sigma(l)}} \widetilde{V}_{\pmb\sigma(l)}$ with closed immersion $\tilde{i}_l: \widetilde{\Ev}_l \rightarrow \widetilde{\Ev}$. Since $\tilde{i}_l$ is the inclusion of a stratum, the morphism is the pullback of the inclusion $\mathcal{A}_{\Sigma(\widetilde{\Ev}),\rho} \rightarrow \mathcal{A}_{\Sigma(\widetilde{\Ev})}$ of idealized Artin fans associated with a cone $\rho \in \Sigma^{[1]}(\widetilde{\Ev})$. Consider now the cartesian diagram:

\begin{equation}
\begin{tikzcd}
\mathfrak{M}^{\ev}_{\pmb\gamma'}\times_{\widetilde{\Ev}} \widetilde{\Ev}_{l} \arrow{r}\arrow{d} & \widetilde{\Ev}_{l} \arrow{d}\arrow{r}& \mathcal{A}_{\Sigma(\widetilde{\Ev}),\rho}\arrow{d}\\
\mathfrak{M}^{\ev}_{\pmb\gamma'} \arrow{r} & \widetilde{\Ev} \arrow{r} & \mathcal{A}_{\Sigma(\widetilde{\Ev})}.
\end{tikzcd}
\end{equation}

Since $\mathfrak{M}^{\ev}_{\pmb\gamma'} \rightarrow \widetilde{\Ev}$ is log smooth, $\mathfrak{M}^{\ev}_{\pmb\gamma'}\times_{\widetilde{\Ev}} \widetilde{\Ev}_{l}  \rightarrow \widetilde{\Ev}_{l}$ is log smooth. Since $\widetilde{\Ev}_{l}$ is irreducible representing $P_\rho \cap [\widetilde{\Ev}]$ for a Cartier divisor $P_\rho$ pulled back from the Artin fan, it follows by Lemma \ref{linlem} that $\mathfrak{M}^{\ev}_{\pmb\gamma'}\times_{\widetilde{\Ev}} \widetilde{\Ev}_{l}$ is equidimensional with fundamental class computed as:

$$[\mathfrak{M}^{\ev}_{\pmb\gamma}\times_{\widetilde{\Ev}} \widetilde{\Ev}_{l}] = \tilde{i}_l^!([\mathfrak{M}^{\ev}_{\pmb\gamma'}]) = \sum_{\pmb\gamma' \subset \pmb\gamma''} a_{\pmb\gamma''}[\mathfrak{M}^{\ev}_{\pmb\gamma''}],$$
with $a_{\pmb\gamma'} = |\coker(\Sigma(\ev)^*P_\rho: \gamma^{''\gp}_{\NN}/\gamma^{'\gp}_{\NN} \rightarrow \ZZ)|$. By commutativity of bivariant products, the previous equation yields:

$$\tilde{i}^!_l[\mathscr{M}_{\pmb\gamma'}]^{\vir} = \sum_{\pmb\gamma' \subset \pmb\gamma''} a_{\pmb\gamma''}[\mathscr{M}_{\pmb\gamma''}]^{\vir}.$$
Note that $\tilde{i}:\prod_{l \in L(G_{\tau})} \widetilde{V}_{\pmb\sigma(l)} \rightarrow \widetilde{\Ev}$ factors through $\tilde{i}_l$, and we let $\tilde{i} = \tilde{i}_l\tilde{i}'$ denote the factorization. By commutativity of proper pushforward with bivariant classes and the definition of the pushforward of bivariant classes under regular emeddings, the class $\ev^*\tilde{i}_*\alpha\cap[\mathfrak{M}^{\ev}_{\pmb\gamma}]$ is expressed as:
\begin{equation}\label{pusheq}
\begin{split}
(\ev^*\tilde{i}_*\alpha)\cap[\mathscr{M}_{\pmb\gamma'}]^{\vir} &= (\ev^*\tilde{i}_{l*}\tilde{i}'_*\alpha)\cap[\mathscr{M}_{\pmb\gamma'}]^{\vir}\\
&= (\tilde{i}_{l*}\ev^*\tilde{i}'_*\alpha)\cap[\mathscr{M}_{\pmb\gamma'}]^{\vir} \\
&= \tilde{i}_{l*}(\ev^*\tilde{i}'_*\alpha \tilde{i}_{l}^!([\mathscr{M}_{\pmb\gamma'}]^{\vir}))\\
&= \tilde{i}_{l*}(\sum_{\pmb\gamma' \subset \pmb\gamma''} a_{\pmb\gamma''}\ev^*\tilde{i}'_*\alpha\cap[\mathscr{M}_{\pmb\gamma''}]^{\vir})\\
\end{split}
\end{equation}

The pushforward of the previous class along $\text{st}: \mathscr{M}(\widetilde{X},\pmb\gamma) \rightarrow \mathscr{M}(X,\pmb\tau)$ can be expressed in terms of punctured log Gromov-Witten classes of $X$ by induction.

\end{proof}

\section{Log-Orbifold correspondence for canonical wall structures and generalized broken line expansions}

In this section, we give our first application of Theorem \ref{mthm1}, demonstrating the invariants used to define the canonical wall structure of Gross and Seibert may also be computed via log Gromov-Witten invariants with insertions coming from an alternate log \'etale model, see Equation \ref{blinet} and \ref{wallt}. This presentation in turn will give Corollary \ref{blorb}, providing an expression of the canonical wall structure in terms of orbifold Gromov-Witten invariants, aspects of which were first considered by You in \cite{orbBL}.

Let $(X,D)$ be a log Calabi Yau in the sense of \cite{int_mirror}, i.e. we must have either $c_1(T^{log}_X)$ is nef or antinef, or we may write $c_1(\Omega_X(log\text{ }D)) = K_X  + D =_\mathbb{Q} \sum_i a_iD_i$ with $a_i\ge 0$. In the latter case, the choice of representative of $K_X + D$ determines a subcomplex $\mathscr{P} \subset \Sigma(X)$ called the essential skeleton, with $1$ dimensional cones associated to divisors $D_i$ with $a_i = 0$. In the more restricted setting of \cite{scatt}, Gross and Siebert consider the following collection of tropical types which they call wall types and broken line types respectively. We include a generalization of broken line types for general log Calabi-Yau pair which do not necessarily satisfy the assumptions in \cite{scatt}:

\begin{definition}
For a log Calabi-Yau pair $(X,D)$ satisfying the assumptions of \cite{scatt}, a wall type $\tau = (G,\sigma,u)$ is a type of punctured log map to a target $X$ satisfying:
\begin{enumerate}
\item $G$ is a genus $0$ graph with $L(G) = \{L_{out}\}$ with $\pmb\sigma(L_{out}) \in \mathscr{P}$, and $u_\tau := u(L_{out}) \not= 0$. 
\item $\tau$ is realizable.
\item Let $h: \Gamma(G,l) \rightarrow \Sigma(X)$ be the corresponding universal family of tropical maps, and $\tau_{out}$ the cone corresponding to $L_{out}$. Then dim $\tau = \dim B-2$, and dim $h(\tau_{out}) = \dim B-1$. 
\end{enumerate}
\end{definition}

\begin{definition}\label{bltype}

For a log Calabi-Yau pair in the sense of \cite{int_mirror}, a broken-line type $\tau = (G,\sigma,u)$ is a type of punctured log map to a target $X$ satisfying:
\begin{enumerate}
\item $G$ is a genus $0$ graph with $L(G) = \{L_{\out},L_{\text{in}}\}$ with $\pmb\sigma(L_{\out}),\pmb\sigma(L_{\text{in}}) \in \mathscr{P}$, $u(L_{\out}) \not= 0$ and $u(L_{\text{in}}) \in B(\NN)$. 
\item $\tau$ is realizable and when $D$ contains a zero dimensional stratum is balanced in the sense of \cite{scatt}.
\item Let $h: \Gamma(G,l) \rightarrow \Sigma(X)$ be the corresponding universal family of tropical maps, and $\tau_{out}$ the cone corresponding to $L_{out}$. Then dim $\tau = \dim \pmb\sigma(L_{\out})-1$, and dim $h(\tau_{\out}) = \dim \pmb\sigma(L_{\out})$. 
\end{enumerate}
\end{definition}

The virtual dimension of the moduli stack $\mathscr{M}(X,\pmb\tau)$ for any decorated wall type or broken-line type $\pmb\tau$ is $\dim X - \dim B$. When $\pmb\tau$ is a broken line type, we define
 $$N_{\pmb\tau} = \frac{\deg(\ev_{x_{\out}}^*([\pt]_{X_{\pmb\sigma(L_{\out})}})[\mathscr{M}(X,\pmb\tau)]^{vir})}{|Aut(\tau'/\beta)|},$$
where $\beta$ is the potentially unrealizable combinatorial type with total curve class $\textbf{A}$ with a single vertex mapping to $\{0\} \in \Sigma(X)$ with two legs of contact orders $u(L_{\operatorname{in}})$ and $u(L_{\out})$.

These invariants are used in \cite{scatt} to construct broken line expansions of theta functions $\vartheta_{p,x}$ and wall functions $\mathfrak{p}_{\pmb\tau}$, with $\pmb\tau$ a wall type in the latter case. Both will be given by elements $\kk[Q][\Lambda_{x}]$, with $\kk[Q] = \kk[NE(X)]/I$ for some monomial ideal I such that all but finitely many curve classes are contained in $I$, and $x \in h(\text{int}(\tau_{\out}))$ for the wall function. 

For the broken line type, given $p \in B(\NN)$, $x \in \sigma' \in \Sigma(X)$ a general point with $\dim \sigma = \dim X$ and $\textbf{A} \in NE(X)$, we let $B(p,x,\textbf{A})$ be the collection of all decorated broken-line types $\pmb\tau$ of total degree $\textbf{A}$ with $u(L_{\operatorname{in}}) = p$ and $x \in \im(\tau_{\out})$. Letting $k_{\tau} = |\coker(\tau_{\out,\NN}^{\gp}\rightarrow \pmb\sigma(L_{\out})^{\gp}_{\NN})|$, these invariants are used to define the broken-line expansion:

\begin{equation}\label{blexp}
\vartheta_{p,x}= \sum_{\substack{\pmb\tau \in B(p,x)}} k_{\tau}N_{\pmb\tau}t^{\textbf{A}(\pmb\tau)}z^{-u(L_{\out})} \in \kk[Q][\sigma_{\NN}^{\gp}].
\end{equation}

For wall types, we let $x \in \mathscr{P}\setminus \mathscr{P}^{[n-2]}$ be a point of the essential skeleton not contained in the codimension $2$ skeleton of $\mathscr{P}$ and $p \in \sigma_{x,\NN}^{\gp}$ be an integral tangent vector of $x$ in the smallest cone $\sigma_x$ containing $x$. Let $W(p,x,\textbf{A})$ be the collection of all decorated wall types $\pmb\tau$ of total degree $\textbf{A}$ with $u(L_{\out}) = p$ with $x \in h(int(\tau_{\out})$. Letting $k_{\tau} = |\coker(\tau_{\out,\NN}^{\gp} \rightarrow h(\tau_{\out})^{\gp}_{\NN})|$, we will consider the following product of the associated wall functions:

\begin{equation}\label{wfunc}
\mathfrak{p}_{p,x,\textbf{A}} = \prod_{\pmb\tau \in W(p,x,\textbf{A})} \mathfrak{p}_{\pmb\tau} = \exp(\sum_{\pmb\tau \in W(p,x,\textbf{A})} k_{\tau}N_{\pmb\tau}t^{\textbf{A}}z^{-p})
\end{equation}

In the following two section, we will express both $\vartheta_{p,x}$ as well as $\mathfrak{p}_{p,x,\textbf{A}}$ in terms of orbifold Gromov-Witten theory of a root stack. We will then use this expression in the broken line case in Theorem \ref{blhom} to show that the broken line expansions yield ring homomorphisms outside the setting when $D$ is maximally degenerate. 

\subsection{Broken line expansion from orbifold Gromov-Witten theory of a root stack}

In this subsection, we establish the following using Theorem \ref{mthm1} as well as the log-orbifold correspondence of \cite{BNR2}:

\begin{proposition}\label{blorb}
In the setting of \cite{int_mirror}, the broken-line expansions $\vartheta_{p,x}$ are equal to the orbifold theta functions $\vartheta_{p}(x)$ on a suitable log \'etale model of $X$. 
\end{proposition}


%


\begin{proof}
For $\textbf{A} \in NE(X)$, we wish to compute $\vartheta_{p,x}[t^{\textbf{A}}]$. In order to compute the contributions to $\vartheta_{p,x}[t^{\textbf{A}}]$ coming from decorated broken line types with $u(L_{\out})$ fixed, we consider the moduli stack $\mathscr{M}(X,\beta')$, with $\beta'$ the non-realizable tropical type with two marked points with contact orders $p$ and $u(L_{out})$ and total curve class $\textbf{A}$. With $\beta'$ fixed, we choose an snc log \'etale model $\widetilde{X} \rightarrow X$ such that for all realizable tropical types $\pmb\tau$ marked by $\beta$, $h(\tau_{\out})$ is a union of cones of the subdivision and for which any lift $\beta$ of $\beta'$ with $u(L_{\out}) \in \pmb\sigma_x$ for a maximal cone $x \in \sigma_x \subset \Sigma(X)$ has at most one divisor $D_i$ such that $u(L_{\out})(D_i) < 0$ and one divisor $D_j$ for which $u(L_{\out})(D_j) > 0$. Such a subdivision can be found and is cofinal in the collection of all log blowups by \cite[Proposition $6.1$]{BNR2}. Moreover, by considering the refined virtual class of \cite{BNR2}, $\mathscr{M}(\widetilde{X},\beta)$ has refined virtual class of dimension $n-1$ and virtually equidimensional if $u(L_{out})$ is a non-negative contact order, and refined virtual dimension $n-2$ otherwise. In the latter case, we also have $\mathscr{M}(\widetilde{X},\beta)$ is virtually equidimensional. Indeed, the tropical moduli space $\omega$ of any virtual component of $\mathscr{M}(\widetilde{X},\beta)$ must be non-zero, hence $\vdim \mathscr{M}(\widetilde{X},\omega) < \dim n-1$. Since each component has dimension at least the refined dimension, we have $\vdim \mathscr{M}(\widetilde{X},\omega) = \dim n-2$ for all virtual components of $\mathscr{M}(\widetilde{X},\beta)$. We thus have the virtual decomposition:
 
 \begin{equation}\label{vdcmp}
 [\mathscr{M}(\widetilde{X},\beta)]^{\vir} = \sum_{\omega} \frac{a_{\omega}}{|Aut(\omega/\beta)|}[\mathscr{M}(\widetilde{X},\omega)]^{\vir}.
 \end{equation}
 To compute the multiplicity $a_{\omega}$, letting $k$ be the number of components of the boundary of $\widetilde{X}$, we consider the cone stack of tropical punctured maps $T$ to $\mathbb{R}_{\ge 0}^k$, with associated Artin fan $\mathcal{A}_T$, with the cone stack as defined following Definition $1.3$ in \cite{BNR2}. Then $\mathfrak{M}(\widetilde{X},\beta)$ is smooth over the algebraic stack $\mathcal{V}(T)$, defined by the fiber diagram:
 
\begin{equation}
\begin{tikzcd}
\mathcal{V}(T) \arrow{r} \arrow[d] & \mathcal{A}(T)\arrow[d,"s"]\\
B\mathbb{G}_m \arrow[r] & \mathbb{A}^1/\mathbb{G}_m
\end{tikzcd}.
\end{equation}

In particular, $\mathcal{V}(T)$ is equidimensional with $[\mathcal{V}(T)] = \sum_{\omega} a_\omega[\mathcal{V}(\omega)]$, where for $\eta_{\omega}$ the geometric generic point of a component $\mathcal{V}_{\omega}$. Taking local toric charts yields:

\begin{equation}\label{multi}
a_\omega = \val_{\eta_{\omega}}(s) = |\coker(\ev_v:\omega_{\NN}^{\gp} \rightarrow \ZZ)|.
\end{equation}
Virtual pullback of the given decomposition of $[\mathcal{V}(T)]$ now gives the desired description of $[\mathscr{M}(\widetilde{X},\beta)]^{\vir}$.

Since $x \in \Sigma(X)$ was general, we have it's image $x\in \Sigma(X)_{\pmb\sigma(L_{\out})}$ is contained in a maximal cone $x \in \sigma$. As $\Sigma(\widetilde{X})$ is simplicial, so is the quotient $ \Sigma(\widetilde{X})_{\pmb\sigma(L_{\out})}$. Thus, for every one dimensional face $\rho \subset \sigma$, there is a PL function $P_{\rho}$ uniquely determined by having slope $1$ along $\rho$ and vanishing on all other $1$-dimensional faces of $\Sigma(\widetilde{X})$. We set $P_x = \prod_{\rho \in \sigma^{[1]}} P_{\rho}$. Since by assumption the support of the vector $u(L_{\out})$ has size $1$ if $u(L_{\out})$ is a non-negative contact order and size $2$ otherwise, we have $\deg(P) = \dim \sigma-1$ if $u(L_{\out})$ is a non-negative contact order, and $\deg(P) = \dim \sigma - 2$ otherwise.


Letting $\overline{\pmb\omega}$ be the decorated tropical type of map to $X$ for which $\pmb\omega$ is a tropical lift, an application of Theorem \ref{mthm1} now yields:
 \begin{equation}\label{pushbl}
st_*(\ev^*(P_x)\cap[\mathscr{M}(\widetilde{X},\beta)]^{\vir}) = \sum_{\pmb\omega}  \sum_{\overline{\pmb\omega}\subset\pmb\tau} \frac{\deg_{\tau}(a_{\omega}P_xQ_{\omega})}{|Aut(\tau/\beta)|}[\mathscr{M}(X,\pmb\tau)]^{\vir}.
\end{equation}
For future reference, we define 
$$N_{p,u(L_{\out})}^{x,\textbf{A}}:= \deg(\ev^*[\pt]\cap[\mathscr{M}(\widetilde{X},\beta)]^{\vir}) = \deg(\ev^*(\pi^*[\pt]\cup P_x)\cap[\mathscr{M}(\widetilde{X},\beta)]^{\vir}).$$
 We note a tropical type $\pmb\tau$ contributes to the above equation only if the outgoing leg $L_{\out}$ intersects the interior of the maximal cone $\sigma$ containing $x$. In particular, any broken line type contributes to the righthand side of Equation \ref{maineq1}. On the other hand if $\dim h_{\tau}(\tau_{\out}) < \dim \sigma$, then as $\pmb\tau$ is marked by $\beta$, $\tau_{\out}$ cannot intersect the interior of maximal cone by construction of the modification $\widetilde{X} \rightarrow X$. Hence we must have $\dim h_{\tau}(\tau_{\out}) = \dim \sigma$. Since $\deg(P_x) = \vdim \mathscr{M}(\widetilde{X},\beta) +\dim \sigma - \dim X$, we must also have $\dim\tau \le \dim \sigma-1$ for all contributing $\tau$. In order to have $h_{\tau}(\tau_{\out}) = \dim \sigma$, we must in fact have $\dim \tau = \dim \sigma-1$. As the first two conditions for a broken line type are automatically satisfied by $\tau$ and we just verified the third condition, we deduce the only contributions to the righthand side of Equation \ref{maineq1} are broken line types.

To evaluate $\deg_{\tau}(a_\omega P_xQ_{\omega})$, for a broken line type $\pmb\tau$, let $v \in V(G_{\tau})$ be the unique vertex contained in $L_{\out}$ and note that $\ev_{L_{\out}}: \tau \rightarrow \Sigma(X)_{\pmb\sigma(L_{\out})}$ is an inclusion of cones by condition $3$ of Definition \ref{bltype}. After potentially further subdividing, we may assume $\sigma\subset \ev_{L_{\out}}(\tau)$, so $\ev_{L_{\out}}^{-1}(\sigma) \subset \tau$ is a simplicial subcone of dimension $\dim\tau$. Moreover, we have $a_\omega Q_{\omega}$ is a linear function on the cone $\ev_{L_{\out}}^{-1}(\sigma)$ given by pulling back the PL function on $\sigma$ which has slope $1$ on the one dimensional face $\pmb\sigma(v) \subset \sigma$ for $v \in V(G_{\omega})$ the vertex containing $L_{\out}$, and vanishing on all other one dimensional faces. As a result, since $ev_v^*(P_x)$ is a product of integral PL functions on $\tau$, by Example \ref{simpdeg}, we have:

\begin{equation}\label{blmult}
\begin{split}
\deg_{\tau}(a_\omega P_xQ_{\omega}) &= \deg_{\tau}(\ev_v^*(P_\sigma)) \\
&= |\coker(\ev_v: \tau^{\gp}_{\NN} \rightarrow \ZZ^{\dim \sigma} = \sigma^{\gp}_{\NN})| \\
&= a_{\tau}.
\end{split}
\end{equation}
Since $k_{\tau} = a_{\tau}|u(L_{out})|$, multiplying both sides of Equation \ref{pushbl} by $|u(L_{\out})|$ yields:
 
 \begin{equation}
|u(L_{\out})| st_*(\ev^*(P_x)\cap[\mathscr{M}(\widetilde{X},\beta)]^{\vir}) = \sum_{\pmb\omega}\sum_{\overline{\pmb\omega} \subset \pmb\tau \in B(p,x,\textbf{A})} k_{\tau}N_{\pmb\tau} .
 \end{equation}
Finally, in the case that $u(L_{\out})$ is not a positive contact order, we observe that for a broken line type $\pmb\tau \in B(p,x,\textbf{A})$, there exists a unique $\omega$ such that stablization induces a map $\omega \rightarrow \tau$. To see this, note that the vertex $v \in L_{\out} \subset G_{\omega}$ must have image a $1$-dimensional cone of $\Sigma(\widetilde{X})$ corresponding to the non-positive contact order of $u(L_{\out})$. Since the tropical moduli of $\omega$ is completely determined by its image in $\tau$, and the tropical moduli of $\tau$ is completely determined by its image in $\Sigma(X)$, it follows that any two types $\omega_1,\omega_2$ which map to $\tau$ via the tropicalization of the stabilization morphism agree. 

 We therefore conclude the following equality of elements of $\kk[Q][\sigma_{\NN}^{\gp}]$:
 
 \begin{equation}\label{blinet}
 \sum_{\beta} |u(L_{out})|\deg(\ev^*(\pt)\cap[\mathscr{M}(\widetilde{X},\beta)]^{\vir})t^{\textbf{A}}z^{-u(L_{out})} = \vartheta_{p,x}.
 \end{equation}
 By the main result of \cite{BNR2}, the invariants used to define the lefthand side of Equation \ref{blinet} are also given by $\deg(\ev^*(P)\cap[\mathscr{M}^{\orb}(\widetilde{X}_{\vec{r}},\beta)]^{\vir})$ for some choice of piecewise polynomial $P$, with $\mathscr{M}^{\orb}(\widetilde{X}_{\vec{r}},\beta)$ the moduli stack of twisted stable maps to a root stack $\widetilde{X}_{\vec{r}} \rightarrow \widetilde{X}$, with $\vec{r} \in \NN^k$ a vector of sufficiently large integers with $k$ the number of boundary divisors of $\widetilde{X}$. Moreover, in the broken line case, You in \cite{orbBL} constructs an analogous broken-line expansion for orbifold theta functions on snc log Calabi-Yau pairs using orbifold invariants. By \cite[Theorem $7.1$]{orbBL}, under the assumption that the output contact order $u(L_{out})$ contains only one negative contact order, You shows that the orbifold broken line expansion is equal to the left hand side of Equation \ref{blinet}, with log invariants replaced with orbifold invariants. We therefore conclude that You's broken line theta function $\vartheta_p(x)$ on an appropriate cofinal collection of log \'etale blowup coincides with intrinsic broken line expansion.
 \end{proof}
 
 \subsection{Wall functions from orbifold Gromov-Witten theory of a root stack}
 
 \begin{proposition}\label{worb}
 The product of wall function $\mathfrak{p}_{p,x,\textbf{A}}$ is the exponential of a generating series of orbifold Gromov-Witten invariants of root stacks of a cofinal collection of log blowups. 
 \end{proposition}
 
\begin{proof}
We now turn to the wall type case, in particular we assume $(X,D)$ satisfies the additional assumptions taken in \cite{scatt}. We wish to express wall type invariants by integrals on an appropriate moduli stack $\mathscr{M}(\widetilde{X},\beta)$ for an appropriate choice of subdivision $\widetilde{X}$ and lift $\beta$ of a fixed combinatorial type $\beta'$ with one leg $L_{\out}$ and total curve class $\textbf{A}$. Note that for any decorated tropical type $\pmb\omega'$ marked by $\beta'$, we have $\vdim \mathscr{M}(X,\pmb\omega') \le \dim X - 2$. Consider the images $\Sigma(h_{\pmb\omega'}(\omega'_{\out})) \subset \Sigma(X)$ for decorated tropical types $\pmb\omega'$ types marked by $\beta'$ such that $[\mathscr{M}(X,\pmb\omega')]^{\vir} \not =0$. Virtual dimension constraints ensure that $\dim \omega' \le \dim X - 2$, hence in particular $\dim\Sigma(h_{\omega'}(L_{\out}))\le \dim X-1$. 

After picking a general point $x \in\Sigma(h_{\omega'}(\omega'_{\out}))$ for some $\omega'$ with $\dim\Sigma(h_{\omega'}(L_{\out}))= \dim X-1$, consider a log blowup $\widetilde{X} \rightarrow X$ and lift $\beta$ of $\beta'$ such that for any realizable tropical type $\omega$ marked by $\beta'$, the image $h_{\omega}: \omega_{\out} \rightarrow \Sigma(X)$ is a union of cones of $\widetilde{\Sigma(X)}$, the lift $u(L_{\out})$ contained in $\text{Star}(\sigma)$ for $\sigma$ the minimal dimension cone containing $x$ has at most one negative contact order and there is at most one component $D_i \subset D$ such that $u(L_{\out})(D_i) < 0$. As in the broken line case, this blowup exists and is cofinal in the collection of log blowups by \cite[Proposition $6.1$]{BNR2}, and $\mathscr{M}(\widetilde{X},\beta)$ is virtually equidimensional of dimension $\dim X - 2$ if $u(L_{\out})$ is a non-negative contact order, and $\dim X -3$ otherwise, with virtual class as in Equation \ref{vdcmp}. Since $x$ is a general point of $h_{\omega'}(\omega'_{\out})$, we have $\dim \sigma = \dim X - 1$.

Let $P_{x}$ be the piecewise polynomial function given by the product of all piecewise linear function associated with one dimensional faces of the image of $\sigma$ under $\text{Star}(\pmb\sigma(L_{\out})) \rightarrow \Sigma(\widetilde{X})_{\pmb\sigma(L_{\out})}$. In particular, if $p$ is a non-negative contact order, we have $\deg(P_{x}) = \dim X -2$ and otherwise $\deg(P_x) = \dim X - 3$. Applying Theorem \ref{mthm1} to the invariant $st_*(P_{x}\cap[\mathscr{M}(\widetilde{X},\beta)]^{\vir})$ gives:
 
 \begin{equation}
 st_*(\ev^*(P_x)\cap[\mathscr{M}(\widetilde{X},\beta)]^{\vir}) = \sum_{\pmb\omega}  a_{\omega}\sum_{\pmb\omega\subset \pmb\tau} \frac{\deg_{\tau}(P_xQ_{\omega})}{|Aut(\tau/\beta)|}[\mathscr{M}(X,\pmb\tau)]^{\vir}.
 \end{equation}
 
Now let $W(p,x,\textbf{A})$ be the set of wall types $\pmb\tau$ on $X$ with $u(L_{\out}) = p$ and $h(\tau_{\out})\cap\sigma^{\circ} \not= \emptyset$. As in the broken line case, a type $\pmb\tau$ only contributes to the righthand side of the above equation if $L_{\out}$ intersects the interior of the cone $\sigma$. By the same argument as in the broken line case, we have only wall types contained in $W(p,x)$ contribute to the righthand side. An analogous argument to the broken line case yields $\deg_{\tau}(a_{\omega}P_xQ_{\omega}) = |\ev_v: \tau^{\gp}_{\NN}\rightarrow \sigma(v)| = a_{\tau}$, and since $k_{\tau} = |u(L_{\out})|a_{\tau}$, we conclude that:
 
  \begin{equation}\label{wallt}
|u(L_{\out})| st_*(\ev^*(P_x)\cap[\mathscr{M}(\widetilde{X},\beta)]^{\vir})t^{\textbf{A}}z^{-u(L_{\out})} = \sum_{\pmb\tau \in W(p,x,\textbf{A})} k_{\tau}N_{\pmb\tau} t^{\textbf{A}}z^{-u(L_{\out})}.
 \end{equation}
After exponentiating both sides of the above equation, we recover product of wall function $\mathfrak{p}_{p,x,\textbf{A}}$. As in the broken line case, the main result of \cite{BNR2} ensures that the invariant on the lefthand side of Equation \ref{wallt} is an orbifold Gromov-Witten invariant on a blowup contained in a cofinal collection, as desired. 
\end{proof}

 \subsection{Broken line expansions respect the product rule}
 
In \cite[Theorem $6.1$]{scatt}, Gross and Siebert show that the assignment of broken line expansions to theta functions extends to an algebra homomorphism $R_{(X,D)} \rightarrow \kk[\sigma_{p,\NN}^{\gp}][Q]$. Using the relationship with generalized broken line expansions with orbifold Gromov-Witten invariants, we generalize this result to all log Calabi-Yau pairs considered in the more general setting of \cite{int_mirror} in the following theorem:

 \begin{theorem}\label{blhom}
For $p,q \in B(\NN)$, $x \in \sigma \in B$ a general point of a maximal dimensional cone, and $s \in \sigma_{\NN}^{\gp}$, we have:
 \begin{equation}\label{blWDVV}
 \sum_{(\textbf{A}_1, \textbf{A}_2,s')} N_{p_1,p_2,s'}^{\textbf{A}_1}|s|N_{s',s}^{x,\textbf{A}_2} = \sum_{(\textbf{A}_1, \textbf{A}_2 ,s')} |s'|N_{p_2,s'}^{x,\textbf{A}_1}|s-s'|N_{p_1,s-s'}^{x,\textbf{A}_2}.
 \end{equation}
 In particular, the assignment of broken line expansions extends to a homomorphism of $\kk[Q]$ algebras $R_{(X,D)} \rightarrow \kk[\sigma_{p,\NN}^{\gp}][Q]$.
 \end{theorem}
 
 \begin{proof}
Let $u \in \sigma^{\gp}_{\NN}$ be an integral tangent vector to $x$ such that $u$ and $-s+u$ non-negative contact orders with $D_i(u),D_i(-s+u) >> D_i\cdot \textbf{A}$ for all irreducible boundary divisors $D_i$. More precisely, since $X$ is projective, there are only finitely many decompositions $\textbf{A} = \textbf{A}' + \textbf{A}''$ for effective curve classes $\textbf{A}',\textbf{A}''$, and we suppose that $D_i(u),D_i(-s+u) > D_i\cdot\textbf{A}'$ for all $i$ and effective curve classes $\textbf{A}'$ decomposing $\textbf{A}$. By \cite[Theorem $8.3$]{logWDVV}, after passing to an appropriate log \'etale modification $\widetilde{X}\rightarrow X$, we have the following relation of refined punctured invariants in the sense of \cite{BNR2}:
 
 \begin{equation}\label{WDVV2}
\begin{split}
\sum_{(\textbf{A}_1, \textbf{A}_2,s',l)} \langle[1]_{p_1},[1]_{p_2},T_{-s',l}\rangle_{0,\textbf{A}_1,3}\langle T^l_{s'},[1]_{-u},[\pt]_{-s+u}\rangle_{0,\textbf{A}_2,3} \\
= \sum_{(\textbf{A}_1, \textbf{A}_2 ,s',l)} \langle[1]_{p_2},[1]_{-u},T_{-s',l}\rangle_{0,\textbf{A}_1,3}\langle T^l_{s'},[1]_{p_1},[\pt]_{-s+u}\rangle_{0,\textbf{A}_2,3}.
\end{split}
\end{equation}

Moreover, we may assume that $|\{D_i\text{ }|\text{ }D_i(-u)< 0\}| \le 1$ by \cite[Proposition $6.1$]{BNR2}, and $-s+u$ corresponds to a divisorial valuation. Note that since $x \in \sigma$ was general, we have $x$ is contained in a maximal dimension cone of $\widetilde{X}$. Since the tangent spaces are unaffected by the subdivision, we will also refer to this maximal cone of $\Sigma(\widetilde{X})$ as $\sigma$. 


 After replacing $X$ with $\widetilde{X}$, consider a contribution coming from $s' \in \sigma^{\gp}_{\NN}$. We note that the invariants $\langle T^l_{s'},[1]_{-u},[\pt]_{-s+u}\rangle_{0,\textbf{A}_2,3}$ and $\langle T^l_{s'},[1]_{p_1},[\pt]_{-s}\rangle_{0,\textbf{A}_2,3}$ are non-zero only if $s' \in B(\ZZ)$ by \cite[Lemma $8.6$]{logWDVV} and the expression for the refined virtual class in terms of realizable tropical types. In particular, the only non-zero contributions to either side of Equation \ref{WDVV2} satisfy $s' \in B(\ZZ)$. 

Letting $\beta$ be the combinatorial type associated with the invariant $\langle T^l_{s'},[1]_{-u},[\pt]_{-s+u}\rangle_{0,\textbf{A}_2,3}$, the refined virtual dimension of $\mathscr{M}(X,\beta)$ is:
\begin{equation}\label{vdimb}
\begin{split}
\vdim \mathscr{M}(X,\beta) &= \dim X  -3 + |L(\beta)| - |\{D_i\text{ }|\text{ }D_i(-u)< 0\}| - |\{D_i\text{ }|\text{ }D_i(s')< 0\}| \\
&= \dim X - 1 - |\{D_i\text{ }|\text{ }D_i(s')< 0\}.
\end{split}
\end{equation}
Since $[\pt]_{-s+u}$ is a codimension $\dim X - 1$ condition, we must have $s'$ is a non-negative contact order in order for the invariant in question to be non-zero, and $T^l_{s'} = [1]_{s'}$. 
Now consider any component $\mathfrak{M}_{\eta}$ of $\mathfrak{M}(\mathcal{X},\beta)$ with non-zero virtual pullback to $\mathscr{M}(X,\beta)$ with associated decorated tropical type $\pmb\omega = \pmb\omega_{\eta}$ and let $v \in V(G_{\omega_{\eta}})$ be the vertex contained in the leg $L_{-u}$ with contact order $-u$. The multiplicity of the component can be computed as in \ref{multi}, giving $a_{\omega} = |\coker(\ev_v: \omega_{\NN}^{\gp} \rightarrow \ZZ)|$. For a fixed divisor $D_i$, consider the induced tropical type $\omega_i$ of log stable map relative to $D_i$, and cut $\omega_i$ along all edges $e$ with $u(e)\not= 0$ containing a vertex $v'$ with $\pmb\sigma(v') = 0$, and consider the tropical type $\omega_{i,v}$ corresponding to the connected component containing the vertex $v$. Letting $\textbf{A}'$ be the total curve class of the $\pmb\omega_{i,v}$, the log balancing condition imposes the constraint:

$$\textbf{A}'\cdot D_i = \sum_{l \in L(G_{\omega_{i,v}})} D_i(u(l)).$$

Since any new legs produced from cutting must satisfy $D_i(u(e)) < 0$ with magnitude bounded above by $\textbf{A}''\cdot D_i$ with $\textbf{A}''$ an effective curve class decomposing $\textbf{A}$, by taking $u$ sufficiently large, we can ensure that $\omega_{i,v}$ inherits at least one leg from $\omega$. Since there are only finitely many possibilities for the positive contact $s'$, taking $u$ sufficiently large once again ensures that one of these legs must be the leg $L_{-s+u}$ with contact order $-s+u$. Letting $v'$ be the vertex contained in the leg $L_{-s+u}$, since $\pmb\sigma(v') \not= 0$ for all tropical types $\omega_{v,i}$ associated with the divisors $D_i$ which contain $X_{\pmb\sigma(v)}$, we must have $\pmb\sigma(v')$ and $\pmb\sigma(v)$ map into the same $\dim B$ dimensional cone of the original target $X$. With this condition, we wish to show that $\pi_*\textbf{A}(v') = 0$. To do so, we consider the following generalization of Lemma $8.5$ of \cite{qpint}:

\begin{lemma}\label{contract}
Let $X \rightarrow \overline{X}$ a log \'etale modification, and $\pmb\omega$ be a genus $0$ decorated realizable tropical type of punctured log map to $X$ with a single vertex $v$ and a leg $L_{\out}$ with $\pmb\sigma(v) \in \Sigma(X)$ a cone with interior mapping to a maximal cone of $\Sigma(\overline{X})$ under the subdivision $\Sigma(X) \rightarrow \Sigma(\overline{X})$. If $\pi_*\textbf{A}(v) \not= 0$, then $[\mathscr{M}(X,\pmb\omega)]^{\vir}\cap \ev^*(\pt) = 0$. 
\end{lemma}

\begin{proof}
Recall that for $z \in X_{-s+u}$, there are the point constrained moduli stacks $\mathscr{M}(X,\pmb\omega)_z = \mathscr{M}(X,\pmb\omega)\times_{\underline{X_{-s+u}}} z$ and $\mathfrak{M}^{\ev}(\mathcal{X},\pmb\omega)_z$ and a perfect obstruction theory for the natural morphism $\mathscr{M}(X,\pmb\omega)_z \rightarrow \mathfrak{M}^{\ev}(\mathcal{X},\pmb\omega)_z$ pulled back from the morphism of unconstrained moduli stacks. Moreover, Gysin pullback along the regular inclusion $z \rightarrow X_{-s+u}$ yields the virtual fundamental class $[\mathscr{M}(X,\pmb\omega)_z]^{\vir}$ such that $$z_*[\mathscr{M}(X,\pmb\omega)_z]^{\vir} = \ev^*_x([\pt])\cap[\mathscr{M}(X,\pmb\omega)]^{\vir}.$$ To prove the lemma, it therefore suffices to show that $[\mathscr{M}(X,\pmb\omega)_z]^{\vir} \not= 0$ only if $\pi_*\textbf{A}(v) = 0$. 

By the assumption that the interior of $\pmb\sigma(v)$ maps to a maximal cone of $\Sigma(X)$, the stratum $X_{\pmb\sigma(v)}$ is a split toric bundle over the stratum associated with the maximal dimensional stratum of the original log Calabi-Yau $\overline{X}$. Moreover, $\pmb\omega$ is a tropical lift of a unique decorated tropical type $\pmb\tau$ of punctured log map to $\overline{X}$, and in particular there is a vertex $v \in V(G_{\tau})$ which maps into a maximal dimension cone of $\Sigma(\overline{X})$. By cutting at all compact edges of $G_{\omega}$ containing $v$, we produce another decorated tropical type $\pmb\omega'$ containing a single vertex, which is the tropical lift of a unique tropical type $\pmb\tau'$ with a single vertex decorated by $\pi_*\textbf{A}(v) $. If $[\mathscr{M}(X,\pmb\omega')_z]^{\vir} \not=0$, then $[\mathscr{M}(X,\pmb\omega)_z]^{\vir}\not=0$ by \cite[Theorem $A.14$]{int_mirror}. Thus, it suffices to show $[\mathscr{M}(X,\pmb\omega')_z]^{\vir} \not=0$ only if $\pi_*\textbf{A}(v) = 0$. 

%
%
%


Consider a tropical type $\pmb\tau''$ with a single vertex $v$ and two legs, one with contact order $u(D_i) = \textbf{A}\cdot D_i - l(D_i)$, and the other with contact order $l$ as before. Stability of the map ensures that the moduli stack $\mathscr{M}(\overline{X},\pmb\tau'')$ is non-empty only if $\pi_*\textbf{A}(v) \not= 0$. If we now assume $\pi_*\textbf{A}(v) \not= 0$, the argument \cite[Lemma $8.5$]{qpint} extends beyond the setting of smooth divisors to give a partial stabilization maps $\mathscr{M}(\overline{X},\pmb\tau') \rightarrow \mathscr{M}(\overline{X},\pmb\tau'')$ and $\mathfrak{M}^{\ev}(\mathcal{X},\pmb\tau') \rightarrow \mathfrak{M}^{\ev}(\overline{\mathcal{X}},\pmb\tau'')$, and for $z' \in \overline{X}_{\pmb\sigma(v)}$ general, a cartesian diagram:

\begin{equation}\label{vancart}
\begin{tikzcd}
\mathscr{M}(X,\pmb\omega')_z \arrow{r} \arrow{d} &\mathscr{M}(\overline{X},\pmb\tau')_{z'} \arrow{d}\arrow{r} & \mathscr{M}(\overline{X},\pmb\tau'')_{z'} \arrow{d} \\
\mathfrak{M}^{\ev}(\mathcal{X},\pmb\omega')_z\arrow{r} & \mathfrak{M}^{\ev}(\overline{\mathcal{X}},\pmb\tau')_{z'}\arrow{r} & \mathfrak{M}^{\ev}(\overline{\mathcal{X}},\pmb\tau'')_{z'}.
\end{tikzcd}
\end{equation}
Moreover, the standard obstruction theories for the vertical morphisms are compatible. Since the tropical modulus of $\tau''$ is completely determined by the image of $v$, we have $\dim \overline{X}_{z'} = \dim \overline{X} - \dim \pmb\sigma(v) = \dim \overline{X} - \dim \tau''$, and the virtual dimension of $\mathscr{M}(\overline{X},\pmb\tau'')_z$ is:
$$\vdim \mathscr{M}(\overline{X},\pmb\tau'')_z = (\dim \overline{X} - 3) + 2 - \dim B - (\dim \overline{X} - \dim B)  = -1.$$ 
It now follows from \cite[Lemma $A.13$]{int_mirror} that $[\mathscr{M}(X,\pmb\omega')_z]^{\vir} = 0$ if $\pi_*\textbf{A}(v) \not= 0$, as required. 
\end{proof}


Consider the tropical type $\pmb\omega_v$ produced by cutting across all edges of $G_{\omega}$ which contain $v$. Such a tropical type is considered in  Lemma \ref{contract}, hence $[\mathscr{M}(X,\pmb\omega_v)]^{\vir}\cap \ev^*(pt)\not=0$ only if $\pi_*\textbf{A}(\pmb\omega_v) = 0$. In particular, the curve must satisfy the usual tropical balancing condition and must contain at least $3$ legs. 

We now show that we must have $v = v'$ and $\val(v) = 3$ if $[\mathscr{M}(X,\pmb\omega)_z]^{\vir} \not= 0$. First, suppose either $v \not= v'$ or $\val(v) > 3$. Then $v$ would be contained in an edge which is not contained in the convex hull of the legs of $G_\omega$. For $v \in e$ one such edge, let $\omega'$ be the tropical type produced by cutting at all edges containing $v$ other than $e$, and $\pmb\omega^c$ the disconnected tropical type of the remainder. Consider also the tropical type $\pmb\omega''$ given by cutting $\pmb\omega'$ at $e$ and taking the connected component not containing $v$. Splitting induces a morphism $\mathscr{M}(X,\pmb\omega') \rightarrow \mathscr{M}(X,\pmb\omega'')$. This morphism gives the top horizontal morphism in the following commutative diagram:
\begin{equation}\label{fgtdi}
\begin{tikzcd}
\mathscr{M}(X,\pmb\omega')_z\arrow{r}\arrow{d} & \mathscr{M}(X,\pmb\omega'')_z\arrow{d}\\
\mathfrak{M}^{\ev}(\mathcal{X},\pmb\omega')_z\arrow{r} & \mathfrak{M}^{\ev}(\mathcal{X},\pmb\omega'')_z.
\end{tikzcd}
\end{equation}
We claim that this diagram is cartesian. To see this, suppose we have a family of maps $C'' \rightarrow X$ marked by the type $\pmb\omega''$, and a family of maps $C' \rightarrow \mathcal{X}$ of type $\pmb\omega'$ such that $C'' \rightarrow X \rightarrow \mathcal{X}$ is isomorphic to a restriction of $C' \rightarrow \mathcal{X}$ to components marked by vertices contained in $V(G_{\omega''})$. The composition $C'' \rightarrow X \rightarrow \overline{X}$ can be extended as a stable map to $C' \rightarrow \overline{X}$ by contracting all components marked by $v$, and composing further gives a prestable map $C' \rightarrow \overline{\mathcal{X}}$. The previous prestable map in turn agrees with the composition $C' \rightarrow \mathcal{X}\rightarrow \overline{\mathcal{X}}$. Indeed, agreement upon restriction to $C'' \subset C'$ holds by assumption, and agreement on $C_v$ follows from the fact that $C_v \rightarrow \overline{X}$ is genus $0$ and contracted, and the tropicalization of the provided map $C_v \rightarrow \mathcal{X}$ satisfies the ordinary tropical balancing condition. Thus, $C' \rightarrow \overline{X}$ admits a log enhancement. Since we have a factorization $C' \rightarrow \mathcal{X} \rightarrow \overline{\mathcal{X}}$, we have a unique lift $C' \rightarrow X$ of type $\pmb\omega'$ compatible with the given data. It is straightforward from this construction also respects the point constraint, again using the fact that $C_v \rightarrow X$ is contracted. 

To see that the obstruction theories are related via pullback, we denote by $g: C'' \rightarrow X$ and $f: C' \rightarrow X$ the log stable maps of type $\pmb\omega''$ and $\pmb\omega'$ induced by $S \rightarrow \mathscr{M}(X,\pmb\omega')_z \rightarrow \mathscr{M}(X,\pmb\omega'')_z$ and $\overline{g},\overline{f}$ their respective stabilizations. Since we have the equality $\pi^*T_{\overline{X}}^{\log} = T_X^{\log}$, $c: C' \rightarrow C''$ is a contraction of rational components, and $c_*\mathcal{O}_{C'}(-x') = \mathcal{O}_{C''}(-x'')$ for $x',x''$ the marked sections of $C'$ and $C''$ possessing the point constraint, it follows as in \cite[Section $6$]{bir_GW} that:
 $$Rc_*f^*T_X^{\log}(-x')=Rc_*\overline{f}^*T_{\overline{X}}^{\log}(-x') = Rc_*c^*\overline{g}^*T_{\overline{X}}^{\log}(-x'')=g^*T_X^{\log}(-x'').$$
For $p'': C'' \rightarrow S$ and $p': C' \rightarrow S$ the family of log curves, applying $Rp''_*$ to the above equations shows $Rp'_*f^*T_X^{\log}(-x') = Rp''_*g^*T_X^{\log}(-x'')$. In particular, the obstruction theories are related via pullback.

Combining \cite[Theorem $C$]{punc} with the Diagram \ref{fgtdi} gives the following Cartesian diagram:


\begin{equation}
\begin{tikzcd}
\mathscr{M}(X,\pmb\omega)_z \arrow{r}\arrow{d} & \mathscr{M}(X,\pmb\omega^c)\times\mathscr{M}(X,\pmb\omega'')_z\arrow{d}\\
\mathfrak{M}^{\ev}(\mathcal{X},\pmb\omega)_z \arrow{r} & \mathfrak{M}^{\ev}(\mathcal{X},\pmb\omega^c)\times \mathfrak{M}^{\ev}(\mathcal{X},\pmb\omega'')_z.
\end{tikzcd}
\end{equation}
Moreover, the natural obstruction theories are related via pullback. Finally, $$\vdim \mathscr{M}(X,\pmb\omega'')_z = \dim X - 2 - (\dim X - 1) = -1.$$ It follows by \cite[Lemma $A.13$]{int_mirror} that $[\mathscr{M}(X,\pmb\omega)_z]^{\vir} = 0$ assuming either $v\not= v'$ or $\val(v) > 3$.

By the constraints established above, $\pmb\omega$ is a gluing along an edge $e$ of a two pointed tropical type $\pmb\omega'$ contributing to the righthand side of Equation \ref{vdcmp} with a three pointed tropical type $\pmb\omega_v$ with a single vertex mapping to a ray decorated by a curve class which is contracted by $\pi$. The contribution from the type $\pmb\omega$ may be calculated as above using Theorem \ref{mthm1}. By analogous arguments used in the proof of Corollary \ref{blorb}, the only contributions to the invariant $\deg\ev^*(\pt)\cap [\mathscr{M}(X,\pmb\omega)]^{\vir} $ are from types $\pmb\tau$ which are gluings of generalized broken line types $\pmb\tau'$ with a three pointed tropical types $\pmb\tau_v$ with a single vertex decorated by the contracted curve class. After recalling that the multiplicity of the component $\mathfrak{M}_{\eta}$ in $\mathfrak{M}(\mathcal{X},\beta)$ is $a_{\omega}$, by the same argument given for evaluating the analogous degree for broken line types, the contribution of $\pmb\tau$ to the invariant of interest is:
\begin{equation}
\begin{split}
\deg_{\tau}(a_{\omega}Q_{\omega}P_x)\deg[\mathscr{M}(\overline{X},\pmb\tau)]^{\vir} &= |\coker(\ev_v: \tau^{\gp}_{\NN} \rightarrow \sigma^{\gp}_{\NN})|\deg[\mathscr{M}(\overline{X},\pmb\tau)]^{\vir} \\
&= |s|a_{\tau}\deg[\mathscr{M}(\overline{X},\pmb\tau)]^{\vir},
\end{split}
\end{equation}
with $a_{\tau}$ as defined in Equation \ref{blmult}. 



Note that $\ev_v: \tau_v \rightarrow \pmb\sigma(v)$ is surjective, hence $\pmb\tau$ is a tropically transverse gluing\footnote{Gross defines a tropical type $\tau$ to be tropically transverse if a condition is satisfied by the splitting of $G_{\tau}$ across all edges. We note that the required gluing results also apply after only partially splitting the tropical type, with the obvious modification of the tropical transverse condition.}. Since $\pmb\sigma(v)$ is a top dimensional cone of $\Sigma(X)$, for every edge $v' \in e$, flatness of the evaluation map $\ev_{p_e}: \mathfrak{M}^{\ev}(\overline{\mathcal{X}},\pmb\tau)\rightarrow \overline{X}_{\pmb\sigma(v)}$ follows from \cite[Theorem $5.3$]{trglue}. We may therefore apply Theorems $5.1$ and $5.5$ of \cite{trglue} to compute the contribution of $\deg[\mathscr{M}(\overline{X},\pmb\tau)]^{\vir}$. Letting $v'' \in e$ be the vertex not equal to $v$, the gluing factor from loc. cit. is given by $|\coker(\tau^{'\gp}_{\NN}\times\tau^{\gp}_{v,\NN}\times \ZZ \rightarrow \pmb\sigma(e))|$, with the map given by the derivative of the map sending a triple $(h_{\tau'},h_{\tau},n)$ of two tropical maps and an integer to $h_{\tau'}(v'') - h_{\tau_v}(v) - nu(e)$. Since $\ev_v: \tau_v \rightarrow \pmb\sigma(v)$ is surjective on lattice points, the tropical multiplicity is one, and we conclude the contribution of $\pmb\tau$ is $|s|a_{\tau}N_{\pmb\tau'}$. Summing over all contributing types $\pmb\tau$ to the virtual class $[\mathscr{M}(\widetilde{X},\beta)]^{\vir}$ now yields:



\[\langle [1]_{s'},[1]_{-u},[\pt]_{-s+u}\rangle_{0,\textbf{A}_2,3} =  |s|N_{s',s}^{x,\textbf{A}_2}\]

Turning to the three pointed invariant $\langle T_{s'}^l,[1]_{p_1},[\pt]_{-s+u}\rangle_{0,\textbf{A}_2,3}$, by an analogous argument to the one following Equation \ref{vdimb}, the logarithmic balancing condition implies that $s'$ must be a purely negative contact order with $s' \in \sigma_{-s+u,\NN}^{\gp}$, with magnitude positively correlated with the magnitude of $-s+u$. Virtual dimension constraints further impose $|\{D_i\text{ }|\text{ }D_i(s') < 0\} = 1$ and $T_{s'}^l = [1]_{s'}$. With these constraints, analogous arguments as above give:
 \begin{equation}
 \langle [1]_{s'},[1]_{p_1},[\pt]_{-s+u}\rangle_{0,\textbf{A}_2,3} = |s-s'-u|N_{p_1,s-s'-u}^{x,\textbf{A}_2}.
 \end{equation}
 
 Since $s'$ is purely negative, $-s'$ is a purely positive contact order with $D_i(s') >> D_i \cdot \textbf{A}'$ for all effective curve classes decomposing $\textbf{A}$, and similar arguments once again yield:
 \begin{equation}
 \langle[1]_{p_2},[1]_{-u},[\pt]_{-s'}\rangle_{0,\textbf{A}_1,3} = |s'+u|N_{p_2,s'+u}^{x,\textbf{A}_1}.
 \end{equation}
 
Thus, by substituting the above relations into Equation \ref{WDVV2}, we find:
 
 \begin{equation}\label{wdvv2}
 \sum_{(\textbf{A}_1, \textbf{A}_2,s')} N_{p_1,p_2,s}^{\textbf{A}_1}|s|N_{s',s}^{x,\textbf{A}_2} = \sum_{(\textbf{A}_1, \textbf{A}_2 ,s'' = s'+u)} |s''|N_{p_2,s''}^{x,\textbf{A}_1}|s-s''|N_{p_1,s-s''}^{x,\textbf{A}_2}.
 \end{equation}
 
 The terms of the righthand side of Equation \ref{wdvv2} clearly match the terms appearing in the sum on the righthand side of Equation \ref{blWDVV}, hence the desired equation follows.

 \end{proof}
 
 \begin{remark}
 \begin{enumerate}
 \item Note the proof of the theorem above also gives another way of expressing the two pointed invariants $N_{p,q}^{x,\textbf{A}}$, namely the three pointed invariants on a log \'etale modification with contact orders $p,-u,-q+u$ for an appropriately chosen integral tangent vector $u \in \pmb\sigma^{\gp}_{\NN}$. In \cite{orbBL}, a similar expression is given in Proposition $7.1$ in terms of mid-ages. Since the vanishing results of \cite{logWDVV} are expressed in terms of orbifold Gromov-Witten invariants without mid ages, we used a relation on the modification which did not use mid ages. 
 \item Peter Zaika in upcoming work both gives another proof of the relation above as well as a study of the wall crossing behavior of generalized broken line expansions. As You outlines in \cite{orbBL}, such an extension of the canonical wall structure description of theta functions is desirable for the investigation of the Doran-Harder-Thompson conjecture via studying intrinsic mirror symmetry for intermediate Tyurin degenerations. 
 \end{enumerate}
 \end{remark}
 
Since both sides of Equation \ref{blWDVV} are finite, the product of broken line expansions is well defined in general. Additionally, virtual dimension considerations ensure that $N_{p,0} \not= 0$ when $p = 0$. Thus, for any $x \in B$, we have $\prod_{i=1}^m \vartheta_{p_i}^x[z^0t^\textbf{A}] = \prod_{i=1}^m \vartheta_{p_i}[\vartheta_0t^{\textbf{A}}]$. By \cite[Theorem $1.3$]{logWDVV}, letting $\beta$ be the tropical type with total curve class $\textbf{A}$ and contact orders $0,p_1,\ldots,p_m$, and $x_{\out}$ the marked point with contact order $0$, we have
 
 \[N_{p_1,\ldots,p_m}^{\textbf{A}} = \int_{[\mathscr{M}(X,\beta)]^{\vir}}\psi_{x_{\out}}^{m-2}\ev^*_{x_{\out}}(\pt) = \prod_{i=1}^m \vartheta_{p_i}[\vartheta_0t^{\textbf{A}}]\]
 Thus, we derive an expression for the descendent integral above in terms of the two pointed invariants $N_{p,i}^x$, generalizing \cite[Theorem $1.5$]{mirrcomp} to the general log Calabi-Yau setting:
 
 \begin{corollary}
 With notation as above, we have:
 \begin{equation}
 \begin{split}
 N_{p_1,\ldots,p_m}^{\textbf{A}} &= \sum_{\substack{q_1+\cdots+q_m = 0,\\\textbf{A}_1+\cdots+\textbf{A}_m = \textbf{A}}} \prod_{i=1}^m |q_i|N^{x,\textbf{A}_i}_{p_i,q_i}\\
 &=  \sum_{\substack{q_1+\cdots+q_m = 0,\\
 \textbf{A}_1+\cdots+\textbf{A}_m = \textbf{A}}}\sum_{\pmb\tau_1,\dots,\pmb\tau_{m}}\prod_i k_{\tau_i}N_{\pmb\tau_i}.
 \end{split}
 \end{equation}
 The sum after the second equality is over decorated tropical types $\pmb\tau_i$ with legs with contact order $p_i$ and $q_i$ and total curve class $\textbf{A}_i$. 
 \end{corollary}
  
We conclude this section with a final application to Fano mirror symmetry. In \cite{qpint}, it is shown that for any pair of a Fano variety $X$ together with an anticanonical divisor $D \in |-K_X|$ for which there exists a log resolution $p: (X',D') \rightarrow (X,D)$ such that $(X',D')$ is log Calabi-Yau, the mirror algebra $R_{(X',D')}$ contains an element $W_D \in R_{(X',D')}$ such that the series $\pi_{W_D} = \sum_{k \ge 0\text{, }p_*\textbf{A}\cdot (-K_X) = k} W_D^k[\vartheta_0t^{p_*\textbf{A}}]t^k$ is the regularized quantum period of $X$ in the sense of \cite{fanmir}. The generalized broken line expansion for $W_D$ with respect to any point the essential skeleton of $X\setminus D$ thus gives Corollary \ref{fanocor}:
  
  \begin{proof}[Proof of Corollary \ref{fanocor}]
Let $X' \rightarrow X$ be a log resolution of the pair $(X,D)$. By assumption, we have the resulting snc pair $(X',D')$ is a log Calabi-Yau pair. Moreover, for $\nu_{D_i}$ the divisorial valuation on $\kk(X)$ associated with the divisor $D_i$, we have $\nu_{D_i} \in B(\NN)$. Consider the element $W_{D}  =\sum_{i} \vartheta_{\nu_{D_i}} \in R_{(X',D')}$, and the broken line expansion $W_{D,x} = \sum_{i} \vartheta_{\nu_{D_i},x}$. We claim that after replacing each coefficient $c_{i,j}x^jt^{\textbf{A}}$ of $\vartheta_{D_i,x}$ with $c_{i,j}x^jt^{-K_X\cdot p_*\textbf{A}}$, we produce a power series in $t$ with coefficients in the integral tangent lattice $\Lambda_x$. To see this, note that $X' \rightarrow X$ is projective, with relative polarization $E$ represented by a Cartier divisor supported on the complement of $U$. In particular, there is a PL function on $\Sigma(X')$ such that for any log stable map $f: C \rightarrow X'$ with tropical type $\tau$, the log balancing condition ensures that $E\cdot f_*([C]) = \sum_{l \in L(G_{\tau})} E(u(l))$. Equipping $X'$ with the polarization given by $L = p^*(-K_X) + E$, we have $L\cdot f_*([C])$ is completely determined by $p_*f_*([C])$ and the contact orders of the marked points of $C$. Since only finitely many curve classes $\textbf{A} \in NE(X')$ satisfy $\textbf{A}\cdot L = k$ for any fixed $k$, for fixed monomial $x^{j}t^{\textbf{B}}$ for $\textbf{B} \in NE(X)$, there are only finitely many non-zero terms of $\vartheta_{D_i,x}$ which are sent to a term of the form $c_{i,j}x^jt^{\textbf{B}}$ with $c_{ij}$ non-zero under the replacement described above. Thus, the series 
$$\vartheta_{D_i,x}^X := \sum_{p^*(-K_X)\cdot \textbf{A} = k,\text{ }j \in \Lambda_x} |j|N_{i,j}^{x,\textbf{A}}x^jt^k$$
yields a power series in $\kk[\Lambda_x][[t]]$. Moreover, \cite[Theorem $1.1$]{qpint} ensures that the classical period of $W_{D,x}^X = \sum_i \vartheta_{D_i,x}^X$ recovers the regularized quantum period of $X$. 
  \end{proof}



\section{Double Ramification Cycle with target log variety}
In order to use Theorem \ref{mthm1} to study the punctured log Gromov-Witten theory of a toric bundle over a log smooth base scheme $S$, we will proceed by first studying a natural class of punctured log Gromov-Witten invariants generalizing the double ramification cycle with target introduced and studied in \cite{DRtarget}. 

Before introducing this class of invariant, we first consider the case $S$ is a smooth projective variety with a trivial log structure. Let $X =Tot(L)$ be the total space of a line bundle $L \in \Pic(S)$, equipped with the divisorial log structure associated with the zero section. Note $\Sigma(X) = \mathbb{R}_{\ge 0}$, and after picking a curve class $\textbf{A} \in H_2(S)$, define the decorated realizable tropical type $\pmb\tau$ of tropical map to $\Sigma(X)$ to have a single vertex $v$ with genus $g$ with decorating curve class $\textbf{A}$, and $n$ legs $l_1,\ldots,l_n$ with $\pmb\sigma(v) = \mathbb{R}_{> 0}$ and $u(l_i) = a_i$ satisfying $\sum_i a_i = c_1(L)\cdot \textbf{A}$. Denote the resulting moduli space of log stable maps by $\mathscr{M}(X,\pmb\tau)$. Note that all families of curves appearing in this space must map into the zero section $0 \in X$. In particular, it follows that $\mathscr{M}(X,\pmb\tau)$ is proper. Note there exists an open substack $\mathscr{M}(X,\pmb\tau)^\circ$ corresponding to maps from smooth $k$-marked genus $g$ curves to $X$ of type $\tau$. In fact, it can be shown that we have an identification:

\[\mathscr{M}(X,\pmb\tau)^\circ \cong DR^\circ_{g,\textbf{A},(a_i)}(S,L) = \{(f:C \rightarrow S)| f^*L(\sum_i -a_ip_i) \cong \mathcal{O}\}.\]
We strengthen the observation above to hold over compactifications in the following sense:

\begin{theorem}\label{punctorub}

Letting $\rho: \mathscr{M}(X,\pmb\tau) \rightarrow \mathscr{M}_{g,n}(S,\textbf{A})$ be the map forgetting log structures, we have the following identity in $A_*(\mathscr{M}_{g,n}(S,\textbf{A}))$:

\[\rho_*([\mathscr{M}(X,\pmb\tau)]^{\vir}) = [DR_{g,\textbf{A},(a_i)}(S,L)]\]

\end{theorem}

\begin{proof}


Letting $\proj_S(\mathcal{O}_S\oplus L) = \mathbb{P}(L)$, we first recall that the data of $\textbf{A}$ and $(a_i)$ yield a moduli space $\mathscr{M}^{\rel}_{g,\textbf{A},(a_i)}(\mathbb{P}(L))$ of stable maps relative to the zero and infinity section and its rubber variant $ \mathscr{M}^{\sim}_{g,\textbf{A},(a_i)}(\mathbb{P}(L)).$
We also recall the moduli space $\mathscr{M}^{\rel}_{g,\textbf{A},(a_i)}(\mathbb{P}(L))$ of log stable maps to $\mathbb{P}(L)$ equipped with the divisorial log structure at the $0$ and $\infty$ sections. These moduli stacks are equipped with virtual fundamental classes $[\mathscr{M}^{\rel}_{g,\textbf{A},(a_i)}(\mathbb{P}(L))]^{\vir}$ and $[\mathscr{M}^{\sim}_{g,\textbf{A},(a_i)}(\mathbb{P}(L))]^{\vir}$ respectively and come with a morphism 
$$\epsilon: \mathscr{M}_{g,\textbf{A},(a_i)}^{\rel}(\mathbb{P}(L)) \rightarrow \mathscr{M}^{\sim}_{g,\textbf{A},(a_i)}(\mathbb{P}(L)).$$
 There is a forgetful morphism $\forget: \mathscr{M}^{\sim}_{g,\textbf{A},(a_i)}(\mathbb{P}(L)) \rightarrow \mathscr{M}_{g,\textbf{A},n}(S)$, and
\[\forget_*([\mathscr{M}^{\sim}_{g,\textbf{A},(a_i)}(\mathbb{P}(L))]^{\vir}) = [DR_{g,\textbf{A},(a_i)}(S,L)].\]

By Lemma $2$ of \cite{topview}, when there exists a marked point $p_i$ such that $a_i = 0$, we have the following expression for $[\mathscr{M}^{\sim}_{g,\textbf{A},(a_i)}(\mathbb{P}(L))]^{\vir}$:

\begin{equation}\label{rigid}
[ \mathscr{M}^{\sim}_{g,\textbf{A},(a_i)}(\mathbb{P}(L))]^{\vir}  = \epsilon_*(\ev_{p_i}^*([D_0])\cap [\mathscr{M}^{\rel}_{g,\textbf{A},(a_i)}(\mathbb{P}(L))]^{\vir})
\end{equation}
The proof of the equality above in loc cit. is via localization.

We wish to express $[\mathscr{M}(X,\pmb\tau)]^{\vir}$ in a manner similar to Equation \ref{rigid}. To do so, we first replace $X$ with $X\times \mathbb{A}^1$, where $\mathbb{A}^1$ is equipped with its toric log structure, and the realizable tropical type $\tau$ with the tropical type $\tau'$ with underlying graph $G_\tau$, $\pmb\sigma(v)$ the deepest stratum of $\Sigma(X\times \mathbb{A}^1) \cong \mathbb{R}_{\ge 0}^2$ and $u(L_i) = (a_i,0)$. By Theorem $6.2$ of \cite{trglue}, we have $\mathscr{M}(X,\pmb\tau) \cong \mathscr{M}(X\times \mathbb{A}^1/\mathbb{A}^1,\pmb\tau')$, with the isomorphism respecting obstruction theories. We consider the log \'etale modification $\mathcal{P}:= \widetilde{X\times\mathbb{A}^1} \rightarrow X\times \mathbb{A}^1$ given by blowing up the zero section on the central fiber, and the tropical lift $\gamma'$ of $\tau'$ with $\pmb\sigma(v)$ given by the unique $1$-dimensional cone mapping into the interior of $\Sigma(X\times \mathbb{A}^1)$. As the component of the special fiber associated with this $1$-dimensional cone is $\mathbb{P}(L)$ with $\overline{\mathcal{M}}_{\mathbb{P}(L)}^{\gp}/\pi^*(\overline{\mathcal{M}}_{\mathbb{A}^1}^{\gp})$ supported only on the zero and infinity section of $\mathbb{P}(L)$, it follows again by Theorem $6.2$ of \cite{trglue} that $\mathscr{M}(\mathcal{P}/\mathbb{A}^1,\pmb\gamma') \cong \mathscr{M}_{g,\textbf{A},(a_i)}(\mathbb{P}(L))$ with the isomorphism respecting obstruction theories. Letting $\rho: \mathscr{M}_{g,\textbf{A},(a_i)}(\mathbb{P}(L)) \rightarrow \mathscr{M}(X,\pmb\tau)$ be the morphism induced by the blowdown and the identifications above, we now claim the following equality analogous to Equation \ref{rigid}:
\begin{equation}\label{localiz}
[\mathscr{M}(X,\pmb\tau)]^{\vir} = \rho_*(\ev_{p_i}^*([D_0])\cap [\mathscr{M}_{g,\textbf{A},(a_i)}(\mathbb{P}(L))]^{\vir})
\end{equation}

To see this equality, note that the class $[D_0]$ corresponds to a piecewise linear function on $\Sigma(X)$ which has slope $1$ upon restriction to the $1$ dimensional cone associated with the divisor $D_0$ and $0$ on all other $1$ dimensional cones. For dimension reasons, the only non-vanishing contribution to Equation \ref{maineq1} is given by $\tau$. Since $\ev_{v,\NN}: \tau_{\NN} \rightarrow \NN_{\ge 0}^2$ is an isomorphism, and the piecewise polynomial $[D_0]Q_{\gamma}$ is non-vanishing only in the interior of the cone bounded above by the diagonal, we have $\deg_{\tau}([D_0]Q_{\gamma}) = |\coker(\ev_{v,\NN})| = 1$, and the desired equality follows from Theorem \ref{mthm1}.  Furthermore, by \cite[Theorem $1.1$]{compGW}, there is a moduli stack $\operatorname{Kim}(X,\beta)$ equipped with a virtual class $[\operatorname{Kim}(X,\beta)]^{\vir}$, and a roof:

\[
\begin{tikzcd}
&&\arrow[dll,"\Theta"] \operatorname{Kim}_{g,\textbf{A},(a_i)}(\mathbb{P}(L)) \arrow[drr,"\Upsilon"]\\
\mathscr{M}^{\rel}_{g,\textbf{A},(a_i)}(\mathbb{P}(L))&&&&\mathscr{M}_{g,\textbf{A},(a_i)}(\mathbb{P}(L))
\end{tikzcd}
,\]
satisfying the equations:
 $$\Theta_*([\operatorname{Kim}_{g,\textbf{A},(a_i)}(\mathbb{P}(L))]^{\vir}) = [\mathscr{M}_{g,\textbf{A},(a_i)}(\mathbb{P}(L))^{\rel}]^{\vir}$$
$$\Upsilon_*([\operatorname{Kim}_{g,\textbf{A},(a_i)}(\mathbb{P}(L))]^{\vir} ) = [\mathscr{M}_{g,\textbf{A},(a_i)}(\mathbb{P}(L))]^{\vir}.$$ Thus, by the projection formula and Equations \ref{rigid} and \ref{localiz}, we have the equations: 

\begin{equation}\label{prigid}
\begin{split}
[\mathscr{M}^{\sim}_{g,\textbf{A},(a_i)}(\mathbb{P}(L))]^{\vir} &= \epsilon\Theta_*(\ev_{p_i}^*([D_0])\cap [\operatorname{Kim}_{g,\textbf{A},(a_i)}(\mathbb{P}(L))]^{\vir})\\
[\mathscr{M}(X,\pmb\tau)]^{\vir} &= \rho\Upsilon_*(\ev_{p_i}^*([D_0])\cap [\operatorname{Kim}_{g,\textbf{A},(a_i)}(\mathbb{P}(L))]^{\vir}).
\end{split}
\end{equation}

Hence, by pushing both equation in \ref{prigid} to $A_*(\mathscr{M}_{g,n}(S,\textbf{A}))$, we see the righthand sides of the two resulting equations are equal, and we derive the claimed equality assuming a non-relative marking. 

To remove the assumption of a point of contact order $0$, we note that we have a forgetful morphism $\st: \mathscr{M}_{g,\textbf{A},(0,(a_i))}((\mathbb{P}(L)) \rightarrow \mathscr{M}_{g,\textbf{A},(a_i)}((\mathbb{P}(L))$, and we let $x_0$ denote the marked section of the universal curve $\mathcal{C} \rightarrow \mathscr{M}_{g,\textbf{A},(0,(a_i))}((\mathbb{P}(L))$. By the dilaton equation, we have:

\[\st_*(\psi_{x_0}[\mathscr{M}_{g,\textbf{A},(0,(a_i))}((\mathbb{P}(L))]^{\vir}) = (2g-2+n)[\mathscr{M}_{g,\textbf{A},(a_i)}((\mathbb{P}(L))]^{\vir}.\]
Moreover, letting $\pmb\tau^{+}$ be the decorated realizable tropical type by adding one contracted leg to $\pmb\tau$, by an application of the known case of the proposition, we have:
\[\rho_*(\psi_{x_0}[\mathscr{M}(X,\pmb\tau^+)]^{\vir}) = st_*(\psi_{x_0}[\mathscr{M}^{\sim}_{g,\textbf{A},(a_i,0)}]^{\vir}) = (2g-2+n)[\mathscr{M}^{\sim}_{g,\textbf{A},(a_i))}]^{\vir}\]
Letting $\forget: \mathscr{M}(X,\pmb\tau^+) \rightarrow \mathscr{M}(X,\pmb\tau)$ and $\forget: \mathfrak{M}^{stab}(\mathcal{X},\pmb\tau^+)$ be the maps which forget the additional contact order zero point and stabilizes the resulting prestable map, the theorem now reduces to establishing the following analogue of the dilaton equation:
\[\forget_*(\psi_{x_0}[\mathscr{M}(X,\pmb\tau^+)]^{\vir}) = (2g-2+n)[\mathscr{M}(X,\pmb\tau)]^{\vir}\]
To prove this, observe that the forgetful map $\mathscr{M}(X,\pmb\tau) \rightarrow \mathfrak{M}(\mathcal{X},\pmb\tau)$ factors through the open inclusion $\mathfrak{M}^{stab}(\mathcal{X},\pmb\tau) \subset \mathfrak{M}(\mathcal{X},\pmb\tau)$ of prestable curves to $\mathcal{X}$ of decorated type $\tau$ with source curves in which all degree $0$ components together with their special points are stable. Similarly, we have $\mathscr{M}(X,\pmb\tau^+) \rightarrow \mathfrak{M}(\mathcal{X},\pmb\tau^+)$ factors through the analogous open substack $\mathfrak{M}^{stab}(\mathcal{X},\pmb\tau^+) \subset \mathfrak{M}(\mathcal{X},\pmb\tau^+)$. Moreover, forgetting $x_0$ and partial stabilization of the domain curve induces a morphism $\mathfrak{M}^{stab}(\mathcal{X},\pmb\tau^+) \rightarrow \mathfrak{M}_{A}^{stab}(\mathcal{X},\pmb\tau)$. Using this morphism, we have a cartesian diagram:

\[
\begin{tikzcd}
\mathscr{M}(X,\pmb\tau^+) \arrow{r}\arrow{d} & \mathfrak{M}^{stab}_{A}(\mathcal{X},\pmb\tau^+)\arrow{d}\\
\mathscr{M}(X,\pmb\tau) \arrow[r,"\epsilon_{\tau}"] & \mathfrak{M}^{stab}_{A}(\mathcal{X},\pmb\tau)
\end{tikzcd}.
\]

It is straightforward to verify the obstruction theories associated with the horizontal morphisms above are compatible, hence by \cite{vpull}, we have:

\[\forget_*(\psi_{x_0}[\mathscr{M}(X,\pmb\tau^+)]^{\vir}) = \epsilon_{\tau}^!(\forget_*(\psi_{x_0}[\mathfrak{M}^{stab}(\mathcal{X},\pmb\tau^+)])).\]
Thus, the dilaton equation reduces to showing:
\[\forget_*(\psi_{x_0}[\mathfrak{M}^{stab}(\mathcal{X},\pmb\tau^+)]) = (2g-2+n)[\mathfrak{M}^{stab}(\mathcal{X},\pmb\tau)].\]
To show this, we wish to show the claim $\forget: \mathfrak{M}^{stab}(\mathcal{X},\pmb\tau^+)\rightarrow \mathfrak{M}^{stab}(\mathcal{X},\pmb\tau)$ is the pullback along forgetful map $\mathfrak{M}^{stab}(\mathcal{X},\pmb\tau) \rightarrow \mathfrak{M}_{g,n,\textbf{A}}$ remembering only the domain marked curve of the partial stabilization map $\mathfrak{M}_{g,n+1,\textbf{A}} \rightarrow \mathfrak{M}_{g,n,\textbf{A}}$. Assuming this claim, then by \cite[Proposition $2.1.1$]{Cos}, the morphism $\mathfrak{M}^{stab}(\mathcal{X},\pmb\tau^+) \rightarrow \mathfrak{M}^{stab}(\mathcal{X},\pmb\tau)$ is the pullback of the universal family over $\mathfrak{M}_{g,n,\textbf{A}}$, hence the underlying stack morphism $\forget$ is the universal family of prestable domains. The desired equality would then follow by the standard proof of the dilaton equation for the moduli space of stable curves. 

To prove the claim, we first observe the analogous tropical statement, i.e. that the map of cones $\tau^{'+} \rightarrow \tau'$ for any tropical type $\tau'$ marked by $\tau$ is the universal family of tropical curves over $\tau'$. Indeed, to give a tropical map of type $\tau^{'+}$, it is necessary and sufficient to give a tropical map $\Gamma \rightarrow \Sigma(X)$ of type $\tau$, and a potentially new vertex $v'$ anywhere on $\Gamma$ which will contain the contracted leg associated with $x_0$. Since there is no constraint on the location of this new vertex, the tropical statement follows. 

Now note that since the structure map of a log curve is saturated, the fine and saturated and ordinary fiber products of $\mathfrak{M}(\mathcal{X},\pmb\tau)$ and $\mathfrak{M}_{g,n+1,\textbf{A}}$ over $\mathfrak{M}_{g,n,\textbf{A}}$ agree. Thus, it suffices to show that $\mathscr{M}(\mathcal{X},\pmb\tau^+)$ satisfies the required universal property for maps coming from fine and saturated log schemes. Thus, suppose we have a fine and saturated log scheme $S$, log prestable curves $C',C/S$ over $S$, a partial stabilization morphism $C' \rightarrow C$ after forgetting the marked section $x_0$, and a log map $C \rightarrow \mathcal{X}$ marked by the type $\pmb\tau$. By the tropical fiber product of the previous paragraph, the morphism $\Sigma(S) \rightarrow \Sigma(\mathfrak{M}_{g,n+1})$ corresponding to the family of tropical curves $\Sigma(C')\rightarrow \Sigma(S)$ factors through $\Sigma(S) \rightarrow \Sigma(\mathfrak{M}(\mathcal{X},\pmb\tau^+))$. In particular, the universal family admits a map $\Sigma(C') \rightarrow \Sigma(X)$ marked by the tropical type $\tau^+$. By the characterization of morphisms from fine and saturated log schemes to a Zariski Artin fans given in \cite[Proposition $2.10$]{decomp}, this uniquely determines a log map $C' \rightarrow \mathcal{X}$ marked by the tropical type $\tau^+$, as required.

\end{proof}

\begin{remark}
Xuanchun Lu has an independent proof of Theorem \ref{punctorub} via an explicit comparison of obstruction theories used to define the virtual classes.
\end{remark}

By Theorem \ref{punctorub} and the main result of \cite{DRtarget}, we observe that the image of a certain punctured log Gromov-Witten class in the Chow theory of the moduli space of stable maps is tautological:

\begin{corollary}
$\forget_*([\mathscr{M}(X,\pmb\tau)]^{\vir}) \in A_*(\mathscr{M}_{g,n}(S,\textbf{A}))$ has an expression in terms of $[\mathscr{M}_{g,n}(S,\textbf{A})]^{\vir}$ and tautological classes on $\mathscr{M}_{g,\textbf{A},n}(S)$. 
\end{corollary}

The double ramification cycle with target variety is built out of maps to a $\mathbb{P}^1$-bundle over a smooth base scheme $S$ relative to the zero and infinity sections. In order to prove the main theorem in the desired generality, we wish to generalize the definition of the DR cycle with target to allow for a \emph{log} smooth projective pair $(S,D)$ as a base scheme. Moreover, we also will need to show that this new class in $A_*(\mathscr{M}_{g,n}(S,\textbf{A}))$ is tautological. 

With these ends in mind, let $S$ be a log smooth scheme with Artin fan $\mathcal{S}$, $L \in Pic(S)$ a line bundle on $S$, and $\pi: X = Tot(L) \rightarrow S$ the total space of the line bundle over $S$. We give $X$ the log structure coming from the snc divisor $\pi^{-1}(D) + D_0$, with $D_0 \subset X$ the zero section. Note that the Artin fan for $X$ is simply $\mathcal{X} = \mathcal{S}\times \mathcal{A}_{\NN}$ and $\Sigma(X) = \Sigma(S) \times \mathbb{R}_{\ge 0}$. Let $\pmb\tau$ a decorated realizable tropical type of log stable map to $S$ with a single vertex $v'$, $L(G_{})= \{l_i'\}$ and total curve class $\textbf{A}$. Furthermore, for $n = |L(G_{\tau})|$, let $a_1,\ldots,a_n \in \ZZ$ satisfy:
\[L\cdot \textbf{A} = \sum_i a_i\]
Using this additional data, we define a realizable type $\gamma$ of tropical map to $\Sigma(X)$, which we will call a double ramification tropical type. The type has graph with a single vertex $v$, $\pmb\sigma(v) = \pmb\sigma(v') \times \mathbb{R}_{> 0} \subset \Sigma(S) \times \mathbb{R}_{\ge 0}$ and $n$ legs with contact order $u({l_i}) = (u({l_i'}),a_i) \in \pmb\sigma(l_i')^{\gp}_{\NN} \times \ZZ$. We have a corresponding moduli space of punctured log maps $\mathscr{M}_{\pmb\gamma}:=\mathscr{M}(X,\pmb\gamma)$, and note that all curves associated with points of this moduli space must scheme theoretically factor through $D_0 \cong X$. Additionally, we have a morphism $\epsilon: \mathscr{M}(X,\pmb\gamma) \rightarrow \mathscr{M}(S,\pmb\tau)$ as well as $\epsilon':\mathfrak{M}(\mathcal{X},\pmb\gamma) \rightarrow \mathfrak{M}(\mathcal{S},\pmb\tau)$ given in the latter case by composing with $\mathcal{X} = \mathcal{S}\times \mathcal{A}_{\NN} \rightarrow \mathcal{S}$ and taking the associated basic family of prestable maps to $\mathcal{S}$ marked by $\pmb\tau$. Note that in either case, no stabilization of the source curve is required. We define the DR cycle with target log variety as:

\begin{equation}
[DR_{g,(a_i)}(S,\pmb\tau,L)] = \epsilon_*([\mathscr{M}(X,\pmb\gamma)]^{\vir}) \in A_*(\mathscr{M}(S,\pmb\tau)).
\end{equation}

When $S$ has trivial log structure, Theorem \ref{punctorub} demonstrates that this class agrees with the ordinary double ramification cycle with target. To show that this class is tautological in general, we will modify arguments appearing in \cite{BNtaut} to this setting. First, we consider the following fiber product:
\[\begin{tikzcd}
\mathfrak{M}_{S,\pmb\gamma} \arrow[r,"j"]\arrow[d,"i"] & \mathfrak{M}_{\pmb\gamma}\arrow{d}\\
\mathscr{M}(S,\pmb\tau) \arrow{r} & \mathfrak{M}_{\pmb\tau}
\end{tikzcd}\]

We pullback the obstruction theory for the bottom horizontal morphism to equip $\mathfrak{M}_{S,\pmb\gamma} \rightarrow \mathfrak{M}_{\pmb\gamma}$ with a perfect obstruction theory. In particular, $\mathfrak{M}_{S,\pmb\gamma}$ is equipped with a virtual fundamental class $[\mathfrak{M}_{S,\pmb\gamma}]^{\vir} = j^![\mathfrak{M}_{\pmb\gamma}]$. Note that we have a canonical map $h: \mathscr{M}_{\pmb\gamma} \rightarrow \mathfrak{M}_{S,\pmb\gamma}$ which factors $\mathscr{M}_{\pmb\gamma} \rightarrow \mathfrak{M}_{\pmb\gamma}$ and $\mathscr{M}_{\pmb\gamma} \rightarrow \mathscr{M}(S,\pmb\tau)$. In order to use this square to describe the virtual class $[\mathscr{M}(S,\pmb\gamma)]^{\vir}$, a defect which will need to be addressed is the vertical morphisms are not proper. 

In order fix the defect above as well as to gain a more explicit understanding for the deformation-obstruction theory of the morphism $\mathscr{M}_{\pmb\gamma} \rightarrow \mathfrak{M}_{S,\pmb\gamma}$, we recall the notion of stability from \cite{logDR}:


\begin{definition}{(\cite{logDR} Definition $28$)}
Given a family $\pi: C\rightarrow B$ of log smooth curves, a stability condition of degree $d$ for $\pi$ is the data of a function $\theta: V(\Gamma_b) \rightarrow \mathbb{Q}$ for every geometric point $b \in B$, satisfying the conditions:

\begin{enumerate}
\item $\sum_{v \in V(\Gamma_b)} \theta(v) = d.$
\item For $b \rightarrow b'$ an \'etale specialization, hence a marking of $\Gamma_{b'}$ by the graph $\Gamma_b$, and every $v \in V(\Gamma_b)$, we have the equality:
\[\theta(v) = \sum_{v_i\text{ marked by }v} \theta(v_i).\]

\end{enumerate}

\end{definition}

A choice of stability condition $\theta$ for $\pi$ determines a modular compactification $Pic^0 \subset Pic^0_{\theta}$ of $\theta$-semistable line bundles on (quasi-stable models of) fibers of $\pi$, see \cite{logDR} Section $4$ for details. In particular, under the condition that the family admits a section $x_1: B \rightarrow C$ and the stability condition is \emph{non-degenerate} and \emph{small}, $\theta$-semistability agrees with $\theta$-stability and every multidegree $0$ line bundle $L$ on a fiber $C_s$ of $\pi$ is $\theta$-stable. 

Given a family of log curves $\pi$, a stability condition $\theta$ of degree $d$, and a line bundle $L \in Pic(C)$ of degree $d$, Theorem $35$ of \cite{logDR} constructs a log \'etale modification $\rho: B^\theta_L \rightarrow B$ such that $B^{\theta}_{L}$ represents the following fibered category over $LogSch/B$:

\begin{definition}
$B^{\theta}_{L}$ represents the fibered category over $LogSch/B$ with objects tuples:

\[(T/B,\widehat{C}\rightarrow C_T,\alpha),\]
where $T$ is a log scheme over $B$, $\widehat{C}\rightarrow C$ is a quasi stable model over the log curve $C_T = T\times_S C$, and $\alpha$ is a PL function on $\Sigma(\widehat{C})$ which vanishes at the vertex contained in the leg associated with the section $x_1$, and for which $L_T(\alpha)$ is $\theta$-stable. 

\end{definition}
By Remark $36$ of \cite{logDR}, $B^\theta_L$ depends only on the multidegree of the line bundle $L$, hence a tropical divisor on $\Sigma(C)$. In particular, we only need the initial input of the multidegree of the line bundle $L$ to define the stack $B^{\theta}_{L}$

The family of curves which we will consider stability conditions for has base $\mathfrak{M}_{\pmb\tau}$, where as usual we take the smallest open union of strata which contains the image of the forgetful map $\mathscr{M}(S,\pmb\tau) \rightarrow \mathfrak{M}_{\pmb\tau}$, and $C$ the universal curve over $\mathfrak{M}_{\pmb\tau}$. Note in particular that $\mathfrak{M}_{\pmb\tau}$ is finite type, hence there are finitely many combinatorial types of dual graphs of prestable curves over geometric points of $\mathfrak{M}_{\pmb\tau}$. Observe that we have a tropical divisor $deg(L)$ induced by the intersection pairing of $L \in Pic(S)$ with a curve class $\textbf{A} \in NE(X)$. Letting $\sum_i a_iv_i$ be the tropical divisor $\deg(\mathcal{O}(\sum_i a_ip_i))$ with $p_1,\ldots,p_n$ the marked points of the type $\tau$, after a making a choice of stability condition $\theta$, we may form the log \'etale modification $\rho: \mathfrak{M}^\theta_{deg(L) - \sum_i a_iv_i} \rightarrow \mathfrak{M}_{\pmb\tau}$. 

Note that for any prestable log map $C \rightarrow \mathcal{A}_X$ in the image of $\mathscr{M}(X,\pmb\gamma) \rightarrow \mathfrak{M}_{\pmb\gamma}$, letting $\alpha$ be pullback along $f_{\mathcal{A}_X}: C \rightarrow \mathcal{A}_X$ of the primitive integral PL function which is only non-vanishing along the ray associated zero section of $X$, we have by the log balancing condition that $\deg(L) - \deg(\alpha) = (0)$. Assuming the stability condition is \emph{non-degenerate} and \emph{small}, every multidegree $0$ line bundle $L$ on a fiber $C_s$ of $\pi$ is $\theta$-stable, hence in particular so is $\mathcal{O}_C \cong L\otimes\mathcal{O}(-\alpha)$. Finally, letting $h_{v_1}: \Sigma(\mathfrak{M}_{\pmb\gamma}) \rightarrow \mathbb{R}_{\ge 0}$ be the evaluation map at the vertex $v_1$ contained in the leg $l_1$, we have $deg(\alpha - h_{v_1}) = deg(\alpha)$ and $\alpha - h_{v_1}$ vanishes at $v_1$. Hence, by the universal property characterizing $\mathfrak{M}^\theta_{deg(L) - \sum_i a_iv_i}$, we have a factorization $\mathscr{M}_{\pmb\gamma} \rightarrow \mathfrak{M}^\theta_{deg(L) -\sum_i a_iv_i} \rightarrow \mathfrak{M}_{\pmb\tau}$. 

We now let $\mathfrak{M}_{\pmb\gamma}$ be the stack of punctured log maps marked by the decorated type $\pmb\gamma$ such that the partial stabilization $C \rightarrow \overline{C}$ induced by the marking is a quasi-stable model of $\overline{C}$, the map $C \rightarrow \mathcal{A}_S$ is induced by composing the partial stabilization with a morphism $\overline{C} \rightarrow \mathcal{A}_S$ corresponding to a point of $\mathfrak{M}_{\pmb\tau}$, and the associated tropical divisor $deg(\mathcal{O}(-\alpha)) + deg(L)$ is $\theta$-stable. In particular, $\mathfrak{M}_{\pmb\gamma}$ is a finite type open substack of the full stack of prestable log maps to $\mathcal{A}_X$ of decorated type $\pmb\gamma$, the forgetful map induces a map $\mathscr{M}_{\pmb\gamma}\rightarrow \mathfrak{M}_{\pmb\gamma}$, and we have a map $p: \mathfrak{M}_{\pmb\gamma} \rightarrow \mathfrak{M}^\theta_{deg(L) - \sum_i a_iv_i}$ given by mapping $C/T \rightarrow \mathcal{A}_X$ to the object associated to the triple of data $(T \rightarrow \mathfrak{M}_{\pmb\tau}, C \rightarrow \overline{C},h_{v_1}-\alpha)$. We claim that $\mathfrak{M}_{\pmb\gamma} \cong \mathfrak{M}^\theta_{deg(L) - \sum_i a_iv_i} \times B\mathbb{G}_m$, and $p$ is simply projection under this isomorphism. In particular, there is a section  $s: \mathfrak{M}^\theta_{deg(L) - \sum_i a_iv_i} \rightarrow \mathfrak{M}_{\pmb\gamma}$ which is smooth. We prove this in the following lemma and proposition:


\begin{lemma}\label{strict}
Let $\rho$ be the $1$-dimensional face of $\gamma$ spanned by $(0,1)$ in $\gamma \cong \tau \times \mathbb{R}_{\ge 0}$. Then after identifying $\rho$ with a cone of $\Sigma(\mathfrak{M}_{\pmb\gamma})$:

\[\Sigma(\mathfrak{M}^\theta_{deg(L) - \sum_i a_iv_i}) = \Sigma(\mathfrak{M}_{\gamma})_{\rho}\]

\end{lemma}

\begin{proof}
Recall that the cones of $\Sigma(\mathfrak{M}^\theta_{deg(L) - \sum_i a_iv_i})$ correspond to tuples $(\pmb\tau',\widehat{G} \rightarrow G_{\omega},D, I)$, where $G_{\tau'}$ is a prestable graph associated with a stratum of $\mathfrak{M}_{\pmb\tau}$, $\widehat{G}$ is the prestable graph of a quasistable model of $G$, $D$ is a $\theta$-stable tropical divisor on $\widehat{G}$ and $I$ is an acyclic flow on $\widehat{G}$ such that $\text{div}(I) = \deg(L) - \sum_i a_iv_i - D$. Moreover, the points of these cones correspond to tropical maps $\Gamma \rightarrow \Sigma(S)$ of type $\omega$ and a PL map $\alpha: \hat{\Gamma} \rightarrow \mathbb{R}$ inducing the flow $I$ with $\alpha(v_1) = 0$ for $v_1$ the vertex contained in the leg associated with the section $x_1$. Letting $\omega$ be one such cone and $q \in \omega_{\NN}$ an integral point, giving the data $(\hat{\Gamma}_q\rightarrow \Gamma_q, \alpha:\hat{\Gamma}_q \rightarrow \mathbb{R})$, we are given a map $h_q: \hat{\Gamma}_q \rightarrow \Sigma(S)$ marked by the tropical type $\pmb\tau$. By considering the PL function $\alpha + k$ for a sufficiently large integer $k$ such that $(\alpha+k)(v) > 0$ for all vertices $v \in V(G)$, we can construct a map $\hat{\Gamma}_q \rightarrow \Sigma(X) = \Sigma(S) \times \mathbb{R}_{\ge 0}$ of type $\pmb\gamma'$ marked by type $\pmb\gamma$, given by $h_q \times (\alpha + k)$. We thus have a point $p \in \pmb\gamma' \in \Sigma(\mathfrak{M}_{\pmb\gamma})$ such that $\Sigma(f)(p) = q$. As a result, for every cone $\omega \in \Sigma(\mathfrak{M}^\theta_{deg(L) - \sum_i a_iv_i})$, there exists a cone $\pmb\gamma' \in \Sigma(\mathfrak{M}_{\pmb\gamma})$ such that the induced map $\pmb\gamma' \rightarrow \pmb\omega$ is a surjection on integral points. Uniqueness of the type $\pmb\gamma'$ is also clear by construction. Moreover, we have $\gamma^{'gp}_{\NN} = \omega^{\gp}_{\NN}\times \ZZ$, with the map $\gamma^{'gp}_{\NN} \rightarrow \omega^{\gp}_{\NN}$ projection onto the first factor under this identification. 

Recall that $\Sigma(\mathfrak{M}_{\gamma})_{\rho}$ is the cone complex whose cones are the images of $\gamma' \subset \gamma^{'gp}$ under the quotient maps $\gamma^{'gp} \rightarrow \gamma^{'gp}/\rho^{\gp}$, with natural gluing morphisms. Note that $\rho$ is mapped to the zero cone by $\Sigma(p)$, hence we must have $\gamma^{'gp}/\rho^{\gp} \cong (\omega^{\gp}\times\ZZ)/\ZZ \cong \omega^{\gp}$. Since the image of $\gamma'$ in $\omega^{\gp}$ is $\omega$ as established in the previous paragraph, the conclusion follows.
\end{proof}

\begin{proposition}\label{smooth}
$\mathfrak{M}_{\pmb\gamma}\cong\mathfrak{M}^\theta_{deg(L) -\sum_i a_iv_i} \times B\mathbb{G}_m$, and in particular there is a smooth section $\mathfrak{M}^\theta_{deg(L) -\sum_i a_iv_i} \rightarrow \mathfrak{M}_{\pmb\gamma} $ of $p$ with relative cotangent bundle $\mathcal{O}$. 
\end{proposition}

\begin{proof}
%
%
%
%
%
%
%

We construct maps going both direction of the desired isomorphism. First, we have a map $p \times \mathcal{O}(h_{v_1}): \mathfrak{M}_{\pmb\gamma}\rightarrow \mathfrak{M}^\theta_{deg(L) -\sum_i a_iv_i} \times B\mathbb{G}_m$, where $h_{v_1}$ is the universal PL function giving the position of the vertex $v_1$ specified in the construction of $f$.

To construct the morphism in the opposite direction, suppose we have a morphism $S \rightarrow \mathfrak{M}^\theta_{deg(L) -\sum_i a_iv_i} \times B\mathbb{G}_m$. This gives the data of a log map $C \rightarrow \mathcal{A}_S$ marked by the type $\pmb\tau$ with associated tropical curve $\Gamma_C$, a PL morphism $\alpha$ on a quasi stable model $\widehat{\Gamma}_C \rightarrow \Gamma_C$ with $deg(\alpha) - deg(L) -\sum_i a_iv_i$ a $\theta$-stable tropical divisor, and a line bundle $\mathcal{O}(h_{v_1})$. Since we can define the desired section via compatible maps defined on \'etale charts, we may assume that the log structure on $S$ induced by $S \rightarrow \mathfrak{M}^\theta_{deg(L) -\sum_i a_iv_i}$ is induced by a map $S \rightarrow \mathcal{A}_{\omega}$ for some cone $\omega \in \Sigma( \mathfrak{M}^\theta_{deg(L) -\sum_i a_iv_i})$. We also produce a morphism $S \rightarrow \mathcal{A}_\omega \times B\mathbb{G}_m$ by taking the projection onto the second factor to $\mathcal{O}(h_{v_1})$. Since $\omega = \gamma'/\rho$ for a unique cone $\gamma' \in \Sigma(\mathfrak{M}_{\pmb\gamma})$ by Lemma \ref{strict}, and the closure of the stratum of $\mathcal{A}_{\mathfrak{M}_{\pmb\gamma}}$ associated with $\gamma'$ is isomorphic to $\mathcal{A}_{\Sigma(\mathfrak{M}_{\gamma})_{\rho}} \times B\mathbb{G}_m$ by Lemma \ref{stratadecomp}, this stratum of $\mathcal{A}_{\mathfrak{M}_{\pmb\gamma}}$ is isomorphic to $\mathcal{A}_{\omega}\times B\mathbb{G}_m$. Thus we have a map $S \rightarrow \mathcal{A}_{\gamma'}$ for some type $\gamma'$ marked by $\gamma$. This morphism yields a map $\Sigma(C) \rightarrow \Sigma(X)$ of type $\pmb\gamma'$, hence a map $C \rightarrow \mathcal{A}_X$ lifting the given map $C \rightarrow \mathcal{A}_S$, finally yielding the data of a map $S \rightarrow \mathfrak{M}_{\pmb\gamma}$. Since the constructed morphism depended only on the data of the initial morphism $S \rightarrow \mathfrak{M}^{\theta}_{deg(L) -\sum_i a_iv_i} \times B\mathbb{G}_m$, the local description of the morphisms glue to give a morphism of stacks $g: \mathfrak{M}^\theta_{deg(L) -\sum_i a_iv_i} \times B\mathbb{G}_m \rightarrow \mathfrak{M}_{\pmb\gamma}$ which is inverse to the morphism in the opposite direction. 

\end{proof}


Now consider the stack $\mathfrak{M}^\theta_{S,\pmb\gamma}$ constructed via the following fiber product:

\begin{equation}\label{stau}
\begin{tikzcd}
\mathfrak{M}_{S,\pmb\gamma} \arrow{r} \arrow{d} & \mathfrak{M}_{\pmb\gamma}\arrow{d}\\
\mathfrak{M}^\theta_{S,\pmb\gamma} \arrow[r,"j"]\arrow[d,"i"] & \mathfrak{M}^\theta_{deg(L) - \sum_i a_iv_i}\arrow{d}\\
\mathscr{M}(S,\pmb\tau) \arrow[r,"\fgt"] & \mathfrak{M}_{\pmb\tau}
\end{tikzcd}
\end{equation}
Using the obstruction theory for the morphism $\fgt$, we let 
$$[\mathfrak{M}_{S,\pmb\gamma}^\theta]^{\vir} = \fgt^![\mathfrak{M}^\theta_{deg(L) - \sum_i a_iv_i}].$$
The choice of stability condition $\theta$ also determines a modular compactification $Pic^0 \subset Pic^0_{\theta}$ of $\theta$-stable line bundles on (quasi-stable models of) fibers of the universal curve $\mathfrak{C} \rightarrow \mathfrak{M}_{\pmb\tau}$, and a section $aj: \mathfrak{M}^\theta_{S,\pmb\gamma} \rightarrow Pic^0_{\theta}$ sending a quadruple $(T/\mathfrak{M}_{\pmb\tau},\widehat{C}\rightarrow C_T,\alpha,f:C_T \rightarrow S)$ to $f^*L \otimes \mathcal{O}(-\alpha - \sum_i a_ip_i)$, see \cite{logDR} Section $4$ for details.

We now observe the following cartesian diagram:
\begin{lemma}\label{cartprod}
The following commuting square is cartesian
\begin{equation}\label{cartdiag}
\begin{tikzcd}
\mathscr{M}_{\pmb\gamma} \arrow{r} \arrow{d} & \mathfrak{M}_{S,\pmb\gamma}^{\theta}\arrow[d,"aj"]\\
\mathfrak{M}_{S,\pmb\gamma}^{\theta} \arrow[r,"e"]&Pic^0_{\theta}.
\end{tikzcd}.
\end{equation}
\end{lemma}
\begin{proof}
Suppose we have a scheme $T$ and two morphisms $c_1,c_2: T \rightarrow \mathfrak{M}^{\theta}_{S,\pmb\gamma}$ making the square commute. By composing with the projection $Pic^0_{\theta} \rightarrow \mathfrak{M}^{\theta}_{S,\pmb\gamma}$, we produce a family of log stable maps $C/T \rightarrow S$ marked by the type $\pmb\tau$, together with a map from a quasi-stable model $\widehat{C} \rightarrow \mathcal{A}_X$ marked by the type $\pmb\gamma$. Since both $e$ and $aj$ are sections of the projection, we must have each morphism $c_i$ is isomorphic to the morphism induced by projection, and we have $c_1$ and $c_2$ are isomorphic. Moreover, letting $f: \widehat{C}/T \rightarrow \mathcal{A}_X$ be the corresponding map, the fact that these sections $e$ and $aj$ are isomorphic after precomposing with $T \rightarrow \mathfrak{M}^{\theta}_{S,\pmb\gamma}$ implies that:

\[f^*L\otimes \mathcal{O}_{\widehat{C}}(-\alpha - \sum_i a_ip_i) \cong\mathcal{O}_{\widehat{C}}.\]
In particular, the quasi-stable model associated with $T \rightarrow \mathfrak{M}^{\theta}_{\deg(L) - \sum_{i}a_iv_i}$ must be trivial i.e. $\widehat{C} \cong C$. Since $f: C \rightarrow S \rightarrow X$ has image entirely within the zero section of $X$, the isomorphism of line bundles given above is necessary and sufficient to guarantee $C \rightarrow X$ admits a unique log enhancement such that the two induced maps $C\rightarrow \mathcal{A}_X$ and $C \rightarrow S$ are as given, and the conclusion follows.
\end{proof}

Pulling back the obstruction theory for the bottom horizontal morphism of Diagram \ref{cartdiag} equips $\mathscr{M}_{\pmb\gamma} \rightarrow \mathfrak{M}_{S,\pmb\gamma}^\theta$ with an obstruction theory. We note that $(R^1\pi_*\mathcal{O}[-1])^\vee \in D^{b}(\mathfrak{M}_{\pmb\gamma})$ is the obstruction bundle for the obstruction theory associated with the regular inclusion of the zero section $e: \mathfrak{M}_{\pmb\gamma} \rightarrow Pic^{0}_{\theta}$. We define the refined virtual class on $\mathscr{M}_{\pmb\gamma}$ as:

%

\[e^!([\mathfrak{M}_{S,\pmb\gamma}^\theta]^{\vir}) = [\mathscr{M}_{\pmb\gamma}]^{\rf} \in A_{\vdim}(\mathscr{M}_{\pmb\gamma}).\] 
We claim that the refined class produced above equals the virtual class on $\mathscr{M}_{\pmb\gamma}$ thought of as a space of punctured log stable maps:

\begin{proposition}\label{comp}
$[\mathscr{M}_{\pmb\gamma}]^{\rf} = [\mathscr{M}_{\pmb\gamma}]^{\vir}$
\end{proposition}

\begin{proof}
We first observe that we have a triple of morphisms:

\begin{equation}\label{triple1}
\begin{tikzcd}
\mathscr{M}_{\pmb\gamma} \arrow[r,"h"] \arrow[bend left,"\fgt"]{rr} &  \mathfrak{M}_{S,\pmb\gamma} \arrow[r,"g"] &  \mathfrak{M}_{\pmb\gamma},
\end{tikzcd}
\end{equation}
where $\fgt$ and $g$ are equipped with obstruction theories $\mathbb{E}_{\fgt} \rightarrow \mathbb{L}_{\mathscr{M}_{\pmb\gamma}/\mathfrak{M}_{\pmb\gamma}}$ and $\mathbb{E}_g \rightarrow \mathbb{L}_{\mathfrak{M}_{S,\pmb\gamma}/\mathfrak{M}_{\pmb\gamma}}$. Letting $\omega_{\pi}[1]$ be the relative dualizing complex for the family of curves $\pi: C\rightarrow \mathscr{M}_{\pmb\gamma}$, which is the relative dualizing sheaf with cohomological degree $-1$, and $f: C \rightarrow X$ is the universal map, observe that we have a map $f^*\Omega_S^{\log}\otimes \omega_{\pi}[1]\rightarrow f^*\Omega^{\log}_{X}\otimes \omega_{\pi}[1]$ associated to the projection. Taking derived pushforward gives a morphism $\phi: \mathbb{E}_g \rightarrow \mathbb{E}_{\fgt}$ which commutes with the morphisms to the cotangent complexes given above. By \cite[Construction $3.13$]{vpull}, the mapping cone of $\phi$ can be enhanced to a perfect obstruction theory for the morphism $h$ in Diagram \ref{triple1}. The mapping cone can be computed by taking the derived pushforward of the mapping cone of $f^*\Omega_S^{\log}\otimes \omega_{\pi}[1]\rightarrow f^*\Omega^{\log}_{X}\otimes (\omega_{\pi}[1])$, yielding $R\pi_*(\Omega_{X/S}^{\log}\otimes (\omega_{\pi}[1]))$. Since $X$ has is relative dimension $1$ over $S$, we have $f^*\Omega_{X/S}^{\log} = f^*(L^{\vee} \otimes \mathcal{O}(D_0)) = f^*L^{\vee} \otimes \mathcal{O}(\alpha).$ As a result, we have:
\[\mathbb{E}_h = R\pi_*(f^*\Omega_{X/S}^{\log}\otimes (\omega_{\pi}[1]))  = R\pi_*(f^*L^{\vee}\otimes \mathcal{O}( \alpha)\otimes (\omega_{\pi}[1])).\]
In addition, we consider the triple of morphisms:

\begin{equation}\label{triple2}
\begin{tikzcd}
\mathscr{M}_{\pmb\gamma} \arrow[r,"i"] \arrow[bend left,"h"]{rr} &  \mathfrak{M}_{S,\gamma}^\theta \arrow[r,"s"] &  \mathfrak{M}_{S,\pmb\gamma}
\end{tikzcd}
\end{equation}
By Proposition \ref{smooth}, $s$ is a smooth morphism, hence the relative cotangent bundle $\mathcal{O}$ defines a perfect obstruction theory for $s$. Moreover, since $L^{\vee} \otimes \mathcal{O}(\alpha) \cong \mathcal{O}$ on $\mathscr{M}_{\pmb\gamma}$, we have by relative duality that $R^0\pi_*L^{\vee}\otimes \mathcal{O}(\alpha) \otimes (\omega_{\pi}[1]) \cong R^0\pi_*\omega_{\pi}[1] = \mathcal{O}$. Additionally, we have the isomorphism $R^0\pi_*\omega_{\pi}[1] \cong \mathscr{H}om(\mathcal{O},R\pi_*\omega_{\pi}[1])$. A non-zero constant section of $\mathcal{O}$ thus induces a non-zero morphism $\mathcal{O} \rightarrow \mathbb{E}_h$. This morphism is uniquely determined by the composition to the truncation $\mathcal{O} \rightarrow \tau_{\ge 0} \mathbb{E}_h \cong \mathcal{O}$ being the scaling map, in particular an isomorphism. Since $h^0(\mathbb{E}_h) \rightarrow \mathcal{O}$ is an isomorphism, we have $\mathcal{O} \rightarrow h^0(\mathbb{E}_h)$ is an isomorphism. Construction $3.13$ of \cite{vpull} again yields an obstruction theory for $\mathscr{M}_{\pmb\gamma} \rightarrow \mathfrak{M}_{S,\gamma}^{\theta}$, with obstruction bundle $(R^1\pi_*L \otimes \mathcal{O}(-\alpha)[-1] )^{\vee}= (R^1\pi_*\mathcal{O}_{C}[-1])^{\vee}$. Moreover, $i^![\mathfrak{M}^{\theta}_{S,\pmb\gamma}]^{\vir} = h^![\mathfrak{M}_{S,\pmb\gamma}]^{\vir} = [\mathscr{M}_{\pmb\gamma}]^{\vir}$ by \cite[Corollary $4.9$]{vpull}.

%
%

Since the map $\mathscr{M}_{\pmb\gamma} \rightarrow \mathfrak{M}^\theta_{S,\pmb\gamma}$ is an embedding, by \cite[Theorem $4.6$]{globnc}, the virtual pullback depends only on the $K$-theory class of the complex $\mathbb{E}$.\footnote{Theorem $4.6$ is only stated as a formula for the virtual fundamental class, i.e. the virtual pullback of the fundamental class. However, Since $\mathfrak{M}_{S,\gamma}^{\theta}$ is a Deligne-Mumford stack, we can make a choice of cycle $V = \sum_i a_i V_i$ with $V_i \subset \mathfrak{M}_{S,\gamma}^{\theta}$ representing the class $[\mathfrak{M}_{S,\gamma}^{\theta}]^{\vir}$ and replacing the cone of the embedding $C_{\mathscr{M}_{\pmb\gamma}/\mathfrak{M}^{\theta}_{S,\gamma}}$ with the closed subcones $C_{\mathscr{M}_{\pmb\gamma}\times_{\mathfrak{M}^{\theta}_{S,\gamma}} V_i/\mathfrak{M}^{\theta}_{S,\gamma}}$, the same proof works to show the virtual pullback depends only on the $K$-theory class of the tangent-obstruction bundle. A complete argument in larger generality is given by Lu and Webb in \cite{relvpull}.} In particular, the operators $e^!,i^!: A_*(\mathfrak{M}_{S,\pmb\gamma}^{\theta}) \rightarrow A_*(\mathscr{M}_{\pmb\gamma})$ agree. Hence $e^![\mathfrak{M}^{\theta}_{s,\pmb\gamma}]^{\vir} = i^![\mathfrak{M}^{\theta}_{S,\pmb\gamma}]^{\vir} = [\mathscr{M}_{\pmb\gamma}]^{\vir}$, as desired.

\end{proof}

Note now that for any tropical type $\pmb\gamma'$ marked by $\pmb\gamma$, we have substacks $\mathscr{M}_{\pmb\gamma'}\subset \mathscr{M}_{\pmb\gamma}$ and $\mathfrak{M}_{\pmb\gamma'}\subset \mathfrak{M}_{\pmb\gamma}$ of log maps which admit markings by the type $\pmb\gamma'$, and we let $[\mathscr{M}_{\pmb\gamma'}]^{\vir} = \fgt^!([\mathfrak{M}_{\pmb\gamma'}])$. Letting $\mathscr{M}(S,\pmb\gamma')$ be the moduli stack of punctured log stable maps marked by $\gamma'$ and $i_{\pmb\gamma'}: \mathscr{M}(X,\pmb\gamma') \rightarrow \mathscr{M}_{\pmb\gamma}$ the natural map, note $i_{\pmb\gamma'*}([\mathscr{M}(S,\pmb\gamma')]^{\vir}) = |Aut(\gamma'/\gamma)|[\mathscr{M}_{\pmb\gamma'}]^{\vir}$. Moreover, by the same proof as Proposition \ref{comp}, we have $[\mathscr{M}_{\pmb\gamma'}]^{\vir} =e^!([\mathfrak{M}^{\theta}_{S,\gamma'}])$, with $\mathfrak{M}^{\theta}_{\gamma',S} = \mathfrak{M}_{\pmb\gamma'} \times_{\mathfrak{M}^{\theta}_{\deg(L)-\sum_i a_iv_i}} \mathfrak{M}^{\theta}_{S,\pmb\gamma}$. Additionally, for any type $\pmb\gamma'$ marked by $\pmb\gamma$, composing the universal tropical family $\Gamma_{\gamma} \rightarrow \Sigma(X)$ with the projection $\Sigma(X) \rightarrow \Sigma(S)$ yields a family of tropical maps of a fixed tropical type $\tau'$, which naturally admits a decoration $\pmb\tau'$ from the decoration of $\pmb\gamma$. The decorated tropical type $\pmb\tau'$ is marked by $\pmb\tau$, and we call this type the projection of $\pmb\gamma'$. Note that the maps $\mathscr{M}(X,\pmb\gamma') \rightarrow \mathscr{M}(S,\pmb\tau)$ factor through the map $\mathscr{M}(S,\pmb\tau') \rightarrow \mathscr{M}(S,\pmb\tau)$

By slight modifications to arguments presented in \cite{BNtaut}, we express the punctured log Gromov-Witten classes of $[\mathscr{M}_{\pmb\gamma}]^{\vir}$ in terms of punctured log Gromov-Witten classes of $[\mathscr{M}(S,\pmb\tau)]^{\vir}$. 

\begin{theorem}\label{DRtaut}
For $\pmb\gamma'$ any tropical type marked by the double ramification tropical type $\pmb\gamma$ which projects to a type $\pmb\tau'$ marked by $\pmb\tau$, and a piecewise polynomial $P \in PP(\mathfrak{M}_{\pmb\gamma})$, we have:
\[\fgt_*P\cap[\mathscr{M}_{\pmb\gamma'}]^{\vir} = \sum_{\pmb\tau' \subset \pmb\tau''} \beta_{\pmb\tau'}\cap [\mathscr{M}_{\pmb\tau''}]^{\vir}.\]
In the above equation, $\beta_{\tau}$ are polynomials in tautological classes on $\mathscr{M}(S,\pmb\tau)$. 
\end{theorem}

\begin{proof}
By \cite{smstab} Corollary $10$, letting $\mathcal{L}$ be the universal line bundle over the universal quasistable model $\pi: C^\theta_{Pic} \rightarrow Pic^0_{\theta}$, we have:
\[ e_*[\mathfrak{M}_{\pmb\gamma}^{\theta}] = c_g(-R\pi_*\mathcal{L})\cap [Pic^{0}_{\theta}].\]
Moreover, we have the cartesian diagram:

\[\begin{tikzcd}
C^\theta \arrow[r,"s"] \arrow[d,"\rho"] & C^{\theta}_{Pic}\arrow[d,"\pi"] \\
\mathfrak{M}_{S,\pmb\gamma}^{\theta} \arrow[r,"aj"] & Pic^0_{\theta}
\end{tikzcd}.\]
Since $\pi$ is both proper and flat, Lemma $6.2$ of \cite{BNtaut} yields:

\[aj^*(-R\pi_*\mathcal{L})  = -R\rho_*(f^*L\otimes \mathcal{O}(-\alpha)).\]
In particular, we have 
\[c_g(aj^*(-R\pi_*\mathcal{L}))  = c_g(-R\rho_*(f^*L\otimes \mathcal{O}(-\alpha))).\]
Since $[\mathscr{M}_{\pmb\gamma'}]^{\vir} = e^![\mathfrak{M}_{S,\pmb\gamma'}^\theta]^{\vir}$ by Proposition \ref{comp}, we have:
\[e_*([\mathscr{M}_{\pmb\gamma'}]^{\vir}) = [\mathfrak{M}_{S,\pmb\gamma'}^{\theta}]^{\vir}\cap c_g(aj^*(-R\pi_*\mathcal{L})) = [\mathfrak{M}_{S,\pmb\gamma'}^{\theta}]^{\vir}\cap c_g(-R\rho_*(f^*L\otimes \mathcal{O}(-\alpha))).\]

Note now that $f^*L\otimes \mathcal{O}(-\alpha)$ is a line bundle whose first Chern class is contained in the smallest subring of $A^*(\mathfrak{C}^\theta)$ containing piecewise polynomial classes pulled back from the Artin fan and $c_1(f^*L)$. Moreover, by the universal expression given in \cite[Lemma $5.1$]{drroot}, $c(R\pi_*f^*L\otimes \mathcal{O}(-\alpha))$ can be expressed in terms of $\pi_*(c_1(L\otimes \mathcal{O}(-\alpha))^2)$ and tautological classes pulled back from the moduli space of prestable curves. Note that:

\begin{equation}\label{pushexp}
\pi_*(c_1(L\otimes \mathcal{O}(-\alpha))^2) = \pi_*(c_1(L)^2) + 2\pi_*(c_1(L)c_1(\mathcal{O}(-\alpha)) + \pi_*(\mathcal{O}(-\alpha)^2).
\end{equation}

The first term is the pullback of a tautological class on $\mathscr{M}(S)$. Moreover, the third term is the pullback along $\mathfrak{M}^{\theta}_{S,\pmb\gamma'} \rightarrow \mathfrak{M}^\theta_{\deg(L) - \sum_i a_iv_i}$ of the analogous class. Since $\mathfrak{M}^\theta_{\deg(L) - \sum_i a_iv_i}$ is a log smooth algebraic stack, \cite[Theorem $67$]{logDR} shows this class is given by a Cartier divisor associated with a PL function on $\Sigma(\mathfrak{M}^\theta_{\deg(L) - \sum_i a_iv_i})$. For the second term on the righthand side of Equation \ref{pushexp}, note that on the universal curve $\mathfrak{C} \rightarrow \mathfrak{M}^\theta_{\deg(L) - \sum_i a_iv_i}$, $c_1(-\alpha)$ is represented by a Cartier divisor of the form $\sum_{\pmb\gamma'',v} a_{\pmb\gamma'',v}\mathfrak{C}_{\pmb\gamma'',v}$, with $\pmb\gamma''$ a decorated tropical type marked by $\pmb\gamma'$, $v \in V(G_{\gamma''})$ a vertex, and $\mathfrak{C}_{\pmb\gamma'',v}$ the component of the universal curve over $\mathfrak{M}^\theta_{\deg(L) - \sum_i a_iv_i,\pmb\gamma''}$ associated with $v$. By commutativity of virtual pullback with flat pullback, we have $\pi_*(2c_1(L)c_1(-\alpha)\cap\pi^*([\mathfrak{M}^{\theta}_{S,\pmb\gamma'})]^{\vir}) = \sum_{\pmb\gamma'',v \in V(G_{\gamma''})}2a_{\pmb\gamma'',v}(c_1(L)\cdot\textbf{A}(v))[\mathfrak{M}^{\theta}_{S,\pmb\gamma''}]^{\vir}$.



 It follows that $c_g(R\rho_*(f^*L\otimes \mathcal{O}(-\rho)))\cap [\mathfrak{M}^{\theta}_{S,\pmb\gamma'}]^{\vir}$ is a sum of virtual classes of strata capped with tautological classes pulled back along $\mathscr{M}_{\pmb\gamma'} \rightarrow \mathscr{M}_{\pmb\tau}$. That is, there exists an equation:

$$c_g(-R\rho_*(f^*L\otimes \mathcal{O}(-\alpha)))[\mathfrak{M}_{S,\pmb\gamma'}^\theta]^{\vir}= \sum_k\fgt^*(\beta_k)P_k[\mathfrak{M}^{\theta}_{S,\pmb\gamma_k'}]^{\vir}$$
with $\pmb\gamma'_k$ realizable tropical types marked by $\pmb\gamma'$ and $\beta_k\in R^*(\mathscr{M}_{\pmb\tau})$ a polynomial in tautological classes and $P_k \in A^*_{\text{op}}(\mathfrak{M}^{\theta}_{S,\pmb\gamma'_k})$ piecewise polynomial classes. Since $i: \mathfrak{M}^\theta_{deg(L)-\sum_i a_iv_i} \rightarrow \mathfrak{M}_{\pmb\tau}$ is induced by pulling back a subdivision of Artin fans with $\mathfrak{M}_{\pmb\tau}$ smooth over a stratum, we derive the following expression by Corollary \ref{pushvpull} and cartesian diagram \ref{stau}:
 

\[
\begin{split}
\fgt_*(P\cap[\mathscr{M}_{\pmb\gamma'}]^{\vir}) &= i_*((PP_k\sum_ki^*(\beta_k))\cap [\mathfrak{M}^{\theta}_{S,\pmb\gamma'_k}]^{\vir}) \\
&=  \sum_k\sum_{\pmb\tau_k \subset \pmb\tau_k'}\beta_k\deg_{\tau_k'}(PP_kQ_{\gamma_k'})\cap [\mathscr{M}_{\pmb\tau_k'}]^{\vir}
\end{split}
.\]

Since all operational Chow classes acting on $A_*(\mathscr{M}_{\pmb\tau_k'})$ coming from piecewise polynomials can be expressed in terms of tautological classes on $\mathscr{M}_{\pmb\tau_k'}$, the conclusion follows from the expression above. 
\end{proof}


\section{Higher-rank DR with log target and log Gromov-Witten theory of toric bundles}

\subsection{Higher-rank DR with logarithmic target}
Using the expression for the virtual class in Theorem \ref{DRtaut}, we can deduce the analogous result for the higher rank by iterated application of Theorem \ref{DRtaut}. For setup, let $\pmb\tau$ be a tropical type of punctured log map considered in the previous section, and $L_1,\ldots,L_{k}$ be a collection of line bundles on $S$, and for $1\le i\le k$ let $T_i = \oplus_{j=1}^i Spec_{S}(L_j)$ be the total space of the corresponding split vector bundle. The vector bundle $T_i$ comes equipped with a morphism $T_i \rightarrow \mathcal{A}_{\NN^i}$ corresponding to the sum of the Cartier divisors given by the pullback of the zero sections along the projections $\pi_j: T_i \rightarrow L_j$. We equip $T_i$ with the pulled back log structure along the morphism $T_i \rightarrow \mathcal{A}_{\NN^i} \times \mathcal{A}_S$, hence $\Sigma(T_i) = \mathbb{R}^k_{\ge 0} \times \Sigma(S)$. Now let $\sigma_k = \mathbb{R}^k_{\ge 0} \times \{0\}$, $\gamma_{k}$ be the tropical type with a single vertex $v$, $\pmb\sigma(v) = \sigma_k$, and legs with contact orders $c_1,\ldots,c_m$, with $L(c_i)\ge 0$ for all piecewise linear function $L$ pulled back from $\Sigma(S)$. After decorating $v$ with an effective curve class $\textbf{A}(v) = \textbf{A} \in H_2(S)$ satisfying $\textbf{A}\cdot L_i = \sum_j L_i(c_j)$ for all $L_i$, we have an associated moduli stack of punctured log maps $\mathscr{M}(T_{k},\pmb\gamma_k)$, which we will also denote by $\mathscr{M}(S,\pmb\gamma_k)$ owing to the fact that all maps marked by the type $\pmb\gamma_k$ factor through the zero section of $T_k \rightarrow S$. Since $\gamma_k$ is a realizable type, by \cite[Proposition $3.30$]{punc}, we have $\mathscr{M}(S,\pmb\gamma_k)$ is virtually equidimensional and reduced. We call this type a \emph{toric contact cycle type}. Moreover, projection induces maps $\mathscr{M}(S,\pmb\gamma_j) \rightarrow \mathscr{M}(S,\pmb\gamma_i)$ for $i < j$. The forgetful map induces a map $\forget: \mathscr{M}(S,\pmb\gamma_k) \rightarrow \mathscr{M}(S,\pmb\tau)$, and we consider the pushforward:
\begin{equation}\label{drlogt}
[DR_{(L_i,a_i),\textbf{A}}(S,\pmb\tau)]:=\forget_*[\mathscr{M}(S,\pmb\gamma_k)]^{\vir} \in A_*(\mathscr{M}(S,\pmb\tau)),
\end{equation}
which we dub the \emph{toric contact cycle with target log variety}. Additionally, for any type $\pmb\gamma$ marked by $\pmb\gamma_i$, we define the projection $\pmb\tau_{\gamma}$ to be the composite of the rank $1$ projections of $\pmb\gamma$. We claim that the virtual classes associated with types $\pmb\gamma$ considered above yield tautological classes in $\mathscr{M}(S,\pmb\tau_{\gamma})$:

\begin{theorem}\label{hrankt}
For any tropical type $\pmb\gamma$ marked by the toric contact type $\pmb\gamma_k$ projecting to a tropical type $\pmb\tau_{\gamma}$ of decorated punctured tropical map to $\Sigma(S)$, and $P$ a piecewise polynomial function on $\Sigma(\mathscr{M}(S,\pmb\gamma))$, there exist tautological classes $\beta_{\pmb\tau'}$ for all decorated types $\pmb\tau'$ marked by $\pmb\tau$ such that:

\[\forget_*P\cap[\mathscr{M}(S,\pmb\gamma)]^{\vir} = \sum_{\pmb\tau_{\gamma}\subset \pmb\tau'} \beta_{\pmb\tau'}\cap[\mathscr{M}_{\pmb\tau'}]^{\vir} \in A_*(\mathscr{M}(S,\pmb\tau_{\gamma})).\] 
\end{theorem}

\begin{proof}
We prove the proposition by induction on the rank $k$, with the case $i=0$ trivial. For the induction step, note the projection map $T_{i} \rightarrow T_{i-1}$ is a line bundle given by the pullback of $L_i$ along the projection map $T_{i-1} \rightarrow S$, and composing with the projection map gives a morphism of moduli stacks $\forget^{i-1}: \mathscr{M}(S,\pmb\gamma_i) \rightarrow \mathscr{M}(S,\pmb\gamma_{i-1})$. By Theorem \ref{DRtaut}, for any tropical type $\pmb\gamma$ marked by $\pmb\gamma_i$ with $\pmb\gamma$ a tropical lift of a type $\pmb\omega$ of log map to $T_{i-1}$ marked by $\pmb\gamma_{i-1}$, there exist tautological classes $\beta_{\pmb\omega'} \in A^*(\mathscr{M}(S,\pmb\omega'))$ for decorated tropical types $\pmb\omega'$ marked by $\pmb\omega$ pulled back from $\mathscr{M}(X,\pmb\tau)$ such that:
$$\forget^{i-1}_*(P\cap [\mathscr{M}(S,\pmb\gamma)]^{\vir}) =\sum_{\pmb\omega\subset \pmb\omega'} \beta_{\pmb\omega'}\cap[\mathscr{M}_{\pmb\omega'}]^{\vir}.$$ By the projection formula and the induction hypothesis, letting $\pmb\tau_{\omega'}$ be the tropical type given by projecting $\pmb\omega'$, there exists tautological classes $\beta_{\pmb\omega',\pmb\tau'} \in A^*(\mathscr{M}(S,\pmb\tau_{\omega'}))$ such that:


\begin{equation}
\begin{split}
\forget_*(P\cap[\mathscr{M}(S,\pmb\gamma)]^{\vir}) &= \sum_{\pmb\omega \subset \pmb\omega'}\sum_{\pmb\tau_{\omega'}\subset \pmb\tau'} \beta_{\pmb\omega'}\beta_{\pmb\omega',\pmb\tau'} \cap[\mathscr{M}_{\pmb\tau'}]^{\vir}\\
&= \sum_{\pmb\tau_{\gamma}\subset \pmb\tau'} (\sum_{\pmb\omega\subset \pmb\omega'}\beta_{\pmb\omega'}\beta_{\pmb\omega',\pmb\tau'})\cap[\mathscr{M}_{\pmb\tau'}]^{\vir}
\end{split}.
\end{equation}
\end{proof}

\subsection{Reconstructing log GW of toric bundles}

We finally shift our focus to the study of the punctured log Gromov-Witten theory of a split toric bundle $X \rightarrow S$, for $S$ a log smooth projective variety. More precisely, with $L_1,\ldots,L_k \in Pic(S)$ line bundles and $T = Tot(\oplus_i L_i)$ the corresponding split vector bundle over $S$ as in the previous section, we let $S^\dagger$ be the stratum of $T$ given by the zero section and $ X = T \times \mathbb{A}^1$. Note in particular that we have $\Sigma(S^\dagger) = \Sigma(S) \times \mathbb{R}_{\ge 0}^{k}$ and $\Sigma(X) = \Sigma(S) \times \mathbb{R}_{\ge 0}^{k+1}$. In particular, letting $\Spec \kk^{\dagger}$ be the standard log point, we have the log morphism $S^{\dagger}\times \kk^{\dagger} \rightarrow X$ given by the inclusion of the zero section. 

Let $\Sigma$ be a fan in $\mathbb{R}^k$ supporting a strictly convex piecewise linear function, and $X_\Sigma$ the corresponding projective toric variety. The line bundles $L_1,\ldots,L_k$ together with the fan $\Sigma$ determine a split toric bundle $X_{S,\Sigma}$ with fiber $X_{\Sigma}$. The following lemma show how the map $X_{S,\Sigma} \rightarrow S$ arises via a log \'etale modificaiton:


\begin{lemma}\label{tbund}
After embedding $i:\mathbb{R}^k \rightarrow \mathbb{R}^{k+1}$ via $i((x_i)) = (1,1+x_1,\ldots,1+x_k)$, consider the polyhedral complex $\mathcal{P} = i(\Sigma)\cap\mathbb{R}_{\ge 0}^{k+1}$, and $\widetilde{\mathbb{R}_{\ge 0}^{k+1}} \rightarrow \mathbb{R}_{\ge 0}^{k+1}$ the cone over $\mathcal{P}$, consider the associated log modification $\widetilde{X} \rightarrow X$ induced by the subdivision $\widetilde{\Sigma(X)} = \Sigma(S) \times \widetilde{\mathbb{R}_{\ge 0}^{k+1}}$. Then the underlying morphism of schemes of the induced log modification $S_{\Sigma} \rightarrow S^\dagger \times \kk^{\dagger}$ is the projective toric bundle $X_{S,\Sigma} \rightarrow S$.
\end{lemma}

\begin{proof}
Letting $\omega_v$ be the unique $1$-dimensional cone of $\widetilde{\Sigma(S^{\dagger})}$ mapping to the interior of the unique maximal cone of $\Sigma(S^{\dagger})$ under $\widetilde{\Sigma(S^{\dagger})}  \rightarrow \Sigma(S^{\dagger})$, observe that $\widetilde{\Sigma(S^{\dagger})} _{\omega_v} \cong \Sigma$. The lemma now follows by Theorem $B$ of \cite{torbund}.
\end{proof}

Note that the map $\widetilde{X} \rightarrow X$ is log \'etale, hence for any tropical type $\pmb\tau$ of punctured log map to $X$ and $\pmb\gamma$ any tropical lift of $\pmb\tau$, Lemma \ref{tbund} together with Theorem \ref{mthm1} allows us to express the Chow classes $st_*(\alpha\cap[\mathscr{M}(\widetilde{X},\pmb\gamma)]^{\vir})$ in terms of log Gromov-Witten classes in $A_*(\mathscr{M}(X,\pmb\tau))$. Given a punctured tropical type of map $\overline{\pmb\gamma}$ to $\Sigma$ decorated by curve classes in the toric bundle $X_{S,\Sigma}$, we pick a representative $\Gamma \rightarrow \Sigma$ such that for all vertices $v \in V(G_{\overline{\pmb\gamma}})$, we have $h(v) + (1,1,\ldots,1) \in \mathbb{R}^k_{\ge 0}$. It follows that $h(e) + (1,1,\ldots,1) \in \mathbb{R}^k_{\ge 0}$ for all compact edges $e \in E(G_{\overline{\pmb\gamma}})$. By cutting unbounded legs of $G_{\overline{\pmb\gamma}}$ whose image under $h + (1,1,\ldots,1)$ is not contained in $\mathbb{R}^k_{\ge 0}$, we produce a decorated punctured tropical curve $\Gamma^{\circ} \rightarrow \Sigma(\widetilde{X})$ with associated decorated tropical type $\pmb\gamma$ of punctured log map to $\widetilde{X}$. By combining Theorems \ref{mthm1} and \ref{mthm3}, we will prove Theorem \ref{mthm2}.

\begin{proof}[Proof of Theorem \ref{mthm2}]
By Corollary \ref{puncmcr1} and the projection formula, the desired result for the toric variety bundle associated with the fan $\Sigma$ holds if it holds for the toric variety bundle associated with a subdivision $\widetilde{\Sigma}$. We may therefore assume that $\Sigma$ is a smooth fan. By \cite[Theorem $6.2$]{trglue}, we have $\mathscr{M}(X_{\Sigma,S},\overline{\pmb\gamma}) \cong \mathscr{M}(\widetilde{X}/\Spec\kk^{\dagger},\pmb\gamma)$ with isomorphism respecting obstruction theories. Hence, it suffices to prove the result for the punctured log Gromov-Witten classes produced from $\mathscr{M}(\widetilde{X}/\Spec\kk^{\dagger},\pmb\gamma)$. Since $\widetilde{X} \rightarrow X$ is log \'etale, it follows that Corollary \ref{mcr1} allows us to express the pushforward of any punctured log Gromov-Witten class for $\widetilde{X}$ in terms of punctured log Gromov-Witten classes produced from $\mathscr{M}(X/\Spec\kk^{\dagger},\pmb\tau)$. Another application of \cite[Theorem $6.2$]{trglue} applied to the trivial family $ X = T\times \mathbb{A}^1$ gives an identification $\mathscr{M}(X/\Spec\kk^{\dagger},\pmb\tau)\cong\mathscr{M}(T,\pmb\tau) = \mathscr{M}(S^{\dagger},\pmb\tau)$ which respects obstruction theories. By Theorem \ref{hrankt}, punctured log Gromov-Witten classes produced from the latter moduli stack all pushforward under $\mathscr{M}(S^{\dagger},\pmb\tau') \rightarrow \mathscr{M}(S,\beta)$ to a sum over strata classes of $\mathscr{M}(S,\beta)$ capped with insertions pulled back from $S$ and tautological classes, as required for the first part of Theorem \ref{mthm2}. By this expression, if we additionally assumed that the punctured log Gromov-Witten classes of $S$ push forward to the tautological ring under the stabilization map, it follows that so to do the punctured log Gromov-Witten classes of $\widetilde{X}$ hence $X_{\Sigma,S}$, as required for the second part of Theorem \ref{mthm2}.

\end{proof}

\nocite{*}
\bibliographystyle{amsalpha}
\bibliography{gluing}

\end{document}